%% file: main.tex
\documentclass[letterpaper]{article}

\usepackage{mathtools}
\usepackage{amsmath,amssymb,amsthm}
\usepackage{graphicx}
\usepackage{float}
\usepackage{tikz}
\usetikzlibrary{arrows.meta,calc,decorations.markings}

\usepackage{cite}
\usepackage[hyphens]{url}

\usepackage{xcolor}
\usepackage{hyperref}

\definecolor{GoogleBlueText}{HTML}{1A73E8}
\definecolor{GoogleGreenText}{HTML}{137333}
\definecolor{GoogleRedText}{HTML}{C5221F}
\definecolor{GoogleAmberText}{HTML}{B06000}

\hypersetup{
    pdftitle={Coloring Monomial Cayley Graphs: The One-Parameter Case},
    pdfauthor={Jonathan Davidson},
    pdfsubject={Algebraic methods for coloring unit-distance graphs related to the Hadwiger-Nelson problem},
    pdfkeywords={Hadwiger-Nelson problem, chromatic number of the plane, unit-distance graphs, monomial Cayley graphs, circular chromatic number, algebraic graph theory},
    colorlinks=true,
    linkcolor=GoogleRedText,
    citecolor=GoogleRedText,
    urlcolor=GoogleBlueText,
    filecolor=GoogleAmberText
}

\newtheorem{theorem}{Theorem}[section]
\newtheorem{proposition}[theorem]{Proposition}
\newtheorem{lemma}[theorem]{Lemma}
\newtheorem{corollary}[theorem]{Corollary}
\newtheorem{conjecture}[theorem]{Conjecture}

\theoremstyle{definition}
\newtheorem{definition}[theorem]{Definition}

\theoremstyle{remark}
\newtheorem*{remark}{Remark}

\newcommand{\C}{\mathbb{C}}
\newcommand{\R}{\mathbb{R}}
\newcommand{\Q}{\mathbb{Q}}
\newcommand{\Z}{\mathbb{Z}}
\newcommand{\N}{\mathbb{N}}
\newcommand{\T}{\mathbb{T}}

\newcommand{\Cay}{\operatorname{Cay}}
\newcommand{\Mon}{\operatorname{Mon}}
\newcommand{\MonCay}{\operatorname{MonCay}}

\DeclarePairedDelimiter{\set}{\{}{\}}

\title{Coloring Monomial Cayley Graphs: \\ The One-Parameter Case}
\author{Jonathan Davidson\thanks{
Email: \texttt{jonathan.j.davidson@gmail.com}.
Website: \protect\url{https://jjdavidson.github.io}.}}
\date{October 2026}

\begin{document}

\maketitle

\begin{abstract}
	How do algebraic relations among unit directions influence the chromatic number of the plane?
	We introduce monomial Cayley graphs as a framework for investigating this question, beginning with graphs whose allowed unit steps are signed powers of a single complex parameter.
	Transcendental parameters give bipartite graphs, while for roots of unity we determine both the chromatic and circular chromatic numbers.
	As a geometric consequence, every graph admitting a unit-distance embedding in the plane in which the angle between any two edge directions is a rational multiple of $\pi$ is three-colorable.
	For algebraic noninteger parameters, four colors suffice, and three suffice when the primitive integer minimal polynomial has leading coefficient at least three.
	We conjecture that three colors suffice for every algebraic noninteger direction parameter and that four suffice for every direction parameter.
\end{abstract}

\input{1-introduction.tex}

\input{2-background.tex}

\input{3-cyclotomic.tex}

\input{4-noncyclotomic.tex}

\input{5-conclusion.tex}

\section*{Disclosure of AI Assistance}

AI tools assisted with drafting portions of the text, preparing TikZ figures, and developing code to verify the sextic witness.
They also assisted with constructing and checking the monomial Cayley graph embeddings of the de Grey and Parts graphs and with exploratory work on higher-order cyclotomic coloring tests.
ChatGPT 6 was used in attempts to resolve Conjectures~\ref{conj:four-color} and~\ref{conj:noninteger-three-color}, but did not resolve either conjecture.
The author takes responsibility for the mathematical arguments, computational claims, figures, references, and final text.

\bibliographystyle{plain}
\begingroup
\raggedright
\bibliography{references}
\endgroup

\clearpage
\appendix
\input{6-appendix.tex}

\end{document}

%% file: 1-introduction.tex

\section{Introduction}

The Hadwiger--Nelson problem asks how many colors are needed to color every point of the plane so that points at distance one receive different colors.
This number is the \emph{chromatic number of the plane}, denoted $\chi(\R^2)$.
A \emph{unit-distance graph} is a graph whose vertices can be placed at distinct points of the plane so that the endpoints of every edge are one unit apart.
Every coloring of the plane restricts to a coloring of such a graph, so a finite unit-distance graph can provide a lower bound for $\chi(\R^2)$.
In 2018, de Grey constructed a five-chromatic unit-distance graph~\cite{deGrey2018}, raising the lower bound for the chromatic number of the plane to five.
In a preprint publicly released on October 6, 2026, OpenAI reports a proof that the plane admits no proper five-coloring, without any regularity assumption on the color classes~\cite{OpenAI2026Plane}.
Together with the seven-coloring obtained from a regular hexagonal tiling, illustrated in Figure~\ref{fig:hex-coloring}, this gives
\[
6\leq\chi(\R^2)\leq7.
\]
Soifer gives an extensive account of the earlier history and development of this problem~\cite{Soifer2024}.

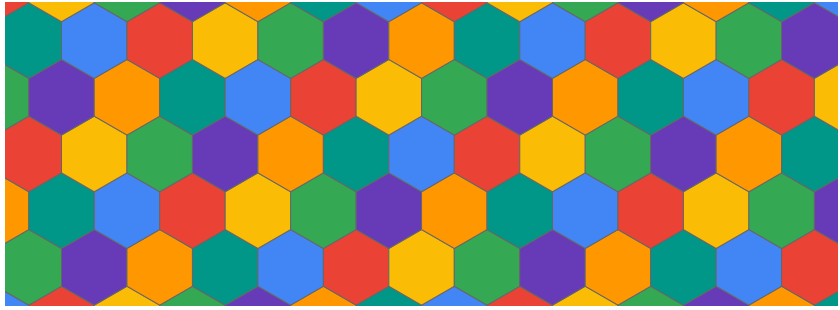
\begin{figure}[ht]
    \centering
    \input{tikz-code/hex-coloring.tex}
    \caption{A seven-coloring of the plane using a hexagonal tiling. The tiles can be scaled so that each has diameter less than one and distinct tiles of the same color are more than one unit apart.}
    \label{fig:hex-coloring}
\end{figure}

De Grey's construction prompted the Polymath16 project, whose participants sought smaller five-chromatic witnesses, more transparent explanations of their coloring obstructions, and constructions requiring still more colors~\cite{Polymath16,Mixon2018PolymathThird}.
One strand of the discussion studied the algebraic structure of the coordinates.
Identifying the plane with $\C$, participants considered rings generated by a few complex numbers of modulus one.
They asked which other elements of these rings had modulus one and how the relations among these directions could be used to construct colorings~\cite{Mixon2018PolymathThird}.
The Moser spindle~\cite{MoserMoser1961} provides a small example in which both the geometry and the algebra can be seen explicitly.

Following Speyer's description of the spindle in Polymath16~\cite{Speyer2018Spindle}, let $\alpha=(1+i\sqrt3)/2$.
The points $0$, $1$, $\alpha$, and $1+\alpha$ form a diamond consisting of two equilateral triangles sharing the edge from $1$ to $\alpha$.
Now let $\beta=(\sqrt{33}+i\sqrt3)/6$.
Both parameters have modulus one, and multiplying the diamond's vertices by $\beta^2$ gives a rotated copy about the origin.
The two diamonds share the origin, and their outer tips, $1+\alpha$ and $(1+\alpha)\beta^2$, are also one unit apart.
Adding the edge between these tips gives the Moser spindle shown in Figure~\ref{fig:moser-spindle}.

\input{tikz-code/moser.tex}

In any three-coloring of the first diamond, the adjacent vertices $1$ and $\alpha$ use two colors.
Both $0$ and $1+\alpha$ are adjacent to these vertices, so they must receive the third color.
The same argument in the rotated diamond forces $(1+\alpha)\beta^2$ to have the color of $0$.
The edge joining the two outer tips makes these requirements incompatible.
A fourth color suffices: three-color the two diamonds with the same color at the origin, then give one outer tip a fourth color.

The coordinates also explain the edge directions.
Since $\alpha-1=\alpha^2$, every edge vector in the first diamond is one of $\pm1$, $\pm\alpha$, or $\pm\alpha^2$.
The edge vectors in the second diamond are obtained by multiplying these by $\beta^2$.
For the connecting edge, we have
\[
(1+\alpha)\beta^2-(1+\alpha)
=(1+\alpha)(\beta^2-1)
=\alpha^2\beta.
\]
Every edge vector is therefore a signed monomial in $\alpha$ and $\beta$, while every vertex is an integer linear combination of these monomials.
This suggests an infinite graph containing the spindle: take the ring $\Z[\alpha,\beta]$ as the vertex set and join two vertices when their difference belongs to
\[
\Mon(\alpha,\beta)=\set{\pm\alpha^a\beta^b:a,b\in\Z_{\geq0}}.
\]
This is a Cayley graph of the additive group of the ring.
More generally, for an additive group $G$ and a set $S=-S$ of nonzero elements, $\Cay(G,S)$ has vertex set $G$ and an edge between $u$ and $v$ whenever $v-u\in S$.
We apply this construction to any collection of unit complex parameters.

\begin{definition}
    Let $\mathcal A$ be a set of complex numbers of modulus one.
    Write $\Mon(\mathcal A)$ for the set of signed finite products of elements of $\mathcal A$, including $\pm1$ from the empty product.
    The \emph{monomial Cayley graph generated by $\mathcal A$} is
    \[
    \MonCay(\mathcal A)=\Cay\bigl(\Z[\mathcal A],\Mon(\mathcal A)\bigr).
    \]
    For a finite parameter set $\mathcal A=\set{\alpha_1,\ldots,\alpha_r}$, we also write $\MonCay(\alpha_1,\ldots,\alpha_r)$.
\end{definition}

The parameters determine both the vertex ring and the allowed steps.
Every vertex is reached from the origin by a finite sequence of these steps, since every element of $\Z[\mathcal A]$ is an integer linear combination of monomials.
Every edge has length one.
There can, however, be pairs of vertices at distance one whose difference is not a signed monomial.
Writing $R=\Z[\mathcal A]$, we have an inclusion of graphs on the same vertex set:
\[
\MonCay(\mathcal A)
\subseteq
\Cay\bigl(R,\set{z\in R:|z|=1}\bigr).
\]
The graph on the right contains all unit-distance edges between points of the ring.
The graph on the left specifies its steps through their multiplicative generators.

This choice makes the allowed directions explicit without requiring a classification of every unit-modulus element of the ring.
Polynomial relations among the parameters describe the additive relations among the permitted steps, giving algebraic conditions that can be used to construct colorings.
The restriction to monomial steps still allows the program to recover the chromatic number of the plane.
To see this, take all algebraic unit directions as parameters:
\[
U_{\mathrm{alg}}=\set{\alpha\in\C:\alpha\text{ is algebraic over }\Q,\ |\alpha|=1}.
\]
Let $\Gamma_{\mathrm{alg}}=\MonCay(U_{\mathrm{alg}})$.
The set $U_{\mathrm{alg}}$ is closed under multiplication and negation and contains $1$, so
$\Mon(U_{\mathrm{alg}})=U_{\mathrm{alg}}$.
Every algebraic unit direction is therefore an allowed step in this countable graph.

Every finite unit-distance graph has an embedding with algebraic coordinates.
The squared-distance formula turns the unit-distance constraints into a finite system of quadratic polynomial equations with rational coefficients in the vertex coordinates.
Strict polynomial inequalities ensure that distinct vertices remain distinct.
Both $\R$ and the field of real algebraic numbers are real closed, so the Tarski--Seidenberg transfer principle~\cite[Theorem~2.80]{BasuPollackRoy2006} ensures that a real solution gives a solution with real algebraic coordinates.
For a connected graph, translate one vertex of this algebraic embedding to the origin.
A path from that vertex to any other expresses the latter as a sum of algebraic unit directions.
The translated graph is consequently a subgraph of $\Gamma_{\mathrm{alg}}$.
Since $\Gamma_{\mathrm{alg}}$ is itself a unit-distance graph, the de Bruijn--Erd\H{o}s compactness theorem~\cite{deBruijnErdos1951} gives
\[
\chi(\Gamma_{\mathrm{alg}})=\chi(\R^2).
\]

Each finite connected witness uses only finitely many algebraic edge directions, which may themselves be chosen as its parameters.
Thus monomial Cayley graphs with finite parameter sets suffice to recover the chromatic number of the plane.
Combining this observation with the lower bound of OpenAI~\cite{OpenAI2026Plane} gives a finite set $\mathcal A\subset U_{\mathrm{alg}}$ for which
\[
\chi(\MonCay(\mathcal A))\geq6.
\]
This consequence is existential: the argument does not identify such a parameter set or bound its size.
The appendix gives explicit embeddings of the five-chromatic graphs of de Grey and Parts using small collections of algebraic parameters~\cite{deGrey2018,Parts2020}.
The wider aim is to understand how relations among a few direction parameters govern colorability, to construct explicit witnesses requiring at least six colors, and to investigate whether seven colors can be forced.

The spindle supplies a four-chromatic finite subgraph using two direction parameters.
We begin the systematic study of this program with one parameter: how many colors can be forced by the powers of a single unit complex number?
Let $\alpha=e^{i\theta}$ now be an arbitrary unit complex number, and write $\Gamma_\alpha=\MonCay(\alpha)$.
Its allowed steps are $\pm\alpha^n$ for nonnegative integers $n$.
Walks from the origin along these steps reach exactly the vertices of $\Z[\alpha]$, each of which has the form
\[
p(\alpha)=a_m\alpha^m+\cdots+a_1\alpha+a_0
\]
with integer coefficients $a_j$.
If $\theta/\pi$ is irrational, the allowed directions are dense in the unit circle.
Yet two colors suffice when $\alpha$ is transcendental: every vertex has a unique polynomial expression, and each allowed step changes the parity of the sum of its coefficients.
For algebraic $\alpha$, polynomial relations identify different expressions for the same vertex.
Understanding how these relations constrain a coloring is the one-parameter problem studied in this paper.
Powers of algebraic integers on the unit circle also occur in Radchenko's construction of finitely generated additive groups whose full unit-distance graphs have infinite vertex degrees~\cite{Radchenko2021}.

Our main results describe this algebraic case.
When $\theta$ is a rational multiple of $\pi$, the parameter $\alpha$ is a root of unity, and we determine $\chi(\Gamma_\alpha)$ exactly.
It is two when the order of $\alpha$ is a power of two, and three otherwise.
We also determine the circular chromatic number, a refinement of the ordinary chromatic number based on separation of colors around a circle~\cite{Vince1988,Zhu2001}.
As a geometric consequence, every unit-distance graph whose edge directions form pairwise angles that are rational multiples of $\pi$ is three-colorable.
For algebraic parameters that are not algebraic integers, the theorem of Garc\'ia-Marco, Knauer, and Menara~\cite[Theorem~4.3]{GarciaMarcoKnauerMenara2026} gives a four-color bound; we give further arithmetic criteria guaranteeing two or three colors.
We conjecture that three colors suffice for every parameter in this latter class.

For algebraic integers, four colors can be necessary even with a single direction parameter.
We construct such an example, answering the question in Cervantes and Krebs~\cite[Example~2.14]{CervantesKrebs2023}.
This leads to our central conjecture: for every unit complex parameter $\alpha$,
\[
\chi(\Gamma_\alpha)\leq4.
\]

Our coloring methods bring together the matrix viewpoint of Cervantes and Krebs~\cite[Examples~2.13--2.14]{CervantesKrebs2023} and the multiplier method described in Liu's survey of distance graphs~\cite[Section~2]{Liu2008}.
In a cyclic quotient $\Z/p\Z$, the multiplier method sends a residue $x$ to $\lambda x$ modulo $p$, choosing $\lambda$ so that the allowed step differences move away from zero around the cycle.
For example, multiplication by $3$ modulo $7$ sends the steps $\pm1$ to $\pm3$.
Placing the residues at equally spaced points of a circle of circumference one then separates adjacent vertices of the seven-cycle by circular distance $3/7$.
Dividing the circle into three half-open arcs of equal length gives a proper three-coloring.

In general, the normalized multiplier takes the form $x\mapsto\lambda x/p+\Z$.
This is an additive map into $\R/\Z$, also called an \emph{additive character}.
We use this formulation of the multiplier method on $\Z[\alpha]$: a character whose values on all allowed steps stay sufficiently far from zero gives a coloring by partitioning the circle into short arcs, as well as a bound on the circular chromatic number.
For algebraic $\alpha$, the minimal polynomial imposes a linear recurrence on the character values at successive powers of $\alpha$.
Polynomial evaluations and congruences provide ways to construct solutions with the required separation.
For algebraic noninteger parameters, we present the character-extension construction underlying the semiminimal Cayley graph bound cited above.

These upper bounds also guide the search for finite subgraphs requiring more colors, following the screening strategy proposed by Gibbs in Polymath16~\cite{Gibbs2018Screening}.
A successful coloring test gives an upper bound for the entire infinite graph; parameter sets left unresolved can then be examined through finite subgraphs.
A satisfiability solver tests whether a prescribed number of colors suffices, and an obstruction can be reduced by deleting vertices while preserving its noncolorability, as in Heule's work on five-chromatic unit-distance graphs~\cite{Heule2018}.
The resulting witness can then be checked independently.

The four-chromatic example announced above brings these two parts of the method together.
Let $\alpha$ be a unit-modulus root of the sextic polynomial
\[f(x)=x^6-x^4-x^3-x^2+1.\]
An explicit periodic character gives a four-coloring of the entire graph $\Gamma_\alpha$.
Starting from a graph on the $830$ distinct sums of selected powers $1,\alpha,\ldots,\alpha^9$, our computational search produces a subgraph with $24$ vertices and $51$ edges that admits no proper three-coloring.
Its edges are verified by exact polynomial reduction, and its non-three-colorability is checked by exhaustive backtracking.
These upper and lower bounds establish $\chi(\Gamma_\alpha)=4$.
Figure~\ref{fig:sextic-coloring} displays the finite witness inside this larger sumset graph.

\begin{figure}[H]
    \centering
    \includegraphics[width=\linewidth]{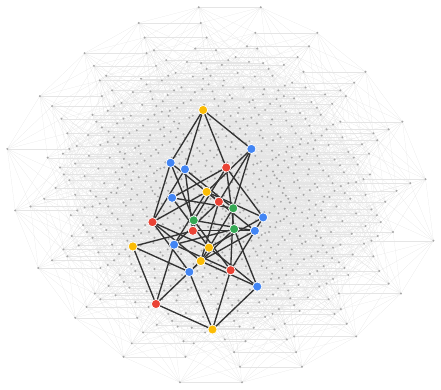}
    \caption{The four-chromatic sextic witness with $24$ vertices and $51$ edges, highlighted inside the $830$-vertex sumset graph used in the search. Colored vertices and dark edges mark the witness; the search graph is shown in gray.}
    \label{fig:sextic-coloring}
\end{figure}

The paper is organized as follows.
\begin{itemize}
    \item Section~2 develops the algebraic framework and character constructions, including the Cyclotomic Congruence Criterion and Cyclotomic Norm Criterion and the explicit eighth-root coloring result.
    \item Section~3 determines the ordinary and circular chromatic numbers for transcendental and cyclotomic parameters and proves the three-color theorem for rational-angle directions.
    \item Section~4 treats noncyclotomic algebraic parameters. It establishes the bounds for algebraic nonintegers, constructs the four-chromatic sextic example, and proves the four-color bound through degree eight.
    \item Section~5 concludes with the remaining one-parameter conjectures and outlines the research program for several parameters and higher dimensions.
    \item Appendix~\ref{app:salem-obstruction} proves the limitation of fixed cyclotomic tests. Appendix~\ref{app:de-grey} presents monomial Cayley embeddings of the five-chromatic constructions of de Grey and Parts.
\end{itemize}

%% file: tikz-code/hex-coloring.tex
\begin{tikzpicture}[line join=round]
    \definecolor{hexcolor0}{HTML}{4285F4}
    \definecolor{hexcolor1}{HTML}{EA4335}
    \definecolor{hexcolor2}{HTML}{FBBC05}
    \definecolor{hexcolor3}{HTML}{34A853}
    \definecolor{hexcolor4}{HTML}{673AB7}
    \definecolor{hexcolor5}{HTML}{FF9800}
    \definecolor{hexcolor6}{HTML}{009688}

    \def\hexradius{0.5}
    \def\hexhalfwidth{5.5}
    \def\hexhalfheight{2}

    \begin{scope}
        \clip (-\hexhalfwidth,-\hexhalfheight)
            rectangle (\hexhalfwidth,\hexhalfheight);
        \foreach \r in {-5,...,5} {
            \foreach \q in {-10,...,10} {
                \pgfmathtruncatemacro{\hexindex}{mod(\q+3*\r+140,7)}
                \pgfmathsetmacro{\hexx}{sqrt(3)*\hexradius*(\q+\r/2)}
                \pgfmathsetmacro{\hexy}{1.5*\hexradius*\r}
                \begin{scope}[shift={(\hexx,\hexy)}]
                    \path[fill=hexcolor\hexindex,
                          draw=black!60, line width=0.3pt]
                        (30:\hexradius)
                        -- (90:\hexradius)
                        -- (150:\hexradius)
                        -- (210:\hexradius)
                        -- (270:\hexradius)
                        -- (330:\hexradius)
                        -- cycle;
                \end{scope}
            }
        }
    \end{scope}

    \pgfresetboundingbox
    \path[use as bounding box]
        (-\hexhalfwidth,-\hexhalfheight)
        rectangle (\hexhalfwidth,\hexhalfheight);
\end{tikzpicture}

%% file: tikz-code/moser.tex
\begin{figure}[ht]
    \centering
    \begin{tikzpicture}[scale=2.2,line cap=round,line join=round]
        \definecolor{spindleblue}{HTML}{4285F4}
        \definecolor{spindlered}{HTML}{EA4335}
        \definecolor{spindlegreen}{HTML}{34A853}
        \tikzset{
            spindle vertex/.style={circle,draw=black,inner sep=0pt,minimum size=4.5pt},
            spindle forced/.style={spindle vertex,fill=spindleblue,minimum size=5.5pt}
        }
        \coordinate (u) at (0,0);
        \coordinate (x) at (1,0);
        \coordinate (y) at (0.5,{sqrt(3)/2});
        \coordinate (v) at (1.5,{sqrt(3)/2});
        \coordinate (xp) at ({5/6},{sqrt(11)/6});
        \coordinate (yp) at ({(5-sqrt(33))/12},{(sqrt(11)+5*sqrt(3))/12});
        \coordinate (vp) at ({(15-sqrt(33))/12},{(3*sqrt(11)+5*sqrt(3))/12});

        \draw[thick] (u)--(x)--(v)--(y)--cycle;
        \draw[thick] (x)--(y);
        \draw[thick] (u)--(xp)--(vp)--(yp)--cycle;
        \draw[thick] (xp)--(yp);

        \draw[spindlered,very thick,dashed,-{Stealth[length=2mm]},
              shorten <=4pt,shorten >=4pt]
            (v)--node[midway,above,sloped,fill=white,inner sep=1.5pt,
                      text=black,font=\small] {$\alpha^2\beta$} (vp);

        \node[spindle forced,label=below left:{$0$}] at (u) {};
        \node[spindle vertex,fill=spindlered,label=below right:{$1$}] at (x) {};
        \node[spindle vertex,fill=spindlegreen,label=above left:{$\alpha$}] at (y) {};
        \node[spindle forced,label=right:{$1+\alpha$}] at (v) {};
        \node[spindle vertex,fill=spindlered,label=below right:{$\beta^2$}] at (xp) {};
        \node[spindle vertex,fill=spindlegreen,label=left:{$\alpha\beta^2$}] at (yp) {};
        \node[spindle forced,label=above:{$(1+\alpha)\beta^2$}] at (vp) {};
        \foreach \vertex in {u,v,vp} {
            \node[circle,draw=black,inner sep=0pt,minimum size=9pt] at (\vertex) {};
        }
    \end{tikzpicture}
    \caption{The Moser spindle. The colors show a three-coloring of the two diamonds before the dashed edge is added. The three ringed vertices must have the same color in every three-coloring, so the additional edge prevents a three-coloring of the spindle. Its arrow records the difference $\alpha^2\beta$ used in the algebraic description.}
    \label{fig:moser-spindle}
\end{figure}

%% file: 2-background.tex

\section{Colorings from Polynomial Relations}

A polynomial relation among the powers of $\alpha$ describes a closed walk in $\Gamma_\alpha$.
To construct a coloring from images of the allowed steps, those images must respect every such relation.
We begin with a worked example in which one relation yields both a coloring and a matching obstruction.
The Shift Criterion and the multiplier method then provide a common framework for the constructions that follow.
We use Eisenstein quotients to obtain the Cyclotomic Congruence Criterion, and then turn to trace polynomials and norms when the integer congruence tests give no bound.

\subsection{Polynomial Evaluation and Circular Colorings}
\label{subsec:three-coloring}

Consider the parameter $\alpha=(5+i\sqrt{39})/8$ from Cervantes and Krebs~\cite[Example~2.13]{CervantesKrebs2023}.
Its primitive integer minimal polynomial is
\[f(x)=4x^2-5x+4.\]
Every vertex of $\Gamma_\alpha$ can be written as $p(\alpha)$ for some $p\in\Z[x]$, but different polynomials may represent the same vertex.
A coloring defined from a polynomial expression must therefore give the same answer for every expression representing that vertex.
The coefficient sum of the relation is $f(1)=3$, suggesting the color assignment
\[c\bigl(p(\alpha)\bigr)=p(1)\pmod3.\]
To show that this assignment is well-defined, suppose $p(\alpha)=q(\alpha)$.
The minimal polynomial property and Gauss's lemma show that $p-q$ is an integer polynomial multiple of $f$.
Evaluating at $1$ therefore gives $p(1)\equiv q(1)\pmod3$.
Thus the color is independent of the polynomial representation of the vertex.
Evaluation at $1$ respects addition, so $c:\Z[\alpha]\to\Z/3\Z$ is additive.
An allowed step $\pm\alpha^n$ changes the color by $\pm1$ modulo three, so adjacent vertices receive different colors.
Equivalently, $c$ gives a graph homomorphism
\[\Gamma_\alpha\to C_3=\Cay(\Z/3\Z,\{\pm1\}),\]
and hence $\chi(\Gamma_\alpha)\leq3$.
Figure~\ref{fig:ck-three-color-pullback} shows two paths whose polynomial expressions differ by $f$, so they reach the same vertex and assign it the same color.

\begin{figure}[ht]
	\centering
	\input{tikz-code/ck-three-color.tex}
	\caption{Two paths representing the same vertex and color, with $p-q=f$ and $p(1)\equiv q(1)\equiv1\pmod3$.}
	\label{fig:ck-three-color-pullback}
\end{figure}
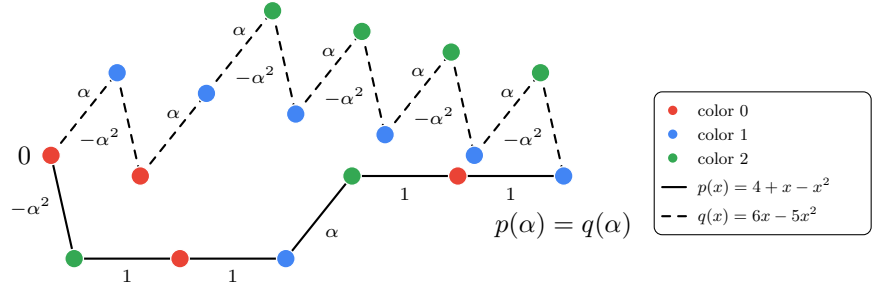

Evaluation at $-1$ gives a second coloring construction for the same parameter.
Since $f(-1)=13$, evaluation at $-1$ defines an additive map $\eta:\Z[\alpha]\to\Z/13\Z$ by $\eta(p(\alpha))=p(-1)\pmod{13}$.
The same divisibility argument shows that this map is well-defined.
Every allowed step $\pm\alpha^n$ maps to $\pm(-1)^n$, so $\eta$ gives a graph homomorphism
\[\Gamma_\alpha\to C_{13}=\Cay(\Z/13\Z,\{\pm1\}).\]
The usefulness of this target comes from the circular arrangement of its colors.

For $r\geq2$, a \emph{circular $r$-coloring} assigns vertices to points on a circle of circumference $r$ so that adjacent vertices have circular distance at least one.
The \emph{circular chromatic number} $\chi_c$ is the infimum of the values of $r$ for which such a coloring exists.
Equivalently, we may use the normalized circle $\T=\R/\Z$ and require circular distance at least $1/r$ across every edge.
An additive homomorphism into $\T$ is called a \emph{character}.
We regard the vertex groups as discrete, so these characters need not be continuous in the Euclidean topology inherited from the plane.
The steps in the cyclic target have images $\pm1$.
Multiplication by $6$ modulo $13$ moves them to $\pm6$, the residues at greatest circular distance from zero.
This is the multiplier method from the introduction, applied to the quotient supplied by $f(-1)$.
After normalization, it gives a character $\psi:\Z[\alpha]\to\T$ defined by
\[\psi\bigl(p(\alpha)\bigr)=\frac{6}{13}p(-1)\pmod1.\]
For every allowed step,
\[\psi(\pm\alpha^n)=\pm\frac{6}{13}(-1)^n\pmod1,\]
whose circular distance from zero is $6/13$.
Since $\psi$ is additive, adjacent vertices have this same circular separation.
Rescaling to a circle of circumference $13/6$ therefore gives
\[\chi_c(\Gamma_\alpha)\leq\frac{13}{6}.\]
The three half-open intervals $[0,6/13)$, $[6/13,12/13)$ and $[12/13,1)$ also give a proper three-coloring, since two points in the same interval have circular distance less than $6/13$.

\begin{figure}[ht]
	\centering
	\input{tikz-code/ck-circular-color.tex}
	\caption{The closed walk $4-5\alpha+4\alpha^2=0$ and its character image.
	Each step advances by $6/13$ modulo $1$; the color intervals give a three-coloring.}
	\label{fig:ck-circular-coloring}
\end{figure}
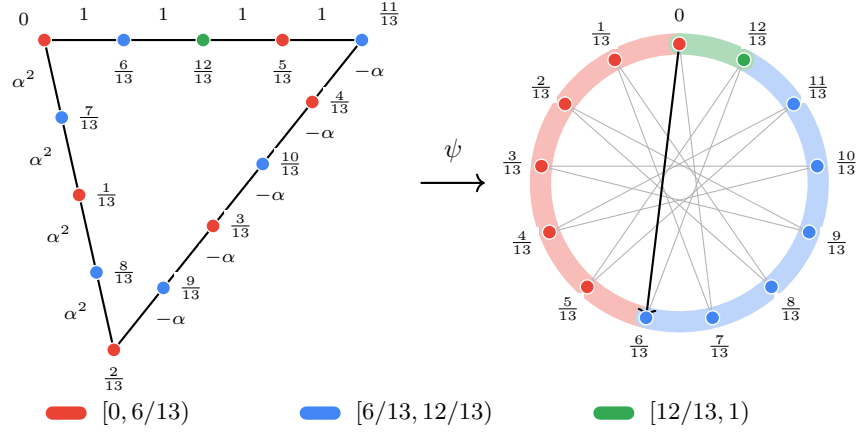

The polynomial relation supplies a matching lower bound.
The identity $4-5\alpha+4\alpha^2=0$ gives a closed walk consisting of four steps in direction $1$, five in direction $-\alpha$, and four in direction $\alpha^2$.
Each of these steps has $\eta$-image $1$, so the first thirteen vertices of the walk have distinct images $0,1,\ldots,12$ modulo thirteen.
They therefore form a copy of $C_{13}$ in $\Gamma_\alpha$.
Each step has $\psi$-image $6/13$, and Figure~\ref{fig:ck-circular-coloring} shows the cycle, its character image, and the three color intervals.

An independent set in $C_{13}$ has at most six vertices.
In a normalized circular $r$-coloring of this cycle, the preimage of every half-open arc of length $1/r$ is independent and hence contains at most six vertices.
Averaging over the position of the arc gives $13/r\leq6$, so $\chi_c(C_{13})\geq13/6$.
Restricting a coloring of $\Gamma_\alpha$ to this cycle gives the same lower bound for $\Gamma_\alpha$.
Together with the constructed character, this proves the exact circular bound
\[\chi_c(\Gamma_\alpha)=\frac{13}{6}.\]
The three-coloring and the odd cycle also give $\chi(\Gamma_\alpha)=3$.

The two constructions prescribe images of the powers of $\alpha$ and use the minimal polynomial to check that these prescriptions define additive maps.
For the rest of this section, let $\alpha$ be an arbitrary algebraic direction parameter with primitive integer minimal polynomial $f(x)=\sum_{j=0}^{d}a_jx^j$, normalized so that $a_d>0$.
As in the example, the minimal polynomial property and Gauss's lemma identify $\Z[\alpha]$ with $\Z[x]/(f)$.
The shifted coefficient relations in Cervantes and Krebs~\cite[Example~2.14]{CervantesKrebs2023} give the following criterion for prescribing the images of the powers of $\alpha$.

\begin{proposition}[Shift Criterion]
	\label{prop:shift-criterion}
	Let $G$ be an abelian group, and let $g_n\in G$ for every $n\geq0$.
	The assignment $\alpha^n\mapsto g_n$ extends to an additive homomorphism $h:\Z[\alpha]\to G$ if and only if
	\[\sum_{j=0}^{d}a_jg_{n+j}=0\]
	for every $n\geq0$.
	When it exists, this extension is unique.
\end{proposition}

\begin{proof}
	The prescribed images define an additive map $L:\Z[x]\to G$ where
	\[L\left(\sum_n b_nx^n\right)=\sum_n b_ng_n.\]
	This map descends to $\Z[x]/(f)\cong\Z[\alpha]$ precisely when it vanishes on $(f)$.
	Since the polynomials $x^nf$ for $n\geq0$ generate $(f)$ additively, this is equivalent to
	\[L(x^nf)=\sum_{j=0}^{d}a_jg_{n+j}=0\]
	for every $n\geq0$.
	The powers of $\alpha$ generate $\Z[\alpha]$ additively, so their prescribed images determine the extension uniquely.
\end{proof}

To use $h$ for coloring, let $S\subseteq G\setminus\{0\}$ be symmetric and suppose that $h(\alpha^n)\in S$ for every $n\geq0$.
An edge difference $\pm\alpha^n$ then maps to $\pm h(\alpha^n)\in S$, so $h$ gives a graph homomorphism $h:\Gamma_\alpha\to\Cay(G,S)$.
If $c$ is a proper coloring of the target graph, then $c\circ h$ properly colors $\Gamma_\alpha$.
The same composition preserves the separation required for a circular coloring.
Thus both the ordinary and circular chromatic numbers of $\Gamma_\alpha$ are bounded above by those of the target.

A useful special case is evaluation in a commutative ring $R$ with identity.
If $t\in R$ satisfies $f(t)=0$, then $\alpha^n\mapsto t^n$ satisfies the shifted relations and defines a ring homomorphism $\Z[\alpha]\to R$.
When the step images are nonzero, we can color the resulting Cayley graph and pull the coloring back.
Subsequent maps need only be additive, and the final coloring need only be a graph coloring; preserving multiplication is not required.
We will use this evaluation principle repeatedly, specifying the relation that vanishes and the images of the steps in each target.

The opening example used a multiplier to turn a cyclic quotient into a circular coloring.
We now formulate the separation condition behind that construction and apply it to the simplest evaluations.

\subsection{Multipliers, Characters, and Circulant Targets}
\label{subsec:characters-circulants}

The multiplier method on a cyclic quotient chooses an integer $\lambda$ and sends a residue $x$ modulo $p$ to $\lambda x/p+\Z$; see Liu's survey~\cite[Section~2]{Liu2008}.
Choosing the multiplier controls the circular separation of the step images.
Characters express the same idea for a general additive group.
For $\Z[\alpha]$, we can prescribe the circle images of the powers of $\alpha$ directly, provided that they respect the polynomial relations and remain uniformly separated from zero.
Write $\|t\|_{\T}$ for the circular distance from $t$ to zero.
We seek values $u_n\in\T$ satisfying the shifted relations $\sum_{j=0}^{d}a_ju_{n+j}=0$ in $\T$ for every $n\geq0$.
The Shift Criterion then gives a character $\psi:\Z[\alpha]\to\T$ with $\psi(\alpha^n)=u_n$.
For coloring, we require in addition a uniform separation $\|u_n\|_{\T}\geq\delta$, where $0<\delta\leq1/2$.
These assignments need not come from evaluation at a fixed root of unity, and they may be periodic or nonperiodic.

\begin{lemma}[Character separation]
	\label{lem:character-separation}
	Let $G$ be an abelian group and let $S\subseteq G\setminus\{0\}$ be symmetric.
	If a character $\psi:G\to\T$ satisfies $\|\psi(s)\|_{\T}\geq\delta$ for every $s\in S$, where $0<\delta\leq1/2$, then
	\[\chi_c\bigl(\Cay(G,S)\bigr)\leq\frac1\delta.\]
	In particular, $\Cay(G,S)$ admits an ordinary coloring with at most $\lceil1/\delta\rceil$ colors.
\end{lemma}

\begin{proof}
	Additivity makes the circular distance between adjacent images at least $\delta$.
	Rescaling to a circle of circumference $1/\delta$ therefore gives a circular $1/\delta$-coloring.
	Taking integer parts of representatives in $[0,1/\delta)$ gives at most $\lceil1/\delta\rceil$ ordinary colors, since two representatives with the same integer part differ by less than one.
\end{proof}

Thus separation $1/3$ gives three colors, while separation $1/4$ gives four.
Once the values $u_n$ have been chosen, the circle color of a vertex $v=\sum_n b_n\alpha^n$ is explicitly $\psi(v)=\sum_n b_nu_n$ in $\T$.

These constructions also fit the description by homomorphisms to circulants.
For coprime positive integers $p\geq2q$, the \emph{circular clique}
\[K_{p/q}=\Cay\bigl(\Z/p\Z,\{q,q+1,\ldots,p-q\}\bigr)\]
is the standard target for a circular $p/q$-coloring; see Zhu~\cite{Zhu2001}.
Placing $j$ at $j/p\in\T$ gives separation $q/p$, and $K_{p/1}=K_p$ recovers ordinary colorings.
Conversely, a character with separation at least $q/p$ gives a graph homomorphism to $K_{p/q}$ by setting $c(v)=\lfloor p\widetilde\psi(v)\rfloor\pmod p$, where $\widetilde\psi(v)\in[0,1)$ represents $\psi(v)$.
For adjacent vertices with representatives $x<y$, we have $q\leq p(y-x)\leq p-q$, and hence $q\leq\lfloor py\rfloor-\lfloor px\rfloor\leq p-q$.
This rounding works even when the character has infinite image, although it need not preserve additivity.
When all character values lie in $(1/p)\Z/\Z$, no rounding is needed and the resulting map is additive.
For example, the character in Subsection~\ref{subsec:three-coloring} gives $\Gamma_\alpha\to K_{13/6}$; multiplication by six identifies its earlier target $C_{13}$ with this circular clique.
Figure~\ref{fig:character-rounding} illustrates the rounding map for $K_{7/2}$.

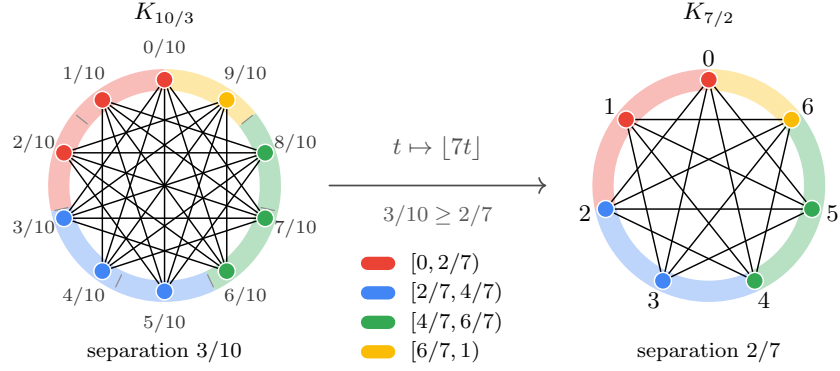
\begin{figure}[htbp]
	\centering
	\input{tikz-code/character-rounding.tex}
	\caption{A four-coloring pulled back along the rounding homomorphism $K_{10/3}\to K_{7/2}$.
	Colored arcs indicate the half-open color intervals.}
	\label{fig:character-rounding}
\end{figure}

The familiar evaluations at $1$, $-1$, and $i$ provide immediate applications.
First, evaluation at $1$ gives the following specialization of the parity criterion of Cervantes and Krebs~\cite[Lemma~2.11]{CervantesKrebs2023}.

\begin{lemma}
	\label{lem:bipartite-criterion}
	The graph $\Gamma_\alpha$ is bipartite if and only if $f(1)$ is even.
\end{lemma}

\begin{proof}
	If $f(1)$ is even, the map $p(\alpha)\mapsto p(1)\pmod2$ is a proper two-coloring.
	If $f(1)$ is odd, the relation $f(\alpha)=0$ gives a closed walk of odd length $\sum_j|a_j|\equiv f(1)\pmod2$, which excludes bipartiteness.
\end{proof}

Evaluation at $\pm1$ gives the three-color test of Cervantes and Krebs~\cite[Example~2.14]{CervantesKrebs2023}.
The next lemma sharpens these familiar congruence tests by mapping their quotients to cycles and using the separation of the resulting characters.

\begin{lemma}
	\label{lem:elementary-evaluation}
	Let $q\geq2$ be an integer.
	If $q$ divides $f(1)$, $f(-1)$, or $f(i)$, then $\Gamma_\alpha$ admits a graph homomorphism to $\Cay(\Z/q\Z,\{\pm1\})$, and
	\[\chi_c(\Gamma_\alpha)\leq\frac{q}{\lfloor q/2\rfloor}\leq3.\]
	In particular, $\Gamma_\alpha$ is bipartite when $q$ is even and three-colorable when $q$ is odd.
\end{lemma}

\begin{proof}
	If $q\mid f(1)$, evaluation gives $p(\alpha)\mapsto p(1)\pmod q$, sending every allowed step to $\pm1$.
	If $q\mid f(-1)$, use $p(\alpha)\mapsto p(-1)\pmod q$, with the same step images.
	For evaluation at $i$, the condition $f(i)\in q\Z[i]$ makes the composite
	\[\Z[\alpha]\longrightarrow\Z[i]/q\Z[i]\longrightarrow\Z/q\Z\]
	well-defined and additive.
	The first map evaluates at $i$, and the second sends $a+bi$ to $a+b\pmod q$.
	Every allowed step first maps to one of $\pm1,\pm i$ and then to $\pm1$.
	In each case, writing $h=\lfloor q/2\rfloor$, composition with $j\mapsto hj/q\pmod1$ gives a character whose step images have distance $h/q$ from zero.
	Lemma~\ref{lem:character-separation} proves the bounds.
\end{proof}

The cyclic target is $C_q$ for $q\geq3$ and $K_2$ for $q=2$.
In particular, $|f(1)|\ne1$ or $|f(-1)|\ne1$ gives three-colorability, with a zero evaluation permitting any modulus.
The same compatibility and separation conditions will also guide constructions beyond these evaluations.
For algebraic noninteger parameters, we will choose successive character values while maintaining separation from zero.
For the trace constructions and the sextic example, we will choose periodic values and verify their shifted relations.
Evaluating at $1$, $-1$, or $i$ reduces the step images to a cycle.
We next consider higher-order roots of unity, where an Eisenstein quotient will supply a common coloring target.

\subsection{Eisenstein Quotients and Circulant Colorings}
\label{subsec:eisenstein-quotients}

The triangular lattice supplies a target with three undirected step directions.
Let $\omega=e^{2\pi i/3}$, and write $\mathcal U=\{\pm1,\pm\omega,\pm\omega^2\}$ for the six units of $\Z[\omega]$.
The lattice has vertices $a+b\omega$ with $a,b\in\Z$, and its allowed steps are the elements of $\mathcal U$.

First identify lattice vertices whose two coordinates agree modulo two.
The four classes $0$, $1$, $\omega$, and $1+\omega$ form $K_4$, since the three nonzero classes are images of Eisenstein units.
The usual triangular-lattice coloring $a+b\omega\mapsto a+b\pmod3$ does not descend to this quotient: the identified vertices $\omega$ and $2+\omega$ have different colors.
Figure~\ref{fig:eisenstein-quotient} shows how identifying vertices can require a fourth color.

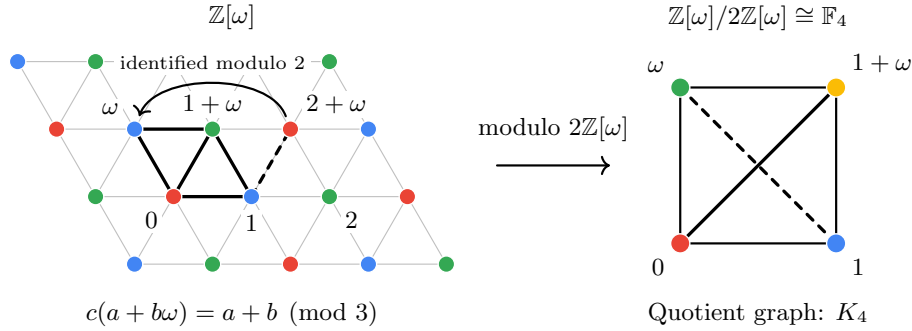
\begin{figure}[ht]
	\centering
	\resizebox{\linewidth}{!}{%
		\input{tikz-code/eisenstein-quotient-background.tex}%
	}
	\caption{Identifying the Eisenstein lattice modulo $2\Z[\omega]$ gives $K_4$.
	The solid and dashed marked edges give all six edges of the quotient.
	The lattice coloring assigns different colors to $\omega$ and $2+\omega$, so it does not descend.}
	\label{fig:eisenstein-quotient}
\end{figure}

For an integer $q\geq2$, let $E_q=\Cay(\Z[\omega]/q\Z[\omega],\mathcal U\bmod q)$.
Using the basis $1,\omega$ and the identity $\omega^2=-1-\omega$, we may also write
\[E_q=\Cay\left((\Z/q\Z)^2,\{\pm(1,0),\pm(0,1),\pm(1,1)\}\right).\]
Thus $E_q$ can be represented by a $q\times q$ array with horizontal, vertical, and $(1,1)$-diagonal edges, with indices read modulo $q$.
We first color these graphs.
Subsection~\ref{subsec:cyclotomic-quotients} will then construct additive maps that send arbitrary cyclotomic directions to Eisenstein units.

For $q\geq3$, we can color $E_q$ by first passing to the circulant
\[C_q(1,2)=\Cay\bigl(\Z/q\Z,\{\pm1,\pm2\}\bigr).\]
The coordinate sum $\sigma_q(i,j)=i+j\pmod q$ gives a graph homomorphism $\sigma_q:E_q\to C_q(1,2)$, since the three positive step directions map to $1$, $1$, and $2$.
An ordinary or circular coloring of this smaller target therefore pulls back to $E_q$.
For example, when $q=7$, the cyclic word $012\,|\,0123$ colors $C_7(1,2)$: every three consecutive entries are distinct, including across the cyclic join.
Assigning to $(i,j)$ the color of $i+j\pmod7$ gives the coloring of $E_7$ in Figure~\ref{fig:eisenstein-seven}.

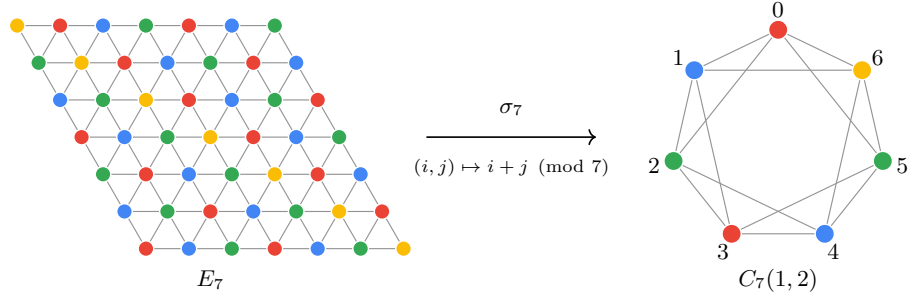
\begin{figure}[htbp]
	\centering
	\resizebox{\linewidth}{!}{%
		\input{tikz-code/eisenstein-seven.tex}%
	}
	\caption{The coordinate-sum map $E_7\to C_7(1,2)$ pulls back the coloring $012\,|\,0123$.
	Wraparound edges in $E_7$ are omitted.}
	\label{fig:eisenstein-seven}
\end{figure}

Heuberger's classification of the chromatic numbers of circulants with two generators~\cite[Theorem~3.2]{Heuberger2003}, together with $C_3(1,2)=K_3$, gives for $q\geq3$
\[\chi\bigl(C_q(1,2)\bigr)=
\begin{cases}
	3,&3\mid q,\\
	5,&q=5,\\
	4,&\text{otherwise}.
\end{cases}
\]
The exception $q=5$ is instructive: the cyclic target requires five colors, but $E_5$ itself has a four-coloring.
We therefore use the cyclic target when it is effective and color $E_5$ directly.

\begin{lemma}
	\label{lem:eisenstein-quotient-colors}
	For every integer $q\geq2$,
	\[\chi(E_q)=
	\begin{cases}
		3,&3\mid q,\\
		4,&3\nmid q.
	\end{cases}\]
	When $3\mid q$, every proper three-coloring is given by $c(i,j)=i+j\pmod3$, up to a permutation of the colors.
\end{lemma}

\begin{proof}
	For $q\geq3$ with $q\ne5$, let $b\in\{0,1,2\}$ be the remainder of $q$ modulo three and put $a=(q-4b)/3$.
	Then $a\geq0$ and $q=3a+4b$.
	Define a cyclic word of length $q$ by
	\[h_q=(012)^a(0123)^b,\]
	where the exponents denote concatenated copies, and color $s\in\Z/q\Z$ by the entry $h_q(s)$, indexed from zero.
	Every three consecutive entries are distinct, including across block joins and the cyclic join, so this is a proper four-coloring of $C_q(1,2)$.
	When $3\mid q$, we have $b=0$ and $h_q(s)=s\pmod3$, giving a three-coloring.
	Pulling it back along $\sigma_q$ gives the coloring
	\[c(i,j)=h_q(i+j\pmod q)\]
	of $E_q$.
	This extends the $q=7$ construction above.
	The case $q=2$ follows from $E_2=K_4$.
	For $q=5$, the circulant target is $C_5(1,2)=K_5$, so we use the direct four-coloring of $E_5$ in Figure~\ref{fig:eisenstein-five}.
	The accompanying array specifies the coloring at every $(i,j)$.
	Reading its indices modulo five, entries at differences $(1,0)$, $(0,1)$, and $(1,1)$ have different values, which verifies all edges, including the wraparound edges.

	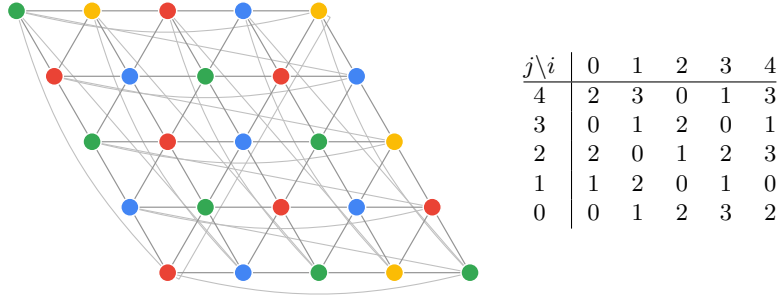
\begin{figure}[htbp]
		\centering
		\input{tikz-code/eisenstein-five-background.tex}
		\caption{A four-coloring of $E_5$, with wraparound edges in light gray.
		The array gives the color of $(i,j)$, with $i$ indexing columns and $j$ indexing rows.
		The labels $0,1,2,3$ denote red, blue, green, and yellow, respectively; all indices are read modulo five.}
		\label{fig:eisenstein-five}
	\end{figure}

	The graph $E_q$ contains a triangle, so at least three colors are necessary.
	Suppose that $q\geq3$ and $c$ is a proper three-coloring, and put $e_1=(1,0)$ and $e_2=(0,1)$.
	For every vertex $z$, the triangles with vertices $\{z,z+e_1,z+e_1+e_2\}$ and $\{z,z+e_2,z+e_1+e_2\}$ share an edge, so their third vertices have the same color:
	\[c(z+e_1)=c(z+e_2).\]
	Thus $c$ is constant on the cosets of $\langle e_1-e_2\rangle$, and we may write $c(i,j)=h(i+j)$ for a map $h:\Z/q\Z\to\{0,1,2\}$.
	The triangles show that $h(s),h(s+1),h(s+2)$ are pairwise distinct for every $s$.
	It follows that $h(s+3)=h(s)$.
	Since consecutive entries are distinct, this is compatible with period $q$ only when $3\mid q$.
	In that case, the first three values of $h$ determine a permutation of the coloring $s\mapsto s\pmod3$.
	Together with $E_2=K_4$ and the upper bounds above, this proves the result.
\end{proof}

\subsection{The Cyclotomic Congruence Criterion}
\label{subsec:cyclotomic-quotients}

We now evaluate at higher-order roots of unity.
Fix a primitive $m$th root $\zeta_m$, let $\Phi_m$ be its cyclotomic polynomial, and write $\mathcal O_m=\Z[\zeta_m]$.
If $f(\zeta_m)\in q\mathcal O_m$ for an integer $q\geq2$, the evaluation principle gives a ring homomorphism $\eta_{m,q}:\Z[\alpha]\to\mathcal O_m/q\mathcal O_m$ by $p(\alpha)\mapsto p(\zeta_m)\pmod{q\mathcal O_m}$.
The allowed steps map to classes of roots of unity.
These are units in a nonzero quotient ring, so none is zero.

To use the Eisenstein colorings, we seek an additive map $L_m:\mathcal O_m\to\Z[\omega]$ that sends every power of $\zeta_m$ into $\mathcal U$.
Such a map descends modulo $q$ and carries the quotient's allowed steps into those of $E_q$.
The identities $1+(-1)=0$ and $1+\omega+\omega^2=0$ supply the zero-sum pairs and triples needed to respect the prime-power relations.
The Shift Criterion (Proposition~\ref{prop:shift-criterion}) reduces this task to choosing labels that respect the cyclotomic relations.
For ninth roots, write $\zeta=\zeta_9$.
The polynomial $\Phi_9(x)=1+x^3+x^6$ gives the three relations
\begin{align*}
	1+\zeta^{3}+\zeta^{6} &= 0	\\
	\zeta+\zeta^{4}+\zeta^{7} &= 0 \\
	\zeta^2+\zeta^{5}+\zeta^{8} &= 0
\end{align*}
Assign $L_9(\zeta^j)=1$, $L_9(\zeta^{j+3})=\omega$ and $L_9(\zeta^{j+6})=\omega^2$ for $0\leq j\leq2$.
The first six powers form an integral basis, and $1+\omega+\omega^2=0$, so these assignments define an additive map with the required images.
Only additivity is needed: $L_9(\zeta)=1$, while $L_9(\zeta^3)=\omega$.
Figure~\ref{fig:ninth-roots-eisenstein} shows the three relations as closed walks and their common image under $L_9$.

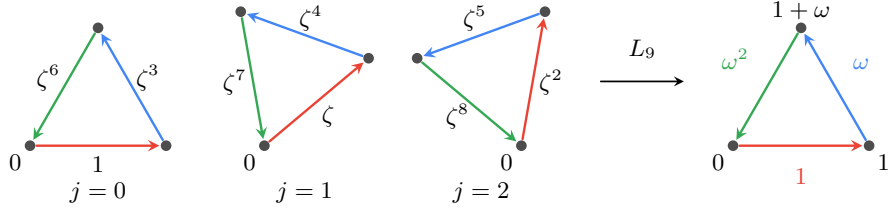
\begin{figure}[htbp]
	\centering
	\input{tikz-code/ninth-roots-eisenstein.tex}
	\caption{The three ninth-root walks have the same image under $L_9$.
	Edge colors indicate the images $1,\omega,\omega^2$.}
	\label{fig:ninth-roots-eisenstein}
\end{figure}

The same construction works for every prime-power order.

\begin{lemma}[Prime Power Case]
	\label{lem:prime-power-labels}
	For every prime power $m=p^a$ with $a\geq1$, there is an additive map $L_m:\mathcal O_m\to\Z[\omega]$ that sends every power of $\zeta_m$ into $\mathcal U$.
\end{lemma}

\begin{proof}
	Put $M=p^{a-1}$ and $\zeta=\zeta_m$.
	Since $\Phi_m(x)=1+x^M+\cdots+x^{(p-1)M}$, the powers of $\zeta$ satisfy
	\[\zeta^i+\zeta^{i+M}+\cdots+\zeta^{i+(p-1)M}=0\]
	for $0\leq i<M$.
	The first $(p-1)M=\varphi(m)$ powers form an integral basis, where $\varphi(m)=\deg\Phi_m$ is Euler's totient function.
	Thus labels on the $m$ powers define an additive map precisely when the labels in each displayed block sum to zero.
	For $p=2$, assign the labels $1,-1$ to each block.
	For odd $p$, assign $1,\omega,\omega^2$, followed by $(p-3)/2$ pairs $1,-1$.
	Each block has sum zero and all labels lie in $\mathcal U$.
	The powers repeat with period $m$, so every power has the required image.
\end{proof}

For general orders, we combine the maps from coprime factors by multiplying their labels.
We state this step for an arbitrary commutative target ring so that it also applies to the finite-field constructions in Section~3.

\begin{lemma}[Composite Case]
	\label{lem:coprime-labels}
	Let $a,b\geq2$ be coprime, let $R$ be a commutative ring with identity, and let $L_a:\mathcal O_a\to R$ and $L_b:\mathcal O_b\to R$ be additive maps.
	There is an additive map $L:\mathcal O_{ab}\to R$ satisfying
	\[L(uv)=L_a(u)L_b(v)\]
	for all $u\in\mathcal O_a$ and $v\in\mathcal O_b$.
	If $L_a$ and $L_b$ send powers of their respective roots to units of $R$, then $L$ sends every power of $\zeta_{ab}$ to a unit of $R$.
\end{lemma}

\begin{proof}
	Write $\xi=\zeta_a$ and $\eta=\zeta_b$.
	Since $a$ and $b$ are coprime, $\xi\eta$ has order $ab$, and both $\xi$ and $\eta$ are powers of $\xi\eta$.
	Hence $\mathcal O_{ab}=\Z[\xi,\eta]$.
	Reduction by the monic polynomials $\Phi_a$ and $\Phi_b$ shows that the products $\xi^i\eta^j$ with $0\leq i<\varphi(a)$ and $0\leq j<\varphi(b)$ span $\mathcal O_{ab}$ over $\Z$.
	There are $\varphi(a)\varphi(b)=\varphi(ab)$ such products, equal to the degree of $\Q(\zeta_{ab})$ over $\Q$, so they form an integral basis.
	Define $L$ on this basis by
	\[L(\xi^i\eta^j)=L_a(\xi^i)L_b(\eta^j)\]
	and extend additively.
	Expanding $u$ and $v$ in their respective bases proves the displayed identity.
	Every power of $\zeta_{ab}$ is a product of powers of $\xi$ and $\eta$, so its image is a unit whenever the factor labels are units.
\end{proof}

Taking $R=\Z[\omega]$ and combining the prime-power maps gives $L_m:\mathcal O_m\to\Z[\omega]$ with every power mapped into $\mathcal U$ for all $m>1$.
For $m=1$, use $L_1(n)=n$ on $\mathcal O_1=\Z$.
Additivity gives $L_m(qz)=qL_m(z)$, so every $L_m$ descends modulo $q$.
We can now combine evaluation with the Eisenstein colorings.
To state the result, we write
\[f(x)\equiv0\pmod{q,\Phi_m(x)}\]
when $f(x)=qA(x)+\Phi_m(x)B(x)$ for some $A,B\in\Z[x]$.
Since $\mathcal O_m\cong\Z[x]/(\Phi_m)$, this is exactly the condition $f(\zeta_m)\in q\mathcal O_m$ used above.

\begin{proposition}[Cyclotomic Congruence Criterion]
	\label{prop:cyclotomic-congruence}
	If $m\geq1$ and $q\geq2$ satisfy
	\[f(x)\equiv0\pmod{q,\Phi_m(x)},\]
	then $\chi(\Gamma_\alpha)\leq4$.
	If $3\mid q$, then $\chi(\Gamma_\alpha)\leq3$.
\end{proposition}

\begin{proof}
	The congruence gives $f(\zeta_m)=qA(\zeta_m)$, so evaluation defines $\eta_{m,q}:\Z[\alpha]\to\mathcal O_m/q\mathcal O_m$ as above.
	Compose it with the reduction $L_{m,q}$ of $L_m$ to obtain
	\[\Z[\alpha]\xrightarrow{\;\eta_{m,q}\;}\mathcal O_m/q\mathcal O_m
	\xrightarrow{\;L_{m,q}\;}\Z[\omega]/q\Z[\omega].\]
	Every allowed step maps to the nonzero class of an Eisenstein unit, so the composite is a graph homomorphism $\Gamma_\alpha\to E_q$.
	Pulling back the colorings from Lemma~\ref{lem:eisenstein-quotient-colors} proves the bounds.
\end{proof}

For $q\geq3$, the construction gives the chain
\[\Gamma_\alpha\to E_q\xrightarrow{\;\sigma_q\;}C_q(1,2).\]
The first arrow combines cyclotomic evaluation with the Eisenstein labels, and the second adds the two Eisenstein coordinates.
For $q\geq3$ with $q\ne5$, a coloring of the circulant pulls back along this chain to color $\Gamma_\alpha$.
The cases $q=2,5$ use the direct colorings of $E_q$.

More generally, a coloring of $\Gamma_{\zeta_m}$ transfers by evaluation to $\Gamma_\alpha$ precisely when it is invariant under translation by every element of $f(\zeta_m)\mathcal O_m$.
It suffices to check invariance under the translations $\zeta_m^j f(\zeta_m)$ for $0\leq j<\varphi(m)$, since these generate the ideal additively.
In the congruence construction, these translations vanish modulo $q\mathcal O_m$.

To check the criterion, use polynomial long division to divide $f(x)$ by $\Phi_m(x)$.
The congruence holds precisely when every coefficient of the remainder is divisible by $q$.
A common divisor $q\geq2$ therefore gives a four-coloring, and divisibility of all coefficients by three gives a three-coloring.
A zero remainder satisfies the test for every $q$.

Only finitely many cyclotomic polynomials need to be checked.
If $\varphi(m)>d$, the remainder is $f$ itself, whose coefficients have gcd one because $f$ is primitive.
The elementary bound $\varphi(m)\geq\sqrt{m/2}$ therefore gives the explicit cutoff $m\leq2d^2$.
Within this range, we need only divide by the cyclotomic polynomials of degree at most $d$.
These divisions give a quick filter for graphs with low chromatic number.

A polynomial can nevertheless avoid all of these tests.
For example,
\[f(x)=x^{10}-x^6+x^5-x^4+1\]
is irreducible and has roots on the unit circle, but the remainder for each of the twenty orders with $\varphi(m)\leq10$ contains a coefficient $1$ or $-1$.
The criterion therefore gives no coloring bound for the corresponding graphs $\Gamma_\alpha$.
This example motivates a more flexible quotient: instead of seeking an integer modulus that annihilates $f(\zeta_m)$, we can impose the evaluated relation itself.
Trace polynomials will make that quotient easier to describe, and the character method will provide its colorings.

\subsection{Trace Polynomials and Cyclotomic Quotients}
\label{subsec:trace-quotients}
\label{subsec:reciprocity}
\label{subsec:cyclotomic-norm}

Cyclotomic evaluation can yield a useful finite target even when the evaluated relation is not divisible by an integer greater than one.
We first use a trace polynomial to express the relation in a real variable.
A worked eighth-root example will then show how the evaluated relation gives a finite target even when every integer congruence test fails.
For the remainder of this section, assume that $\alpha$ is algebraic and is not a root of unity.

Consider the parameter $\alpha=(\sqrt{14}+i\sqrt2)/4$.
Its minimal polynomial can be written as
\[f(x)=2x^4-3x^2+2=x^2\bigl(2(x+x^{-1})^2-7\bigr).\]
The expression $x+x^{-1}$ lets us work with the quadratic $F(y)=2y^2-7$ in place of the quartic.

The same reduction applies to every parameter under consideration here.
Since $\alpha^{-1}=\overline\alpha$, irreducibility implies that $f$ is reciprocal of even degree $2r$, meaning that its coefficients form a palindromic sequence.
Every such polynomial has a unique representation
\[f(x)=x^rF(x+x^{-1}),\]
where $F\in\Z[y]$ is called its \emph{trace polynomial}.
It has degree $r$ and the same leading coefficient as $f$.
For $\zeta_m=e^{2\pi i/m}$, write $t_m=\zeta_m+\zeta_m^{-1}=2\cos(2\pi/m)$.
We call $t_m$ the \emph{real evaluation point}.
The identity
\[f(\zeta_m)=\zeta_m^rF(t_m)\]
lets us perform the cyclotomic evaluation using the trace polynomial.
The factor $\zeta_m^r$ is a unit in $\Z[\zeta_m]$, so an integer $q\geq2$ divides $f(\zeta_m)$ precisely when it divides $F(t_m)$ in that ring.

The roots of unity considered in the table below give integer evaluation points.
These values are computationally convenient.
\[
\begin{array}{|c|c|c|}
\hline
\text{Value of }F & \text{Evaluation at} & \text{Target} \\
\hline
F(2)  & 1       & \Cay(\Z/q\Z,\{\pm1\}) \\
\hline
F(-2) & -1      & \Cay(\Z/q\Z,\{\pm1\}) \\
\hline
F(0)  & i       & \Cay(\Z/q\Z,\{\pm1\}) \\
\hline
F(1)  & \zeta_6 & E_q \\
\hline
F(-1) & \zeta_3 & E_q \\
\hline
\end{array}
\]
In each row, an integer $q\geq2$ dividing the displayed value of $F$ gives the indicated target.
The first three rows follow from Lemma~\ref{lem:elementary-evaluation} and give $\chi_c(\Gamma_\alpha)\leq q/\lfloor q/2\rfloor\leq3$.
The last two follow from the Cyclotomic Congruence Criterion and give four colors, or three when $3\mid q$.
An evaluation of $F$ equal to $1$ or $-1$ supplies no such divisor, so that particular evaluation gives no coloring by this test.

For the quartic example above, $F(y)=2y^2-7$ gives $F(-2)=F(2)=1$.
Those two evaluations supply no modulus, but $F(0)=-7$ allows us to choose $q=7$ in the third row.
We obtain a map to $C_7$ and hence $\chi_c(\Gamma_\alpha)\leq7/3$.
Since $f(1)=1$ is odd, Lemma~\ref{lem:bipartite-criterion} also gives $\chi(\Gamma_\alpha)=3$.

For the degree-ten example from Subsection~\ref{subsec:cyclotomic-quotients}, the trace polynomial is
\[F(y)=y^5-5y^3+4y+1.\]
All five integer evaluation points give $F(y)=1$, and this polynomial also escapes every cyclotomic congruence test from that subsection.
Evaluating at an eighth root nevertheless gives a useful relation.
The real evaluation point is $t_8=\sqrt2$, so
\[f(\zeta_8)=\zeta_8^5F(\sqrt2)=\zeta_8^5(1-2\sqrt2).\]
To make this evaluation vanish, we identify elements whose difference is a multiple of $1-2\sqrt2$.
Since $\zeta_8^5$ is a unit, this also makes $f(\zeta_8)$ vanish and gives the quotient
\[R_8(f)=\Z[\zeta_8]/(1-2\sqrt2)\Z[\zeta_8].\]

We can describe this quotient using two coordinates modulo seven.
First consider its real coefficients, which lie in
\[S_8=\Z[\sqrt2]/(1-2\sqrt2)\Z[\sqrt2].\]
Using brackets for quotient classes, the defining relation is $2[\sqrt2]=1$.
Squaring gives $8=1$, so $7=0$ and hence $[\sqrt2]=4$.
Conversely, substituting $\sqrt2\mapsto4$ modulo seven respects both $(\sqrt2)^2=2$ and $1-2\sqrt2=0$, so $S_8\cong\Z/7\Z$.
The identity $\zeta_8^2=\sqrt2\zeta_8-1$ gives
\[\Z[\zeta_8]=\Z[\sqrt2]\oplus\Z[\sqrt2]\zeta_8.\]
Thus every element of $R_8(f)$ has a unique expression $a+b[\zeta_8]$ with $a,b\in\Z/7\Z$.
The quotient has $49$ elements, and the relation $[\zeta_8]^2=4[\zeta_8]-1$ determines the coordinates of every root direction.

Let $\Gamma_8(f)$ be the Cayley graph on $R_8(f)$ with steps $\pm[\zeta_8^j]$ for $0\leq j<8$.
The evaluation principle gives a graph homomorphism $\Gamma_\alpha\to\Gamma_8(f)$ by $p(\alpha)\mapsto[p(\zeta_8)]$: the defining ideal makes $f(\zeta_8)$ vanish, and the root classes are nonzero units.
To color the finite target, use the additive map $h:R_8(f)\to\Z/7\Z$ where 
\[h\bigl(a+b[\zeta_8]\bigr)=5a+4b\pmod7.\]
The quadratic relation gives the successive labels of $1,\zeta_8,\ldots,\zeta_8^7$ as
\[(5,4,4,5,2,3,3,2).\]
Every allowed step therefore maps to one of $2,3,4,5$, giving
\[\Gamma_\alpha\longrightarrow\Gamma_8(f)
\longrightarrow K_{7/2}
=\Cay(\Z/7\Z,\{2,3,4,5\}).\]
These step images have circular separation at least $2/7$, so $\chi_c(\Gamma_\alpha)\leq7/2$.
The independent sets $\{0,1\}$, $\{2,3\}$, $\{4,5\}$, and $\{6\}$ of $K_{7/2}$ give an explicit four-coloring, shown in Figure~\ref{fig:eighth-root-coloring}.

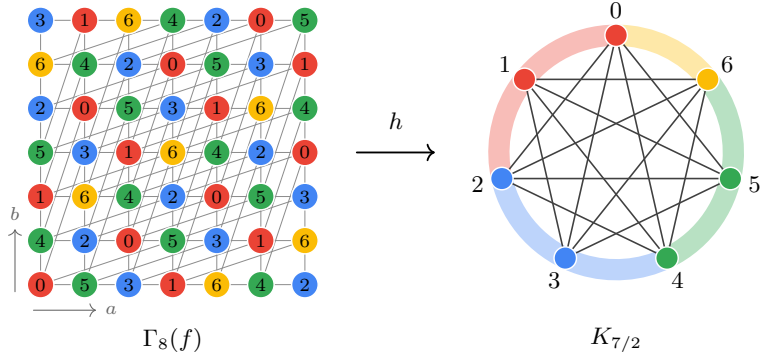
\begin{figure}[htbp]
    \centering
    \input{tikz-code/eighth-root-coloring.tex}
    \caption{The map $h:a+b[\zeta_8]\mapsto5a+4b\pmod7$ pulls back a four-coloring of $K_{7/2}$ to $\Gamma_8(f)$.
    Array entries are the values of $h$; wraparound edges are omitted.}
    \label{fig:eighth-root-coloring}
\end{figure}

The size of this target is already visible in the two conjugate evaluations of $F$:
\[\left|F(\sqrt2)F(-\sqrt2)\right|
=\left|(1-2\sqrt2)(1+2\sqrt2)\right|=7.\]
This absolute real norm counts the choices for each coordinate; its square counts the vertices of $\Gamma_8(f)$.
The norm will also control the available colorings when we pass to other cyclotomic orders.

\subsection{The Cyclotomic Norm Criterion}
\label{subsec:general-norm-criterion}

We now extend the construction from the degree-ten example.
Fix $m\geq3$ and retain the notation $\mathcal O_m=\Z[\zeta_m]$.
As in the example, impose the evaluated relation directly by setting
\[R_m(f)=\mathcal O_m/f(\zeta_m)\mathcal O_m
=\mathcal O_m/F(t_m)\mathcal O_m.\]
The two ideals agree because their generators differ by the unit $\zeta_m^r$.
Irreducibility and the assumption that $\alpha$ is not a root of unity give $f(\zeta_m)\ne0$, so the quotient is finite.
It has more than one element precisely when $f(\zeta_m)$ is not a unit.
In that case, the same evaluation argument gives a graph homomorphism
\[\Gamma_\alpha\longrightarrow\Gamma_m(f)
=\Cay\bigl(R_m(f),\{\pm[\zeta_m^j]:0\leq j<m\}\bigr).\]
These targets belong to the cyclotomic graph family studied by Zhou~\cite{Zhou2019}: in his notation, $\Gamma_m(f)=G_m(A)$ with $A=f(\zeta_m)\mathcal O_m$.
We use colorings of these graphs to obtain bounds for $\Gamma_\alpha$.
If $f(\zeta_m)\in q\mathcal O_m$, the further reduction $R_m(f)\to\mathcal O_m/q\mathcal O_m$ recovers the earlier Cyclotomic Congruence Criterion.

The two-coordinate description also extends to every $m\geq3$.
The relation
\[\zeta_m^2=t_m\zeta_m-1\]
expresses every power of $\zeta_m$ as $a+b\zeta_m$ with $a,b\in\Z[t_m]$.
This expression is unique because the coefficients are real and $\zeta_m$ is nonreal, so
\[\mathcal O_m=\Z[t_m]\oplus\Z[t_m]\zeta_m.\]
It also follows that $\Z[t_m]$ is the full ring of integers of $\Q(t_m)$: an algebraic integer in this real field belongs to $\mathcal O_m$ and has $b=0$ in this representation.
Defining
\[S_m=\Z[t_m]/F(t_m)\Z[t_m],\]
we obtain $R_m(f)\cong S_m^2$ as additive groups, with $(a,b)$ representing $a+b[\zeta_m]$.

To count these coordinates, put $d_m=\varphi(m)/2$, the degree of the real field $\Q(t_m)$.
The absolute real norm is the absolute value of the product over its $d_m$ conjugate evaluations:
\[D_m = \left|N_{\mathbb Q(t_m)/\mathbb Q}\!\left(F(t_m)\right)\right| = \left|\prod_{\substack{1\leq k<m/2\\\gcd(k,m)=1}}F\bigl(2\cos(2\pi k/m)\bigr)\right|.\]
Multiplication by $F(t_m)$ on an integer basis of the real ring has determinant equal to its norm.
The absolute determinant is the index of its image ideal, so $|S_m|=D_m$.
Since the full quotient has two such coordinates,
\[|R_m(f)|=D_m^2.\]
Thus $D_m>1$ gives a nontrivial finite target, while $D_m=1$ gives the one-element quotient.

In the example, $D_8=7$ made each coordinate cyclic.
In general, $S_m$ is a quotient of $\Z^{d_m}$ and hence a product of at most $d_m$ finite cyclic groups.
Consequently, $\Gamma_m(f)$ is a multidimensional circulant with at most $2d_m=\varphi(m)$ cyclic coordinates.
It has two cyclic coordinates whenever $S_m$ is cyclic, in particular when $D_m$ is prime.

The quadratic relation computes the allowed steps just as it did for eighth roots.
Write $v_j\in S_m^2$ for the coordinate vector of $[\zeta_m^j]$.
Starting with $v_0=(1,0)$ and $v_1=(0,1)$, the remaining vectors satisfy
\[v_{j+1}=[t_m]v_j-v_{j-1}\]
for $j\geq1$.
A character of the quotient is specified by two additive characters $\lambda,\mu:S_m\to\T$ through
\[\psi\bigl(a+b[\zeta_m]\bigr)=\lambda(a)+\mu(b).\]
Its image is finite cyclic, so keeping the root images separated from zero gives a circulant coloring target, as the map $h$ did in the example.
The following theorem guarantees separation at least $1/4$ when $D_m$ is sufficiently large.
We start with the Eisenstein character on $\mathcal O_m$, whose root images have separation $1/3$.
It need not vanish on the evaluated ideal, so it need not define a character of the quotient.
The aim is to adjust it to vanish there while moving each root image by at most $1/12$.

There are two steps.
First, the norm is a product of conjugate sizes, and a large product can still have a very small factor.
Multiplying the ideal generator by a suitable real unit balances those sizes without changing the ideal.
Second, we approximate the original character by one that vanishes on the ideal.
The balanced conjugates make the approximation error small on every root direction, preserving separation at least $1/4$.

\begin{theorem}[Cyclotomic Norm Criterion]
	\label{thm:cyclotomic-norm}
	For each $m\geq3$, there is a constant $B_m\geq1$, depending only on $m$, with the following property.
	If the trace polynomial $F$ of a noncyclotomic algebraic direction parameter satisfies $D_m>B_m$, then $R_m(f)$ admits a character whose separation on every root direction is at least $1/4$.
	Consequently, $\Gamma_m(f)$ maps to a four-colorable circulant, and
	\[\chi(\Gamma_\alpha)\leq\chi(\Gamma_m(f))\leq4.\]
\end{theorem}

\begin{proof}
	Fix $m$ and put $K=\Q(t_m)$, $\beta=F(t_m)$, and $D=D_m$.
	The Eisenstein construction in Subsection~\ref{subsec:cyclotomic-quotients} gives a character $\psi_0:\mathcal O_m\to\T$ with $\|\psi_0(\zeta_m^j)\|_{\T}=1/3$ for every $j$.

	First we choose a generator of the same ideal whose conjugates are all large.
	Let $\sigma_1,\ldots,\sigma_{d_m}$ be the real embeddings of $K$.
	Dirichlet's unit theorem says that the vectors $(\log|\sigma_i(u)|)_{i=1}^{d_m}$, for units $u\in\Z[t_m]^\times$, form a full lattice in the hyperplane with coordinate sum zero~\cite[Chapter~8]{SteinANT}.
	The vector $(\log|\sigma_i(\beta)|-\log D/d_m)_{i=1}^{d_m}$ lies in this hyperplane.
	Reducing it into a fixed bounded fundamental region by a unit therefore gives an associate $\beta'$ and a constant $c_m>0$, independent of $\beta$, such that
	\[|\sigma_i(\beta')|\geq c_mD^{1/d_m}\]
	for $1\leq i\leq d_m$.
	Since $\beta'\mathcal O_m=\beta\mathcal O_m$, we may use $\beta'$ to construct a character of the same quotient.

	Next we approximate $\psi_0$ by a character that vanishes on $\beta'\mathcal O_m$.
	Extend each $\sigma_i$ to one complex embedding of $\Q(\zeta_m)$ and identify $\mathcal O_m$ with its image lattice $\Lambda$ in $V=\C^{d_m}$.
	Regard $V$ as a real Euclidean space with inner product $\langle z,w\rangle=\operatorname{Re}\sum_i z_i\overline{w_i}$.
	Its dual lattice is $\Lambda^*=\{v\in V:\langle v,z\rangle\in\Z\text{ for all }z\in\Lambda\}$.
	Lifting the values of $\psi_0$ on an integer basis of $\Lambda$ gives a vector $w_0\in V$ with $\psi_0(z)=\langle w_0,z\rangle\pmod1$.
	A vector $w$ defines a character vanishing on $\beta'\Lambda$ when $\beta'w$ belongs to the dual lattice.
	We therefore approximate $\beta'w_0$ by a dual-lattice vector.
	Since $\Lambda^*$ has a covering radius depending only on $m$, choose $v\in\Lambda^*$ within that distance of $\beta'w_0$ and set $w=v/\beta'$.
	Here multiplication and division by $\beta'$ act coordinatewise through its real embeddings, so
	\[\langle w,\beta'z\rangle=\langle v,z\rangle\in\Z\]
	for every $z\in\Lambda$.
	Thus $\psi(z)=\langle w,z\rangle\pmod1$ descends to $R_m(f)$.

	Dividing the bounded approximation error by the conjugates of $\beta'$ bounds $\|w-w_0\|$ by a constant depending only on $m$ times $D^{-1/d_m}$.
	Every root vector has norm $\sqrt{d_m}$, so Cauchy--Schwarz and the triangle inequality on $\T$ give
	\[\|\psi([\zeta_m^j])\|_{\T}\geq\frac13-C_mD^{-1/d_m}\]
	for a constant $C_m>0$ depending only on $m$.
	For $D>B_m$ with $B_m\geq1$ sufficiently large, this separation is at least $1/4$, and Lemma~\ref{lem:character-separation} gives the claimed coloring bounds.
\end{proof}

The argument gives an existence theorem at each fixed order.
A numerical screening cutoff would also require bounds for the constants coming from the unit lattice and the dual lattice.
The real degree enters the approximation through the exponent $1/d_m=2/\varphi(m)$.
For a fixed order, the theorem leaves only finitely many principal real ideals of bounded norm as possible exceptions.

The eighth-root example above is part of a stronger result: at this order, we can color every nontrivial quotient explicitly and improve the circular bound in most cases.
Polynomial division gives $F(y)\equiv A+By\pmod{y^2-2}$, with real norm $D_8=|A^2-2B^2|$.

For eighth roots, the real evaluation point is $t_8=\sqrt2$, and $\mathcal O_8$ has integer basis $1,\zeta_8,\zeta_8^2,\zeta_8^3$.
The allowed directions are these basis elements and their negatives, so $\Gamma_8(f)$ is a finite quotient of the four-dimensional integer grid.
A character with only two step sizes gives a stronger bound than four colors for almost every nontrivial norm.

The earlier degree-ten example shows how a unit can simplify the evaluated relation without changing the quotient.
Since $-(1+\sqrt2)$ is a unit in $\Z[\sqrt2]$, the identity
\[-(1+\sqrt2)(1-2\sqrt2)=3+\sqrt2\]
lets us use $3+\sqrt2$ in place of $1-2\sqrt2$.
Its coefficients $a=3$ and $b=1$ satisfy $a\geq2b\geq0$, and its norm is $a^2-2b^2=7$.
The construction below will recover the four basis labels $5,4,4,5$ used in that example.

\begin{lemma}[Eighth-root norm criterion]
	\label{lem:eighth-root-norm}
	Suppose $f(1)$ is odd, and write $F(y)\equiv A+By\pmod{y^2-2}$ and $D=|A^2-2B^2|$.
	If $D>1$, then
	\[\chi_c(\Gamma_\alpha)\leq
	\begin{cases}
	7/2,&D=7,\\
	3,&D\ne7.
	\end{cases}\]
	In particular, four colors suffice whenever $D>1$.
\end{lemma}

\begin{proof}
	Write $\zeta=\zeta_8$ and $\beta=A+B\sqrt2$.
	Since $A\equiv F(2)=f(1)\pmod2$, the coefficient $A$ is odd.
	Multiplying $\beta$ by a unit of $\Z[\sqrt2]$ leaves its ideal unchanged.
	The automorphism $\zeta\mapsto-\zeta$ sends $\sqrt2$ to $-\sqrt2$ and permutes the allowed steps, so replacing $\beta$ by its real conjugate gives an isomorphic quotient graph.

	We first use these operations to arrange the coefficients as in the example.
	The unit $1+\sqrt2$ has norm $-1$, so multiplication by it, if needed, makes the norm positive; an overall change of sign then makes both real embeddings positive.
	Multiplication by $3+2\sqrt2$ scales their ratio by $(3+2\sqrt2)^2$, while conjugation inverts the ratio.
	A suitable power of this unit, followed if necessary by conjugation, therefore gives a generator $a+b\sqrt2$ with
	\[1\leq\frac{a+b\sqrt2}{a-b\sqrt2}\leq3+2\sqrt2.\]
	Its norm is $a^2-2b^2=D$.
	The ratio bounds are equivalent to $a\geq2b\geq0$, and the coefficient $a$ remains odd under these operations.

	Now put $v=(D+a)/2$ and $w=v-b$, which are integers because $D$ and $a$ are odd.
	Assign the basis $1,\zeta,\zeta^2,\zeta^3$ the labels $v,w,w,v$ modulo $D$, defining an additive map $L:\Z[\zeta]\to\Z/D\Z$.
	To check that $L$ respects the quotient relation, it suffices to check the four multiples $(a+b\sqrt2)\zeta^j$ for $0\leq j<4$.
	Using $\sqrt2=\zeta-\zeta^3$ and $\zeta^4=-1$, their images reduce to the two expressions
	\begin{align*}
		(a-b)v+bw&=\frac{(a+1)D}{2},\\
		bv+(a+b)w&=\frac{(a+2b+1)D}{2}.
	\end{align*}
	Both are integer multiples of $D$, so $L$ descends to the normalized quotient and sends every allowed step into $\{\pm v,\pm w\}$.

	It remains to bound the separation of these labels.
	Since $a$ is odd, the inequality $a\geq2b$ gives $b\leq(a-1)/2$.
	For $a\geq5$, this implies
	\[D=a^2-2b^2\geq\frac{a^2+2a-1}{2}\geq3a.\]
	The case $a=1$ gives $D=1$, which is excluded.
	For $a=3$, the possibilities $b=0,1$ give $D=9,7$, respectively.
	Thus $a/D\leq1/3$ unless $D=7$, when $a/D=3/7$.
	The character $j\mapsto j/D\pmod1$ sends the two labels to
	\[\frac12\leq\frac wD\leq\frac vD=\frac{1+a/D}{2}<1.\]
	Every signed step therefore has circular distance at least $(1-a/D)/2$ from zero: at least $1/3$ when $D\ne7$, and $2/7$ when $D=7$.
	In particular, the step images are nonzero, giving a graph homomorphism
	\[\Gamma_8(f)\longrightarrow\Cay\bigl(\Z/D\Z,\{\pm v,\pm w\}\bigr).\]
	Pulling back along $\Gamma_\alpha\to\Gamma_8(f)$ and applying Lemma~\ref{lem:character-separation} proves the bounds.
\end{proof}

When $f(1)$ is even, the bipartite criterion already gives two colors.
When $D_8=1$, the quotient has one element and supplies no coloring by evaluation into that quotient.
The eighth-root criterion applies to trace polynomials of arbitrary degree and leading coefficient.

The general criterion separates two tasks: coloring the quotient at a fixed order and choosing an order for a prescribed minimal polynomial.
For each fixed $m$, sufficiently large $D_m$ guarantees a four-colorable target, and the eighth-root construction also settles its small nontrivial norms.
The constants in the general theorem depend on $m$, so finding a successful order for every noncyclotomic algebraic direction parameter remains open.
The following sections apply these evaluation and character constructions to the different algebraic types of direction parameter.

%% file: tikz-code/ck-three-color.tex
\begin{tikzpicture}[
	scale=1.4,
	line cap=round,
	line join=round,
	step label/.style={
		font=\scriptsize,
		fill=white,
		inner sep=1pt
	}
]
	\definecolor{googleRed}{HTML}{EA4335}
	\definecolor{googleBlue}{HTML}{4285F4}
	\definecolor{googleGreen}{HTML}{34A853}

	\pgfmathsetmacro{\ay}{sqrt(39)/8}
	\pgfmathsetmacro{\aay}{5*sqrt(39)/32}

	\coordinate (P0) at (0,0);
	\coordinate (P1) at (0.21875,-\aay);
	\coordinate (P2) at (1.21875,-\aay);
	\coordinate (P3) at (2.21875,-\aay);
	\coordinate (P4) at (2.84375,{\ay-\aay});
	\coordinate (P5) at (3.84375,{\ay-\aay});
	\coordinate (P6) at (4.84375,{\ay-\aay});

	\foreach \i/\steptext/\labelcode in {
		1/-\alpha^2/2,
		2/1/1,
		3/1/1,
		4/\alpha/1,
		5/1/1,
		6/1/1%
	} {
		\pgfmathtruncatemacro{\previous}{\i-1}
		\ifcase\labelcode\relax
			\tikzset{edge position/.style={above=3pt}}
		\or
			\tikzset{edge position/.style={below=3pt}}
		\or
			\tikzset{edge position/.style={left=3pt}}
		\or
			\tikzset{edge position/.style={right=3pt}}
		\fi
		\ifnum\i=4\relax
			\tikzset{edge position/.style={below=3pt,xshift=5pt}}
		\fi
		\draw[thick] (P\previous) -- (P\i)
			node[midway,step label,edge position] {$\steptext$};
	}

	\coordinate (Q0) at (0,0);
	\coordinate (Q1) at (0.625,\ay);
	\coordinate (Q2) at (0.84375,{\ay-\aay});
	\coordinate (Q3) at (1.46875,{2*\ay-\aay});
	\coordinate (Q4) at (2.09375,{3*\ay-\aay});
	\coordinate (Q5) at (2.3125,{3*\ay-2*\aay});
	\coordinate (Q6) at (2.9375,{4*\ay-2*\aay});
	\coordinate (Q7) at (3.15625,{4*\ay-3*\aay});
	\coordinate (Q8) at (3.78125,{5*\ay-3*\aay});
	\coordinate (Q9) at (4,{5*\ay-4*\aay});
	\coordinate (Q10) at (4.625,{6*\ay-4*\aay});
	\coordinate (Q11) at (P6);

	\foreach \i/\steptext/\labelcode in {
		1/\alpha/0,
		2/-\alpha^2/2,
		3/\alpha/0,
		4/\alpha/0,
		5/-\alpha^2/3,
		6/\alpha/0,
		7/-\alpha^2/3,
		8/\alpha/0,
		9/-\alpha^2/3,
		10/\alpha/0,
		11/-\alpha^2/3%
	} {
		\pgfmathtruncatemacro{\previous}{\i-1}
		\ifcase\labelcode\relax
			\tikzset{edge position/.style={above=5pt}}
		\or
			\tikzset{edge position/.style={below=5pt}}
		\or
			\tikzset{edge position/.style={left=5pt}}
		\or
			\tikzset{edge position/.style={right=5pt}}
		\fi
		\ifnum\labelcode>1\relax
			\tikzset{edge position/.style={right=-20pt,yshift=-5pt}}
		\fi
		\draw[thick,dashed] (Q\previous) -- (Q\i)
			node[midway,step label,edge position] {$\steptext$};
	}

	\foreach \vertexname in {P0,P2,P5,Q2} {
		\filldraw[
			fill=googleRed,
			draw=white,
			line width=0.5pt
		] (\vertexname) circle[radius=2.5pt];
	}

	\foreach \vertexname in {P3,P6,Q1,Q3,Q5,Q7,Q9} {
		\filldraw[
			fill=googleBlue,
			draw=white,
			line width=0.5pt
		] (\vertexname) circle[radius=2.5pt];
	}

	\foreach \vertexname in {P1,P4,Q4,Q6,Q8,Q10} {
		\filldraw[
			fill=googleGreen,
			draw=white,
			line width=0.5pt
		] (\vertexname) circle[radius=2.5pt];
	}

	\node[left=4pt] at (P0) {$0$};
	\node[below=10pt] at (P6) {$p(\alpha)=q(\alpha)$};

	\begin{scope}[
		shift={(5.7,-0.75)},
		scale=0.5,
		transform shape
	]
		\draw[rounded corners=3pt] (0,0) rectangle (4.1,2.7);

		\fill[googleRed] (0.35,2.35) circle[radius=2.5pt];
		\node[right,font=\small] at (0.7,2.35) {color $0$};

		\fill[googleBlue] (0.35,1.9) circle[radius=2.5pt];
		\node[right,font=\small] at (0.7,1.9) {color $1$};

		\fill[googleGreen] (0.35,1.45) circle[radius=2.5pt];
		\node[right,font=\small] at (0.7,1.45) {color $2$};

		\draw[thick] (0.15,0.9) -- (0.6,0.9);
		\node[right,font=\small] at (0.7,0.9)
			{$p(x)=4+x-x^2$};

		\draw[thick,dashed] (0.15,0.35) -- (0.6,0.35);
		\node[right,font=\small] at (0.7,0.35)
			{$q(x)=6x-5x^2$};
	\end{scope}
\end{tikzpicture}

%% file: tikz-code/ck-circular-color.tex
\begin{tikzpicture}[
	scale=1.05,
	line cap=round,
	line join=round,
	edge label/.style={
		font=\scriptsize,
		fill=white,
		inner sep=2pt
	},
	vertex label/.style={
		font=\scriptsize,
		fill=white,
		inner sep=1.5pt
	}
]
	\definecolor{googleRed}{HTML}{EA4335}
	\definecolor{googleBlue}{HTML}{4285F4}
	\definecolor{googleGreen}{HTML}{34A853}

	\pgfmathsetmacro{\ay}{sqrt(39)/8}
	\pgfmathsetmacro{\aay}{5*sqrt(39)/32}

	\foreach \j in {0,...,4} {
		\coordinate (V\j) at (\j,0);
	}
	\foreach \j in {5,...,9} {
		\coordinate (V\j) at
			({4-(\j-4)*5/8},{-(\j-4)*\ay});
	}
	\foreach \j in {10,...,12} {
		\coordinate (V\j) at
			({7/8-(\j-9)*7/32},{-5*\ay+(\j-9)*\aay});
	}
	\coordinate (V13) at (V0);

	\foreach \j in {0,...,3} {
		\pgfmathtruncatemacro{\nextvertex}{\j+1}
		\draw[thick] (V\j) -- (V\nextvertex)
			node[midway,above=5pt,edge label] {$1$};
	}
	\foreach \j in {4,...,8} {
		\pgfmathtruncatemacro{\nextvertex}{\j+1}
		\draw[thick] (V\j) -- (V\nextvertex)
			node[midway,right=4pt,edge label] {$-\alpha$};
	}
	\foreach \j in {9,...,12} {
		\pgfmathtruncatemacro{\nextvertex}{\j+1}
		\draw[thick] (V\j) -- (V\nextvertex)
			node[midway,left=4pt,edge label] {$\alpha^2$};
	}

	\foreach \j in {0,...,12} {
		\pgfmathtruncatemacro{\residue}{mod(6*\j,13)}

		\ifnum\residue<6\relax
			\def\pointcolor{googleRed}
		\else
			\ifnum\residue<12\relax
				\def\pointcolor{googleBlue}
			\else
				\def\pointcolor{googleGreen}
			\fi
		\fi

		\filldraw[
			fill=\pointcolor,
			draw=white,
			line width=0.5pt
		] (V\j) circle[radius=2.5pt];

		\ifnum\j=0\relax
			\node[vertex label,above left=4pt] at (V\j) {$0$};
		\else
			\ifnum\j=4\relax
				\node[vertex label,above right=4pt] at (V\j)
					{$\frac{\residue}{13}$};
			\else
				\ifnum\j=9\relax
					\node[vertex label,below=5pt] at (V\j)
						{$\frac{\residue}{13}$};
				\else
					\ifnum\j<4\relax
						\node[vertex label,below=6pt] at (V\j)
							{$\frac{\residue}{13}$};
					\else
						\ifnum\j<9\relax
							\node[vertex label,right=4pt] at (V\j)
								{$\frac{\residue}{13}$};
						\else
							\node[vertex label,right=4pt] at (V\j)
								{$\frac{\residue}{13}$};
						\fi
					\fi
				\fi
			\fi
		\fi
	}

	\draw[->,thick] (4.75,-1.8) -- (5.55,-1.8)
		node[midway,above=4pt] {$\psi$};

	\begin{scope}[shift={(8,-1.8)}]
		\def\circleradius{1.75}
		\def\labelradius{2.12}
		\pgfmathsetmacro{\redend}{90+360*6/13}
		\pgfmathsetmacro{\blueend}{90+360*12/13}

		\draw[googleRed!35,line width=8pt]
			(90:\circleradius)
			arc[start angle=90,end angle=\redend,radius=\circleradius];
		\draw[googleBlue!35,line width=8pt]
			(\redend:\circleradius)
			arc[start angle=\redend,end angle=\blueend,radius=\circleradius];
		\draw[googleGreen!40,line width=8pt]
			(\blueend:\circleradius)
			arc[start angle=\blueend,end angle=450,radius=\circleradius];

		\foreach \k in {0,...,12} {
			\pgfmathsetmacro{\pointangle}{90+360*\k/13}
			\coordinate (C\k) at (\pointangle:\circleradius);
		}

		\foreach \j in {0,...,12} {
			\pgfmathtruncatemacro{\fromresidue}{mod(6*\j,13)}
			\pgfmathtruncatemacro{\toresidue}{mod(6*(\j+1),13)}
			\draw[gray!60,thin] (C\fromresidue) -- (C\toresidue);
		}

		\draw[->,thick] (C0) -- (C6);

		\foreach \k in {0,...,12} {
			\pgfmathsetmacro{\pointangle}{90+360*\k/13}

			\ifnum\k<6\relax
				\def\pointcolor{googleRed}
			\else
				\ifnum\k<12\relax
					\def\pointcolor{googleBlue}
				\else
					\def\pointcolor{googleGreen}
				\fi
			\fi

			\filldraw[
				fill=\pointcolor,
				draw=white,
				line width=0.5pt
			] (C\k) circle[radius=2.5pt];

			\ifnum\k=0\relax
				\node[vertex label] at (\pointangle:\labelradius) {$0$};
			\else
				\node[vertex label] at (\pointangle:\labelradius)
					{$\frac{\k}{13}$};
			\fi
		}
	\end{scope}

	\begin{scope}[shift={(0.1,-4.7)}]
		\draw[googleRed,line width=5pt] (0,0) -- (0.35,0);
		\node[right,font=\small] at (0.5,0)
			{$[0,6/13)$};

		\draw[googleBlue,line width=5pt] (3.2,0) -- (3.55,0);
		\node[right,font=\small] at (3.7,0)
			{$[6/13,12/13)$};

		\draw[googleGreen,line width=5pt] (6.9,0) -- (7.25,0);
		\node[right,font=\small] at (7.4,0)
			{$[12/13,1)$};
	\end{scope}
\end{tikzpicture}

%% file: tikz-code/character-rounding.tex
\begin{tikzpicture}[
    line cap=round,
    line join=round,
    font=\small,
    rounding tick/.style={draw=black!45,line width=0.45pt},
    graph edge/.style={draw=black,line width=0.55pt}
]
    \definecolor{googleRed}{HTML}{EA4335}
    \definecolor{googleBlue}{HTML}{4285F4}
    \definecolor{googleGreen}{HTML}{34A853}
    \definecolor{googleYellow}{HTML}{FBBC05}

    \def\circleradius{1.40}
    \def\targetx{7.20}

    \foreach \centerx in {0,\targetx} {
        \begin{scope}[shift={(\centerx,0)}]
            \foreach \start/\stop/\bandcolor in {
                0/2/googleRed,2/4/googleBlue,
                4/6/googleGreen,6/7/googleYellow%
            } {
                \pgfmathsetmacro{\arcstart}{90+360*\start/7}
                \pgfmathsetmacro{\arcend}{90+360*\stop/7}
                \draw[\bandcolor!35,line width=8pt,line cap=butt]
                    (\arcstart:\circleradius)
                    arc[start angle=\arcstart,end angle=\arcend,
                        radius=\circleradius];
            }
        \end{scope}
    }
    \foreach \j in {0,...,6} {
        \pgfmathsetmacro{\cutangle}{90+360*\j/7}
        \draw[rounding tick] (\cutangle:1.31) -- (\cutangle:1.49);
    }
    \node at (0,2.26) {$K_{10/3}$};

    \foreach \j in {0,...,9} {
        \pgfmathsetmacro{\pointangle}{90+360*\j/10}
        \coordinate (source\j) at (\pointangle:\circleradius);
    }
    \foreach \i in {0,...,9} {
        \foreach \j in {0,...,9} {
            \ifnum\i<\j\relax
                \pgfmathtruncatemacro{\difference}{\j-\i}
                \ifnum\difference>2\relax
                    \ifnum\difference<8\relax
                        \draw[graph edge] (source\i) -- (source\j);
                    \fi
                \fi
            \fi
        }
    }
    \foreach \j in {0,...,9} {
        \pgfmathtruncatemacro{\residue}{floor(7*\j/10)}
        \pgfmathtruncatemacro{\colorindex}{floor(\residue/2)}
        \ifcase\colorindex\relax
            \def\pointcolor{googleRed}
        \or
            \def\pointcolor{googleBlue}
        \or
            \def\pointcolor{googleGreen}
        \or
            \def\pointcolor{googleYellow}
        \fi
        \filldraw[fill=\pointcolor,draw=white,line width=0.55pt]
            (source\j) circle[radius=3.2pt];
        \pgfmathsetmacro{\pointangle}{90+360*\j/10}
        \node[font=\scriptsize,text=black!80]
            at (\pointangle:1.82) {$\j/10$};
    }
    \node[font=\footnotesize] at (0,-2.20) {separation $3/10$};

    \pgfmathsetmacro{\arrowend}{\targetx-2.15}
    \pgfmathsetmacro{\arrowmid}{(2.18+\arrowend)/2}
    \draw[->,black!70,line width=0.65pt]
        (2.18,0) -- (\arrowend,0)
        node[midway,above=6pt] {$t\mapsto\lfloor7t\rfloor$};
    \node[font=\footnotesize,text=black!65] at (\arrowmid,-0.40)
        {$3/10\geq2/7$};

    \begin{scope}[shift={(\arrowmid,-1.03)}]
        \foreach \row/\bandcolor/\interval in {
            0/googleRed/{[0,2/7)},
            1/googleBlue/{[2/7,4/7)},
            2/googleGreen/{[4/7,6/7)},
            3/googleYellow/{[6/7,1)}%
        } {
            \pgfmathsetmacro{\legendy}{-0.39*\row}
            \draw[\bandcolor,line width=5pt]
                (-0.93,\legendy) -- (-0.65,\legendy);
            \node[anchor=west,font=\footnotesize] at (-0.49,\legendy)
                {$\interval$};
        }
    \end{scope}

    \begin{scope}[shift={(\targetx,0)}]
        \foreach \j in {0,...,6} {
            \pgfmathsetmacro{\vertexangle}{90+360*\j/7}
            \coordinate (target\j) at (\vertexangle:\circleradius);
        }
        \foreach \j in {0,...,6} {
            \foreach \step in {2,3} {
                \pgfmathtruncatemacro{\neighbor}{mod(\j+\step,7)}
                \draw[graph edge] (target\j) -- (target\neighbor);
            }
        }
        \foreach \j in {0,...,6} {
            \pgfmathtruncatemacro{\colorindex}{floor(\j/2)}
            \ifcase\colorindex\relax
                \def\vertexcolor{googleRed}
            \or
                \def\vertexcolor{googleBlue}
            \or
                \def\vertexcolor{googleGreen}
            \or
                \def\vertexcolor{googleYellow}
            \fi
            \filldraw[fill=\vertexcolor,draw=white,line width=0.55pt]
                (target\j) circle[radius=3.2pt];
            \pgfmathsetmacro{\vertexangle}{90+360*\j/7}
            \node at (\vertexangle:1.68) {$\j$};
        }
        \node at (0,2.26) {$K_{7/2}$};
        \node[font=\footnotesize] at (0,-2.20) {separation $2/7$};
    \end{scope}
\end{tikzpicture}

%% file: tikz-code/eisenstein-quotient-background.tex
\begin{tikzpicture}[
	scale=1.0,
	line cap=round,
	line join=round,
	vertex label/.style={
		font=\small,
		fill=white,
		inner sep=1.5pt
	}
]
	\definecolor{googleRed}{HTML}{EA4335}
	\definecolor{googleBlue}{HTML}{4285F4}
	\definecolor{googleGreen}{HTML}{34A853}
	\definecolor{googleYellow}{HTML}{FBBC05}

	\pgfmathsetmacro{\height}{sqrt(3)/2}

	\foreach \a in {-1,...,3} {
		\foreach \b in {-1,...,2} {
			\coordinate (E\a-\b) at
				({\a-\b/2},{\b*\height});
		}
	}

	\foreach \a in {-1,...,3} {
		\foreach \b in {-1,...,2} {
			\ifnum\a<3\relax
				\pgfmathtruncatemacro{\nexta}{\a+1}
				\draw[gray!50,thin]
					(E\a-\b) -- (E\nexta-\b);
			\fi
			\ifnum\b<2\relax
				\pgfmathtruncatemacro{\nextb}{\b+1}
				\draw[gray!50,thin]
					(E\a-\b) -- (E\a-\nextb);
				\ifnum\a<3\relax
					\pgfmathtruncatemacro{\nexta}{\a+1}
					\draw[gray!50,thin]
						(E\a-\b) -- (E\nexta-\nextb);
				\fi
			\fi
		}
	}

	\draw[very thick] (E0-0) -- (E1-1);
	\draw[very thick] (E0-1) -- (E1-1);
	\draw[very thick] (E0-0) -- (E0-1);
	\draw[very thick] (E0-0) -- (E1-0);
	\draw[very thick] (E1-0) -- (E1-1);
	\draw[very thick,dashed] (E1-0) -- (E2-1);

	\foreach \a in {-1,...,3} {
		\foreach \b in {-1,...,2} {
			\pgfmathtruncatemacro{\residue}{mod(\a+\b+12,3)}
			\ifcase\residue\relax
				\def\vertexcolor{googleRed}
			\or
				\def\vertexcolor{googleBlue}
			\or
				\def\vertexcolor{googleGreen}
			\fi
			\filldraw[
				fill=\vertexcolor,
				draw=white,
				line width=0.5pt
			] (E\a-\b) circle[radius=3pt];
		}
	}

	\node[vertex label,below left=4pt] at (E0-0) {$0$};
	\node[vertex label,below=5pt] at (E1-0) {$1$};
	\node[vertex label,below right=4pt] at (E2-0) {$2$};
	\node[vertex label,above left=4pt] at (E0-1) {$\omega$};
	\node[vertex label,above=4pt] at (E1-1) {$1+\omega$};

	\node[vertex label,above right=4pt] at (E2-1) {$2+\omega$};

	\draw[
		->,
		thick,
		shorten <=5pt,
		shorten >=5pt
	] (E2-1) to[bend right=70]
		node[
			midway,
			above=4pt,
			font=\scriptsize,
			fill=white,
			inner sep=2pt
		] {identified modulo $2$}
		(E0-1);

	\node[font=\small] at (0.75,2.3)
		{$\Z[\omega]$};
	\node[font=\small] at (0.75,-1.5)
		{$c(a+b\omega)=a+b\pmod3$};

	\draw[->,thick] (4.15,0.4) -- (5.6,0.4)
		node[midway,above=5pt,font=\small]
			{modulo $2\Z[\omega]$};

	\begin{scope}[shift={(7.5,0.4)}]
		\coordinate (Kzero) at (-1,-1);
		\coordinate (Kone) at (1,-1);
		\coordinate (Komega) at (-1,1);
		\coordinate (Ksum) at (1,1);

		\draw[thick]
			(Kzero) -- (Kone) -- (Ksum) -- (Komega) -- cycle;

		\draw[very thick] (Kzero) -- (Ksum);
		\draw[very thick,dashed] (Kone) -- (Komega);

		\filldraw[
			fill=googleRed,
			draw=white,
			line width=0.5pt
		] (Kzero) circle[radius=3.5pt];

		\filldraw[
			fill=googleBlue,
			draw=white,
			line width=0.5pt
		] (Kone) circle[radius=3.5pt];

		\filldraw[
			fill=googleGreen,
			draw=white,
			line width=0.5pt
		] (Komega) circle[radius=3.5pt];

		\filldraw[
			fill=googleYellow,
			draw=white,
			line width=0.5pt
		] (Ksum) circle[radius=3.5pt];

		\node[vertex label,below left=4pt] at (Kzero) {$0$};
		\node[vertex label,below right=4pt] at (Kone) {$1$};
		\node[vertex label,above left=4pt] at (Komega) {$\omega$};
		\node[vertex label,above right=4pt] at (Ksum) {$1+\omega$};

		\node[font=\small] at (0,1.9)
			{$\Z[\omega]/2\Z[\omega]\cong\mathbb F_4$};
		\node[font=\small] at (0,-1.9)
			{Quotient graph: $K_4$};
	\end{scope}
\end{tikzpicture}

%% file: tikz-code/eisenstein-seven.tex
\begin{tikzpicture}[
    scale=0.90,
    line cap=round,
    line join=round,
    font=\small,
    lattice edge/.style={draw=black!42,line width=0.45pt}
]
    \definecolor{googleRed}{HTML}{EA4335}
    \definecolor{googleBlue}{HTML}{4285F4}
    \definecolor{googleGreen}{HTML}{34A853}
    \definecolor{googleYellow}{HTML}{FBBC05}

    \def\q{7}
    \pgfmathtruncatemacro{\N}{\q-1}

    \begin{scope}[shift={(1.98cm,0cm)},x={(0.66cm,0cm)},y={(-0.33cm,0.5715768cm)}]
        \foreach \i in {0,...,\N} {
            \foreach \j in {0,...,\N} {
                \coordinate (V\i-\j) at (\i,\j);
            }
        }

        \foreach \i in {0,...,\N} {
            \foreach \j in {0,...,\N} {
                \ifnum\i<\N\relax
                    \pgfmathtruncatemacro{\nexti}{\i+1}
                    \draw[lattice edge] (V\i-\j) -- (V\nexti-\j);
                \fi
                \ifnum\j<\N\relax
                    \pgfmathtruncatemacro{\nextj}{\j+1}
                    \draw[lattice edge] (V\i-\j) -- (V\i-\nextj);
                    \ifnum\i<\N\relax
                        \pgfmathtruncatemacro{\nexti}{\i+1}
                        \draw[lattice edge] (V\i-\j) -- (V\nexti-\nextj);
                    \fi
                \fi
            }
        }

        \foreach \i in {0,...,\N} {
            \foreach \j in {0,...,\N} {
                \pgfmathtruncatemacro{\residue}{mod(\i+\j,7)}
                \pgfmathtruncatemacro{\colorindex}{ifthenelse(\residue==6,3,mod(\residue,3))}
                \ifcase\colorindex\relax
                    \def\vertexcolor{googleRed}
                \or
                    \def\vertexcolor{googleBlue}
                \or
                    \def\vertexcolor{googleGreen}
                \or
                    \def\vertexcolor{googleYellow}
                \fi
                \filldraw[fill=\vertexcolor,draw=white,line width=0.55pt]
                    (V\i-\j) circle[radius=3.4pt];
            }
        }
    \end{scope}

    \node at (2.97,-0.48) {$E_7$};

    \draw[->,line width=0.65pt] (6.30,1.71473) -- (8.90,1.71473)
        node[midway,above=5pt] {$\sigma_7$}
        node[midway,below=5pt,font=\scriptsize] {$(i,j)\mapsto i+j\pmod7$};

    \begin{scope}[shift={(11.70,1.71473)}]
        \foreach \s in {0,...,6} {
            \pgfmathsetmacro{\angle}{90+360*\s/7}
            \coordinate (C\s) at (\angle:1.65);
        }

        \foreach \s in {0,...,6} {
            \foreach \step in {1,2} {
                \pgfmathtruncatemacro{\target}{mod(\s+\step,7)}
                \draw[lattice edge] (C\s) -- (C\target);
            }
        }

        \foreach \s/\vertexcolor in {
            0/googleRed,1/googleBlue,2/googleGreen,
            3/googleRed,4/googleBlue,5/googleGreen,6/googleYellow%
        } {
            \filldraw[fill=\vertexcolor,draw=white,line width=0.55pt]
                (C\s) circle[radius=4.2pt];
            \pgfmathsetmacro{\angle}{90+360*\s/7}
            \node at (\angle:1.96) {$\s$};
        }
    \end{scope}
    \node at (11.70,-0.48) {$C_7(1,2)$};

\end{tikzpicture}

%% file: tikz-code/eisenstein-five-background.tex
\begin{tikzpicture}[
    scale=1.0,
    line cap=round,
    line join=round,
    font=\small,
    lattice edge/.style={draw=black!42,line width=0.45pt},
    wrap edge/.style={
        draw=black!25,
        line width=0.40pt
    }
]
    \definecolor{googleRed}{HTML}{EA4335}
    \definecolor{googleBlue}{HTML}{4285F4}
    \definecolor{googleGreen}{HTML}{34A853}
    \definecolor{googleYellow}{HTML}{FBBC05}

    \def\q{5}
    \pgfmathtruncatemacro{\N}{\q-1}
    \pgfmathtruncatemacro{\lastinterior}{\N-1}

    \begin{scope}[x={(1cm,0cm)},y={(-0.5cm,0.8660254cm)}]
        \foreach \i in {0,...,\N} {
            \foreach \j in {0,...,\N} {
                \coordinate (V\i-\j) at (\i,\j);
            }
        }

        \foreach \j in {0,...,\N} {
            \draw[wrap edge]
                (V\N-\j) to[bend left=14] (V0-\j);
        }
        \foreach \i in {0,...,\N} {
            \draw[wrap edge]
                (V\i-\N) to[bend right=14] (V\i-0);
        }

        \foreach \j in {0,...,\lastinterior} {
            \pgfmathtruncatemacro{\nextj}{\j+1}
            \draw[wrap edge] (V\N-\j) -- (V0-\nextj);
        }
        \foreach \i in {0,...,\lastinterior} {
            \pgfmathtruncatemacro{\nexti}{\i+1}
            \draw[wrap edge] (V\i-\N) -- (V\nexti-0);
        }

        \def\wrapoffset{0.20}
        \draw[wrap edge]
            (V\N-\N)
            -- ({\N+\wrapoffset/2},{\N-\wrapoffset/2})
            -- ({\wrapoffset/2},{-\wrapoffset/2})
            -- (V0-0);

        \foreach \i in {0,...,\N} {
            \foreach \j in {0,...,\N} {
                \ifnum\i<\N\relax
                    \pgfmathtruncatemacro{\nexti}{\i+1}
                    \draw[lattice edge] (V\i-\j) -- (V\nexti-\j);
                \fi
                \ifnum\j<\N\relax
                    \pgfmathtruncatemacro{\nextj}{\j+1}
                    \draw[lattice edge] (V\i-\j) -- (V\i-\nextj);
                    \ifnum\i<\N\relax
                        \pgfmathtruncatemacro{\nexti}{\i+1}
                        \draw[lattice edge] (V\i-\j) -- (V\nexti-\nextj);
                    \fi
                \fi
            }
        }

        \foreach \i/\j/\colorindex in {
            0/0/0, 1/0/1, 2/0/2, 3/0/3, 4/0/2,
            0/1/1, 1/1/2, 2/1/0, 3/1/1, 4/1/0,
            0/2/2, 1/2/0, 2/2/1, 3/2/2, 4/2/3,
            0/3/0, 1/3/1, 2/3/2, 3/3/0, 4/3/1,
            0/4/2, 1/4/3, 2/4/0, 3/4/1, 4/4/3
        } {
                \ifcase\colorindex\relax
                    \def\vertexcolor{googleRed}
                \or
                    \def\vertexcolor{googleBlue}
                \or
                    \def\vertexcolor{googleGreen}
                \or
                    \def\vertexcolor{googleYellow}
                \fi
                \filldraw[fill=\vertexcolor,draw=white,line width=0.55pt]
                    (V\i-\j) circle[radius=3.4pt];
        }
    \end{scope}

    \node[anchor=west,inner sep=0pt,font=\small] at (4.7,1.75) {
        \begin{tabular}{@{}c|ccccc@{}}
            $j\backslash i$ & $0$ & $1$ & $2$ & $3$ & $4$ \\ \hline
            $4$ & 2 & 3 & 0 & 1 & 3 \\
            $3$ & 0 & 1 & 2 & 0 & 1 \\
            $2$ & 2 & 0 & 1 & 2 & 3 \\
            $1$ & 1 & 2 & 0 & 1 & 0 \\
            $0$ & 0 & 1 & 2 & 3 & 2 \\
        \end{tabular}
    };
\end{tikzpicture}

%% file: tikz-code/ninth-roots-eisenstein.tex
\begin{tikzpicture}[
    line cap=round,
    line join=round,
    font=\small,
    step edge/.style={->,>=stealth,line width=0.9pt,shorten <=2.5pt,shorten >=2.5pt}
]
    \definecolor{googleRed}{HTML}{EA4335}
    \definecolor{googleBlue}{HTML}{4285F4}
    \definecolor{googleGreen}{HTML}{34A853}

    \def\unitlength{1.8}
    \def\labeloffset{0.24}
    \foreach \j/\shiftx/\firststep/\secondstep/\thirdstep in {
        0/0/1/\zeta^3/\zeta^6,
        1/3.1/\zeta/\zeta^4/\zeta^7,
        2/6.5/\zeta^2/\zeta^5/\zeta^8%
    } {
        \begin{scope}[shift={(\shiftx,0)}]
            \pgfmathsetmacro{\angle}{40*\j}
            \coordinate (O\j) at (0,0);
            \coordinate (A\j) at (\angle:\unitlength);
            \coordinate (B\j) at ({\angle+60}:\unitlength);

            \draw[step edge,googleRed] (O\j) -- (A\j);
            \draw[step edge,googleBlue] (A\j) -- (B\j);
            \draw[step edge,googleGreen] (B\j) -- (O\j);

            \node at
                ({0.5*\unitlength*cos(\angle)+\labeloffset*sin(\angle)},
                 {0.5*\unitlength*sin(\angle)-\labeloffset*cos(\angle)})
                {$\firststep$};
            \node at
                ({0.5*\unitlength*(cos(\angle)+cos(\angle+60))+\labeloffset*cos(\angle+30)},
                 {0.5*\unitlength*(sin(\angle)+sin(\angle+60))+\labeloffset*sin(\angle+30)})
                {$\secondstep$};
            \node at
                ({0.5*\unitlength*cos(\angle+60)+\labeloffset*cos(\angle+150)},
                 {0.5*\unitlength*sin(\angle+60)+\labeloffset*sin(\angle+150)})
                {$\thirdstep$};

            \foreach \vertex in {O,A,B} {
                \fill[black!70] (\vertex\j) circle[radius=2pt];
            }
            \node[below left,inner sep=3pt] at (O\j) {$0$};
            \pgfmathsetmacro{\labelx}{0.5*\unitlength*(cos(\angle)+cos(\angle+60))}
            \ifnum\j=0\relax
                \def\labelx{0.9}
            \fi
            \node at (\labelx,-0.62) {$j=\j$};
        \end{scope}
    }

    \draw[->,>=stealth,line width=0.65pt] (7.55,0.85) -- (8.65,0.85)
        node[midway,above=5pt] {$L_9$};

    \begin{scope}[shift={(9.3,0)}]
        \coordinate (T0) at (0,0);
        \coordinate (T1) at (\unitlength,0);
        \coordinate (T2) at (60:\unitlength);

        \draw[step edge,googleRed] (T0) -- (T1)
            node[midway,below=5pt] {$1$};
        \draw[step edge,googleBlue] (T1) -- (T2)
            node[midway,above right=3pt] {$\omega$};
        \draw[step edge,googleGreen] (T2) -- (T0)
            node[midway,above left=3pt] {$\omega^2$};

        \foreach \vertex in {T0,T1,T2} {
            \fill[black!70] (\vertex) circle[radius=2pt];
        }
        \node[below left,inner sep=3pt] at (T0) {$0$};
        \node[below right,inner sep=3pt] at (T1) {$1$};
        \node[above,inner sep=3pt] at (T2) {$1+\omega$};
    \end{scope}
\end{tikzpicture}

%% file: tikz-code/eighth-root-coloring.tex
\begin{tikzpicture}[
    scale=0.94,
    line cap=round,
    line join=round,
    font=\small,
    quotient edge/.style={draw=black!45,line width=0.35pt},
    target edge/.style={draw=black!75,line width=0.60pt},
    quotient vertex/.style={
        circle,draw=white,line width=0.55pt,
        minimum size=10pt,inner sep=0pt,
        font=\scriptsize,text=black
    }
]
    \definecolor{googleRed}{HTML}{EA4335}
    \definecolor{googleBlue}{HTML}{4285F4}
    \definecolor{googleGreen}{HTML}{34A853}
    \definecolor{googleYellow}{HTML}{FBBC05}

    \begin{scope}[x=0.62cm,y=0.62cm]
        \foreach \a in {0,...,6} {
            \foreach \b in {0,...,6} {
                \coordinate (Q\a-\b) at (\a,\b);
            }
        }
        \foreach \a in {0,...,6} {
            \foreach \b in {0,...,6} {
                \foreach \da/\db in {1/0,0/1,1/3,3/1} {
                    \pgfmathtruncatemacro{\nexta}{\a+\da}
                    \pgfmathtruncatemacro{\nextb}{\b+\db}
                    \ifnum\nexta<7\relax
                        \ifnum\nextb<7\relax
                            \draw[quotient edge]
                                (Q\a-\b) -- (Q\nexta-\nextb);
                        \fi
                    \fi
                }
            }
        }

        \foreach \a in {0,...,6} {
            \foreach \b in {0,...,6} {
                \pgfmathtruncatemacro{\residue}{mod(5*\a+4*\b,7)}
                \pgfmathtruncatemacro{\colorindex}{floor(\residue/2)}
                \ifcase\colorindex\relax
                    \def\vertexcolor{googleRed}
                \or
                    \def\vertexcolor{googleBlue}
                \or
                    \def\vertexcolor{googleGreen}
                \or
                    \def\vertexcolor{googleYellow}
                \fi
                \node[quotient vertex,fill=\vertexcolor]
                    at (Q\a-\b) {$\residue$};
            }
        }
    \end{scope}

    \draw[->,black!60,line width=0.4pt]
        (-0.10,-0.36) -- (0.78,-0.36)
        node[right,font=\scriptsize] {$a$};
    \draw[->,black!60,line width=0.4pt]
        (-0.36,-0.10) -- (-0.36,0.78)
        node[above,font=\scriptsize] {$b$};
    \node at (1.86,-0.78) {$\Gamma_8(f)$};

    \draw[->,line width=0.7pt] (4.45,1.86) -- (5.55,1.86)
        node[midway,above=5pt] {$h$};

    \begin{scope}[shift={(8.10,1.86)}]
        \def\circleradius{1.65}
        \def\labelradius{2.00}

        \foreach \first/\last/\intervalcolor in {
            0/2/googleRed,2/4/googleBlue,
            4/6/googleGreen,6/7/googleYellow%
        } {
            \pgfmathsetmacro{\startangle}{90+360*\first/7}
            \pgfmathsetmacro{\endangle}{90+360*\last/7}
            \draw[\intervalcolor!35,line width=8pt,line cap=butt]
                (\startangle:\circleradius)
                arc[start angle=\startangle,end angle=\endangle,
                    radius=\circleradius];
        }

        \foreach \s in {0,...,6} {
            \pgfmathsetmacro{\angle}{90+360*\s/7}
            \coordinate (C\s) at (\angle:\circleradius);
        }

        \foreach \s in {0,...,6} {
            \foreach \step in {2,3} {
                \pgfmathtruncatemacro{\target}{mod(\s+\step,7)}
                \draw[target edge] (C\s) -- (C\target);
            }
        }

        \foreach \s/\vertexcolor in {
            0/googleRed,1/googleRed,2/googleBlue,3/googleBlue,
            4/googleGreen,5/googleGreen,6/googleYellow%
        } {
            \filldraw[fill=\vertexcolor,draw=white,line width=0.55pt]
                (C\s) circle[radius=4.5pt];
            \pgfmathsetmacro{\angle}{90+360*\s/7}
            \node at (\angle:\labelradius) {$\s$};
        }
    \end{scope}
    \node at (8.10,-0.78) {$K_{7/2}$};
\end{tikzpicture}

%% file: 3-cyclotomic.tex

\section{Transcendental and Cyclotomic Parameters}
\label{sec:transcendental-cyclotomic}

The constructions of Section~2 give exact ordinary and circular chromatic numbers for two classes of direction parameters: transcendental numbers and roots of unity.
We begin with the transcendental case, then develop the cyclotomic classification through a twentieth-root example.
The classification leads to a geometric restriction on the edge directions of any four-chromatic unit-distance graph.

\subsection{Transcendental Direction Parameters}
\label{subsec:transcendental}

For a transcendental parameter, the powers of $\alpha$ act as independent coordinate directions.
This gives the familiar parity coloring of an integer lattice.

\begin{theorem}
    \label{thm:transcendental}
    Let $\alpha\in\C$ be transcendental with $|\alpha|=1$.
    Then
    \[\chi(\Gamma_\alpha)=\chi_c(\Gamma_\alpha)=2.\]
\end{theorem}

\begin{proof}
    Let $\Z^{(\Z_{\geq0})}$ denote the group of finitely supported integer sequences, with standard basis vectors $e_n$.
    Evaluation sends a sequence $(b_n)_{n\geq0}$ to $\sum_{n\geq0}b_n\alpha^n$.
    This defines an additive map $\pi:\Z^{(\Z_{\geq0})}\to\Z[\alpha]$, which is surjective by definition and injective because $\alpha$ satisfies no nonzero integer polynomial relation.
    Since $\pi(e_n)=\alpha^n$, it gives an isomorphism of Cayley graphs
    \[\Cay\left(\Z^{(\Z_{\geq0})},\{\pm e_n:n\geq0\}\right)
    \cong\Gamma_\alpha.\]
    Thus $\Gamma_\alpha$ is the integer lattice graph in countably many coordinate directions, with each vertex having only finitely many nonzero coordinates.

    Each allowed step changes the sum of the coordinates by one, so
    \[c\left(\sum_{n\geq0}b_n\alpha^n\right)
    =\sum_{n\geq0}b_n\pmod2\]
    defines a graph homomorphism to $K_2$.
    Composing with $j\mapsto j/2\pmod1$ gives a character whose separation on every allowed step is $1/2$.
    Since $\Gamma_\alpha$ contains the edge $\{0,1\}$, both colorings are optimal.
\end{proof}

\begin{remark}[Adjoining a transcendental direction]
    Suppose $\alpha$ and $\beta$ are unit complex numbers and $\beta$ is transcendental over $\Q(\alpha)$.
    Every element of $\Z[\alpha,\beta]$ has a unique finite expansion $\sum_n c_n\beta^n$ with $c_n\in\Z[\alpha]$.
    Evaluation at $\beta=1$ defines an additive map $r:\Z[\alpha,\beta]\to\Z[\alpha]$ by summing these coefficients.
    It sends each allowed step $\pm\alpha^m\beta^n$ to $\pm\alpha^m$ and fixes $\Z[\alpha]$, giving
    \[\Gamma_\alpha\hookrightarrow\MonCay(\alpha,\beta)\xrightarrow{r}\Gamma_\alpha.\]
    Pulling back colorings along these maps shows that adjoining $\beta$ preserves both chromatic numbers.
    In particular, this applies whenever $\alpha$ is algebraic and $\beta$ is transcendental.

    The graph $\MonCay(U_{\mathrm{alg}}\cup\{\beta\})$ likewise retracts onto $\Gamma_{\mathrm{alg}}$ whenever $\beta$ is transcendental.
    Every element of $\Z[U_{\mathrm{alg}}]$ is algebraic, so evaluation at $\beta=1$ is well-defined.
    Both chromatic numbers are again preserved; the ordinary chromatic number already equals that of the plane.
\end{remark}

\subsection{Cyclotomic Parameters}
\label{subsec:cyclotomic}

Consider $\alpha=\zeta_{20}=e^{2\pi i/20}$ and its fifth-root power $\xi=\alpha^4$.
The relation $1+\xi+\xi^2+\xi^3+\xi^4=0$ closes a walk of five allowed unit steps.
For $0\leq j\leq5$, put
\[v_j=\sum_{k=0}^{j-1}\xi^k,\]
where the empty sum is zero.
Then $v_5=v_0=0$, and the five distinct vertices $v_0,\ldots,v_4$ form a regular pentagon in $\Gamma_\alpha$.
This odd cycle requires three colors.
Moreover, an independent set in the pentagon contains at most two vertices, so the arc-averaging argument from Subsection~\ref{subsec:three-coloring} gives a circular lower bound of $5/2$.
We will construct a coloring that attains both bounds.

The factorization $20=4\cdot5$ suggests combining labels for the fourth and fifth roots.
Set $i=\alpha^5$; since $\alpha=i\xi^4$, we have $\Z[\alpha]=\Z[i,\xi]$.
For the order-four factor, assign the labels
\[\bigl(L_4(1),L_4(i),L_4(-1),L_4(-i)\bigr)=(1,1,-1,-1)\]
in $\Z/5\Z$.
These labels respect the relations $1+(-1)=0$ and $i+(-i)=0$.
For the order-five factor, set $L_5(\xi^j)=1$ for $0\leq j\leq4$.
The fifth-root relation is respected because the five labels sum to zero modulo five.

Lemma~\ref{lem:coprime-labels} combines these assignments into an additive map $L:\Z[\alpha]\to\Z/5\Z$ satisfying $L(uv)=L_4(u)L_5(v)$ for $u\in\Z[i]$ and $v\in\Z[\xi]$.
In particular,
\[L(\alpha^j)=L_4(i^j)L_5(\xi^{4j})=L_4(i^j),\]
so the labels repeat with period four and every allowed step maps to $1$ or $-1$.
Thus $L$ maps $\Gamma_\alpha$ to the cycle $C_5=\Cay(\Z/5\Z,\{\pm1\})$.
Multiplication by two modulo five moves these step images as far from zero as possible.
The normalized multiplier gives the character $\psi=2L/5$ with step values
\[
\begin{array}{|c|c|c|c|c|}
\hline
j\bmod4&0&1&2&3\\
\hline
L(\alpha^j)&1&1&-1&-1\\
\hline
\psi(\alpha^j)&2/5&2/5&3/5&3/5\\
\hline
\end{array}
\]
Every allowed step has circular separation $2/5$, so $\psi$ gives a circular $5/2$-coloring of the entire graph.

This construction also gives a direct recipe for coloring any vertex.
Reduce its polynomial expression modulo $\Phi_{20}(x)=x^8-x^6+x^4-x^2+1$ to obtain $\sum_{j=0}^7a_j\alpha^j$, and compute
\[L\left(\sum_{j=0}^7a_j\alpha^j\right)
=a_0+a_1-a_2-a_3+a_4+a_5-a_6-a_7\pmod5.\]
Assign red, blue, and green according to which of the half-open arcs $[0,1/5)$, $[1/5,3/5)$, and $[3/5,1)$ contains $\psi(z)$.
Equivalently, the three classes of $L$-values are $\{0\}$, $\{1,3\}$, and $\{2,4\}$, respectively.
Each arc has length at most $2/5$, so adjacent vertices receive different colors.
Together with the pentagon, these colorings establish $\chi(\Gamma_\alpha)=3$ and $\chi_c(\Gamma_\alpha)=5/2$.

Each step of the pentagon has $L$-image $1$, so its vertices satisfy $L(v_j)=j\pmod5$ and $\psi(v_j)=2j/5\pmod1$.
Figure~\ref{fig:twentieth-root-coloring} follows these vertices into the circular target, where consecutive vertices are separated by $2/5$.

\begin{figure}[htbp]
    \centering
    \input{tikz-code/twentieth-roots-cyclotomic.tex}
    \caption{The pentagon in $\Gamma_{\zeta_{20}}$ and its image under $\psi=2L/5$.
    Arrows on the left indicate the five unit steps $\xi^j$, where $\xi=\alpha^4$.
    The circular target has edge separation $2/5$, and its three colored arcs pull back to a coloring of the entire graph.
    The embedded pentagon makes both the three-color and circular $5/2$ bounds optimal.}
    \label{fig:twentieth-root-coloring}
\end{figure}
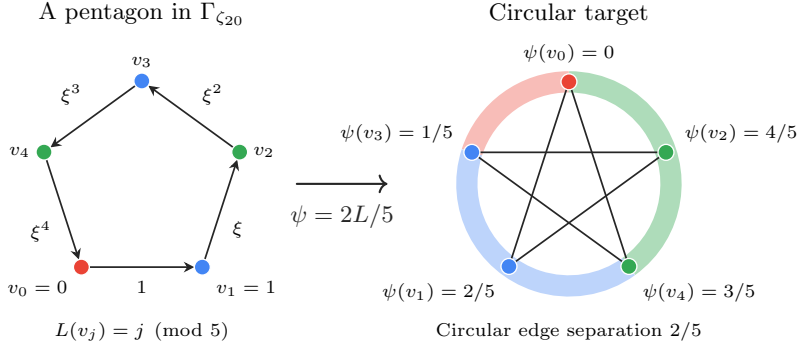

The example suggests the general construction.
Let $\alpha$ have order $n$ with smallest odd prime divisor $p$.
We seek graph homomorphisms
\[C_p\hookrightarrow\Gamma_\alpha\longrightarrow C_p.\]
The cycle inside $\Gamma_\alpha$ supplies lower bounds, while the map to the same cycle pulls back its colorings.
Consequently, the two graphs have the same ordinary and circular chromatic numbers.

For the circular coloring of $C_p=\Cay(\Z/p\Z,\{\pm1\})$, put $h=(p-1)/2$.
The multiplier $h$ gives the character $j\mapsto hj/p\pmod1$, with separation $h/p$ on every edge.
Conversely, an independent set in $C_p$ has at most $h$ vertices, so arc averaging gives $p/r\leq h$ for every circular $r$-coloring.
Thus $\chi_c(C_p)=2p/(p-1)$, while its ordinary chromatic number is three.

To map $\Gamma_\alpha$ to $C_p$, we again assign signs to the powers of $\alpha$, now modulo $p$.
The prime-power relations from Subsection~\ref{subsec:cyclotomic-quotients} explain the choice of the smallest odd prime: every odd relation block has length at least $p$, so it can receive $p$ positive labels followed by cancelling pairs.

\begin{lemma}
    \label{lem:cyclotomic-sign-labels}
    Let $n$ have smallest odd prime divisor $p$.
    There is an additive map $L:\mathcal O_n\to\Z/p\Z$ such that $L(\zeta_n^j)\in\{1,-1\}$ for every $j\geq0$.
\end{lemma}

\begin{proof}
    Consider a prime-power factor $\ell^k$ of $n$, put $M=\ell^{k-1}$, and let $\xi$ be a primitive $\ell^k$th root of unity.
    As in Lemma~\ref{lem:prime-power-labels}, the powers of $\xi$ occur in relation blocks
    \[\xi^a+\xi^{a+M}+\cdots+\xi^{a+(\ell-1)M}=0\]
    for $0\leq a<M$.
    Labels on the $\ell^k$ powers define an additive map on $\Z[\xi]$ precisely when their sum in each block is zero.
    If $\ell=2$, assign the labels $1,-1$.
    If $\ell$ is odd, assign $p$ labels equal to $1$, followed by $(\ell-p)/2$ pairs $1,-1$.
    Their sum is $p=0$ in $\Z/p\Z$.
    Since the powers repeat with period $\ell^k$, every power receives a label in $\{1,-1\}$.
    Lemma~\ref{lem:coprime-labels}, with target ring $\Z/p\Z$, combines these maps over the distinct prime-power factors of $n$.
    Each power of $\zeta_n$ receives a product of signs, which again belongs to $\{1,-1\}$.
\end{proof}

\begin{theorem}
    \label{thm:cyclotomic}
    Let $\alpha$ be a root of unity of order $n$.
    If $n=2^k$ for some $k\geq0$, then
    \[\chi(\Gamma_\alpha)=\chi_c(\Gamma_\alpha)=2.\]
    Otherwise, let $p$ be the smallest odd prime divisor of $n$.
    Then $\chi(\Gamma_\alpha)=3$ and
    \[\chi_c(\Gamma_\alpha)=\frac{2p}{p-1}.\]
\end{theorem}

\begin{proof}
    If $n=1$, then $\Phi_n(x)=x-1$.
    If $n=2^k$ with $k\geq1$, then $\Phi_n(x)=x^{2^{k-1}}+1$.
    In both cases $\Phi_n(1)$ is even, so Lemma~\ref{lem:bipartite-criterion} gives a two-coloring.
    The edge $\{0,1\}$ makes both chromatic numbers exactly two.

    Now suppose that $n$ has smallest odd prime divisor $p$.
    Apply Lemma~\ref{lem:cyclotomic-sign-labels} with $\zeta_n=\alpha$.
    Since every allowed step maps to $1$ or $-1$, this gives the required graph homomorphism $L:\Gamma_\alpha\to C_p$.

    To construct the embedded cycle, put $\beta=\alpha^{n/p}$ and define $v_j=\sum_{k=0}^{j-1}\beta^k$ for $0\leq j\leq p$, with the empty sum equal to zero.
    The relation $1+\beta+\cdots+\beta^{p-1}=0$ gives $v_p=v_0$, and each difference $v_{j+1}-v_j=\beta^j$ is an allowed step.
    Since $\beta$ has order $p$ and $v_j=(1-\beta^j)/(1-\beta)$, the vertices $v_0,\ldots,v_{p-1}$ are distinct.
    They form a copy of $C_p$ in $\Gamma_\alpha$, completing the two maps and hence the classification.
\end{proof}

The classification has a geometric consequence.
For a finite unit-distance graph whose edge directions make rational multiples of $\pi$ with one another, a rotation makes every edge vector a root of unity of a common order.
Following paths then places each connected component in a cyclotomic monomial Cayley graph.

\begin{corollary}
    \label{cor:rational-angle}
    Every graph admitting a unit-distance embedding in the plane in which the angle between any two edge directions is a rational multiple of $\pi$ is three-colorable.
\end{corollary}

\begin{proof}
    First let $H$ be a finite graph with such an embedding, identified with points in $\C$.
    If $H$ has no edges, the conclusion is immediate.
    Otherwise, fix an oriented edge vector $u$.
    Dividing by $u$ rotates this direction to $1$ and every other oriented edge vector $v$ to the root of unity $v/u$.
    There are finitely many such vectors, so their orders have a common multiple $n$ and all are powers of $\zeta_n$.

    In each connected component, choose a vertex $z_0$ and apply $z\mapsto(z-z_0)/u$.
    The image of any vertex is the sum of normalized edge vectors along a path from $z_0$, so it lies in $\Z[\zeta_n]$.
    Every edge difference is a signed power of $\zeta_n$.
    Thus each component embeds as a subgraph of $\Gamma_{\zeta_n}$, and Theorem~\ref{thm:cyclotomic} supplies a three-coloring.
    Using the same three colors on all components colors $H$.

    For an infinite graph, every finite subgraph inherits the rational-angle condition and is therefore three-colorable.
    The de Bruijn--Erd\H{o}s compactness theorem~\cite{deBruijnErdos1951} gives a three-coloring of the entire graph.
\end{proof}

In particular, every four-chromatic unit-distance graph, including the Moser spindle, must have two edge directions whose angle is an irrational multiple of $\pi$.
The cyclotomic classification therefore identifies a necessary geometric feature of the obstruction that motivated our program.

%% file: tikz-code/twentieth-roots-cyclotomic.tex
\begin{tikzpicture}[
    line cap=round,
    line join=round,
    font=\small,
    graph edge/.style={draw=black!85,line width=0.65pt},
    unit step/.style={graph edge,-{Stealth[length=4pt,width=4pt]},
        shorten >=3pt,shorten <=3pt},
    vertex label/.style={font=\scriptsize,inner sep=2pt},
    step label/.style={font=\scriptsize,inner sep=2pt,fill=white}
]
    \definecolor{googleRed}{HTML}{EA4335}
    \definecolor{googleBlue}{HTML}{4285F4}
    \definecolor{googleGreen}{HTML}{34A853}

    \pgfmathsetmacro{\pentagoncenter}{0.5/tan(36)}
    \coordinate (V0) at ({-0.8},{-1.6*\pentagoncenter});
    \coordinate (V1) at ({0.8},{-1.6*\pentagoncenter});
    \coordinate (V2) at ({1.6*(0.5+cos(72))},{1.6*(sin(72)-\pentagoncenter)});
    \coordinate (V3) at (0,{1.6*(sin(72)+sin(144)-\pentagoncenter)});
    \coordinate (V4) at ({-1.6*(0.5+cos(72))},{1.6*(sin(72)-\pentagoncenter)});

    \draw[unit step] (V0) -- (V1) node[midway,below=3pt,step label] {$1$};
    \draw[unit step] (V1) -- (V2) node[midway,below right=2pt,step label] {$\xi$};
    \draw[unit step] (V2) -- (V3) node[midway,above right=2pt,step label] {$\xi^2$};
    \draw[unit step] (V3) -- (V4) node[midway,above left=2pt,step label] {$\xi^3$};
    \draw[unit step] (V4) -- (V0) node[midway,below left=2pt,step label] {$\xi^4$};

    \foreach \j/\pointcolor in {0/googleRed,1/googleBlue,2/googleGreen,3/googleBlue,4/googleGreen} {
        \filldraw[fill=\pointcolor,draw=white,line width=0.5pt]
            (V\j) circle[radius=3pt];
    }
    \node[vertex label,below left=3pt] at (V0) {$v_0=0$};
    \node[vertex label,below right=3pt] at (V1) {$v_1=1$};
    \node[vertex label,right=3pt] at (V2) {$v_2$};
    \node[vertex label,above=3pt] at (V3) {$v_3$};
    \node[vertex label,left=3pt] at (V4) {$v_4$};
    \node at (0,2.25) {A pentagon in $\Gamma_{\zeta_{20}}$};
    \node[font=\scriptsize] at (0,-1.95) {$L(v_j)=j\pmod5$};

    \draw[->,thick,black!85] (2.05,0) -- (3.25,0)
        node[midway,below=4pt] {$\psi=2L/5$};

    \begin{scope}[shift={(5.65,0)}]
        \def\circleradius{1.35}
        \def\circlelabelradius{1.53}
        \node at (0,2.25) {Circular target};

        \draw[googleRed!35,line width=8pt,line cap=butt]
            (90:\circleradius)
            arc[start angle=90,end angle=162,radius=\circleradius];
        \draw[googleBlue!35,line width=8pt,line cap=butt]
            (162:\circleradius)
            arc[start angle=162,end angle=306,radius=\circleradius];
        \draw[googleGreen!40,line width=8pt,line cap=butt]
            (306:\circleradius)
            arc[start angle=306,end angle=450,radius=\circleradius];

        \foreach \k in {0,...,4} {
            \pgfmathsetmacro{\pointangle}{90+72*\k}
            \coordinate (C\k) at (\pointangle:\circleradius);
        }
        \foreach \k in {0,...,4} {
            \pgfmathtruncatemacro{\nextpoint}{mod(\k+2,5)}
            \draw[graph edge] (C\k) -- (C\nextpoint);
        }
        \foreach \k/\pointcolor in {0/googleRed,1/googleBlue,2/googleBlue,3/googleGreen,4/googleGreen} {
            \filldraw[fill=\pointcolor,draw=white,line width=0.5pt]
                (C\k) circle[radius=3pt];
        }

        \node[vertex label,anchor=south] at (90:\circlelabelradius) {$\psi(v_0)=0$};
        \node[vertex label,anchor=south east] at (162:\circlelabelradius) {$\psi(v_3)=1/5$};
        \node[vertex label,anchor=north east] at (234:\circlelabelradius) {$\psi(v_1)=2/5$};
        \node[vertex label,anchor=north west] at (306:\circlelabelradius) {$\psi(v_4)=3/5$};
        \node[vertex label,anchor=south west] at (378:\circlelabelradius) {$\psi(v_2)=4/5$};
        \node[font=\scriptsize] at (0,-1.95) {Circular edge separation $2/5$};
    \end{scope}
\end{tikzpicture}

%% file: 4-noncyclotomic.tex

\section{Noncyclotomic Algebraic Parameters}
\label{sec:noncyclotomic-algebraic}

For an algebraic parameter, a polynomial relation constrains the character value at each successive power.
When the relation is nonmonic, several choices remain; when it is monic, the next value is determined.
We use this distinction to study the parameters left after Section~\ref{sec:transcendental-cyclotomic}, beginning with a construction that makes the available choices visible.

Throughout this section, let $\alpha$ be a noncyclotomic algebraic direction parameter with primitive integer minimal polynomial $f(x)=\sum_{j=0}^{d}a_jx^j$, normalized by $a_d>0$.
Retain the trace notation $f(x)=x^rF(x+x^{-1})$ from Subsection~\ref{subsec:trace-quotients}, where $d=2r$ and $F$ has leading coefficient $a_d$.

\subsection{Algebraic Noninteger Parameters}
\label{sec:algebraic-nonintegers}
\label{subsec:noninteger-circular-bound}

For this subsection and the next, assume that $a_d\geq2$.
Let $\alpha=(1+i\sqrt{195})/14$, whose primitive integer minimal polynomial is
\[f(x)=7x^2-x+7.\]
We aim to keep the character values of its powers in the closed arc $I=[3/7,4/7]$, so that every allowed step has circular separation at least $3/7$.
Begin with $\psi(1)=\psi(\alpha)=1/2$.
The relation $7\alpha^2=\alpha-7$ requires, in $\T$,
\[7\psi(\alpha^2)=\psi(\alpha)-7\psi(1)=0.\]
The seven possible values are $j/7$ for $0\leq j\leq6$.
Both $3/7$ and $4/7$ belong to $I$; choose $\psi(\alpha^2)=3/7$.
The next relation, $7\alpha^3=\alpha^2-7\alpha$, requires
\[7\psi(\alpha^3)=\frac37-\frac72=\frac{13}{14}\pmod1.\]
Of its seven solutions, only $55/98$ belongs to $I$.
Figure~\ref{fig:semiminimal-extension} shows these two extensions.
Although the candidate values move around the circle, their spacing remains $1/7$.
The closed arc $I$ has exactly this length, so it contains a permissible choice at every stage.
The Shift Criterion makes the resulting assignments into a character with separation at least $3/7$, giving a circular $7/3$-coloring.

\begin{figure}[htbp]
    \centering
    \input{tikz-code/semiminimal-extension.tex}
    \caption{Successive choices for $\psi(\alpha^2)$ and $\psi(\alpha^3)$.
    The candidate values are equally spaced by $1/7$ in both panels, so the blue arc $[3/7,4/7]$ always contains a choice.
    Filled green points mark the selected values $3/7$ and $55/98$; the open green point is the other permissible first choice.}
    \label{fig:semiminimal-extension}
\end{figure}
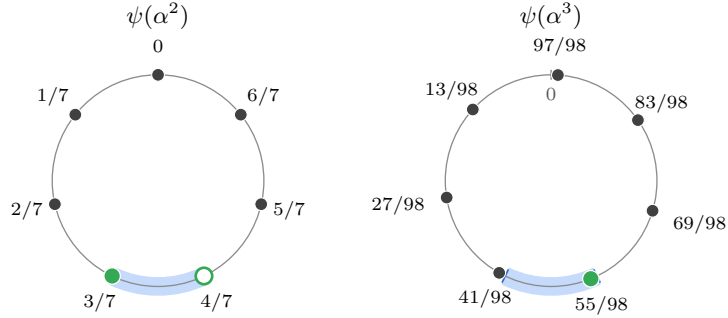

The opening example uses the character-extension mechanism of Garc\'ia-Marco, Knauer, and Menara~\cite[Lemma~2.2]{GarciaMarcoKnauerMenara2026}.
Their theorem for an infinite semiminimal family gives the following bound directly~\cite[Theorem~4.3]{GarciaMarcoKnauerMenara2026}.
Indeed, the first $d$ powers are linearly independent over $\Q$.
For each later power, the least positive integer multiplying it into the integer span of the earlier powers is $a_d$: the shifted minimal polynomial gives this multiple, and Gauss's lemma makes every such multiplier divisible by $a_d$.

\begin{theorem}[Garc\'ia-Marco--Knauer--Menara]
    \label{thm:algebraic-noninteger}
    Let $\alpha$ be an algebraic noninteger direction parameter whose primitive integer minimal polynomial has leading coefficient $a_d\geq2$.
    Then
    \[\chi_c(\Gamma_\alpha)\leq\frac{2a_d}{a_d-1}.\]
    Consequently, $\chi(\Gamma_\alpha)\leq\lceil2a_d/(a_d-1)\rceil$.
    Four colors suffice in general, and three suffice when $a_d\geq3$.
\end{theorem}

\begin{proof}
    Put $\delta=(a_d-1)/(2a_d)$ and take the closed arc
    \[I=\left[\frac12-\frac1{2a_d},\frac12+\frac1{2a_d}\right]\subseteq\T.\]
    Its length is $1/a_d$, and every point of $I$ has circular distance at least $\delta$ from zero.
    Choose $u_0=\cdots=u_{d-1}=1/2$.
    Once $u_0,\ldots,u_{n+d-1}$ have been chosen, choose $u_{n+d}\in I$ satisfying
    \[a_du_{n+d}=-\sum_{j=0}^{d-1}a_ju_{n+j}\quad\text{in }\T.\]
    This equation has $a_d$ equally spaced solutions, so $I$ contains at least one of them.
    The Shift Criterion (Proposition~\ref{prop:shift-criterion}) gives a character $\psi$ with $\psi(\alpha^n)=u_n\in I$ for every $n$.
    Its separation is at least $\delta$, and Lemma~\ref{lem:character-separation} gives the bounds.
\end{proof}

Rounding as in Subsection~\ref{subsec:characters-circulants} gives a finite circulant target; for the opening example, this is $K_{7/3}\cong C_7$.
When $a_d\geq3$, the bipartite criterion distinguishes the two possible ordinary chromatic numbers.

\begin{corollary}
    \label{cor:algebraic-partial-classification}
    If $a_d\geq3$, then
    \[\chi(\Gamma_\alpha)=
    \begin{cases}
        2,&f(1)\text{ is even},\\
        3,&f(1)\text{ is odd}.
    \end{cases}\]
\end{corollary}

The uniform construction leaves a gap only when the leading coefficient is two.
We conjecture that three colors still suffice.

\begin{conjecture}
    \label{conj:noninteger-three-color}
    If $\alpha$ is an algebraic noninteger direction parameter, then 
    \[\chi(\Gamma_\alpha)\leq3.\]
\end{conjecture}

\subsection{Leading Coefficient Two}
\label{subsec:leading-coefficient-two}

When $a_d=2$ and $f(1)$ is even, the graph is bipartite.
For odd $f(1)$, the extension argument gives separation $1/4$ and four colors, but evaluation can do better.
If any of $F(-2)$, $F(0)$, or $F(2)$ has absolute value greater than one, an elementary evaluation gives a three-coloring.
The eighth-root quotient can succeed even when all three tests fail.

Consider
\[f(x)=2x^6-2x^4+x^3-2x^2+2,\]
with trace polynomial $F(y)=2y(y^2-4)+1$.
The cubic $F$ has no rational root and changes sign between $0$ and $1$, giving a real root $y_0\in(0,1)$.
Then $\alpha=(y_0+i\sqrt{4-y_0^2})/2$ has modulus one and satisfies $f(\alpha)=0$.
The field $\Q(y_0)$ is real of degree three, and $\alpha$ is nonreal and quadratic over it, so $\alpha$ has degree six over $\Q$.
Thus $f$ is its primitive integer minimal polynomial.

The factored trace form gives $F(-2)=F(0)=F(2)=1$.
At the eighth-root evaluation, however,
\[F(y)\equiv1-4y\pmod{y^2-2}.\]
The norm is $D_8=|1-2\cdot16|=31$.
In the quotient coordinates from Subsection~\ref{subsec:trace-quotients}, we have $S_8\cong\Z/31\Z$ with $[\sqrt2]=8$.
Define $h:R_8(f)\to\Z/31\Z$ by
\[h\bigl(a+b[\zeta_8]\bigr)=19a+16b\pmod{31}.\]
The relation $[\zeta_8]^2=8[\zeta_8]-1$ shows that every signed root direction maps into $\{12,15,16,19\}$.
Thus
\[\Gamma_\alpha\longrightarrow\Gamma_8(f)\longrightarrow K_{31/12}.\]
After division by $31$, the character values all lie in the middle third of the circle, with separation at least $12/31$.
This gives $\chi_c(\Gamma_\alpha)\leq31/12<3$.
Since $f(1)=1$ excludes bipartiteness, $\chi(\Gamma_\alpha)=3$.

The three equal trace values in this example are forced whenever the elementary evaluations give no coloring bound.
The following observation will also be useful when we turn to algebraic integers.

\begin{lemma}
    \label{lem:trace-evaluation-signs}
    Let $F\in\Z[y]$.
    If $F(-2),F(0),F(2)\in\{1,-1\}$, then these three values are equal.
    If also $F(-1),F(1)\in\{1,-1\}$, then all five values are equal.
\end{lemma}

\begin{proof}
    Every integer polynomial satisfies $F(2)\equiv F(-2)\pmod4$, so the first and last signs must agree.
    We also have $F(2)+F(-2)\equiv2F(0)\pmod8$: the odd powers cancel, and each positive even power contributes a multiple of eight.
    Substituting $F(-2)=F(2)$ gives $F(0)\equiv F(2)\pmod4$, so the middle sign agrees as well.
    Finally, $F(1)\equiv F(-2)\pmod3$ and $F(-1)\equiv F(2)\pmod3$ force the two remaining signs to have the same value.
\end{proof}

Subtracting the common sign produces roots at $-2$, $0$, and $2$.
This factorization, together with the stronger quotient tests, restricts any counterexample to the noninteger conjecture.

\begin{proposition}
    \label{prop:noninteger-counterexample-restrictions}
    Let $\alpha$ be an algebraic noninteger direction parameter with $\chi(\Gamma_\alpha)>3$.
    Then $\chi(\Gamma_\alpha)=4$, and its primitive integer minimal polynomial has leading and constant coefficients two.
    Its trace polynomial has the form
    \[F(y)=y(y^2-4)Q(y)\pm1,\]
    where $Q\in\Z[y]$ is nonzero and has leading coefficient two.
    Moreover, $\Phi_m\nmid f\pmod3$ for every $m\geq1$.
    If $F(y)\equiv A+By\pmod{y^2-2}$, then $D_8=|A^2-2B^2|\in\{1,7\}$.
\end{proposition}

\begin{proof}
    Theorem~\ref{thm:algebraic-noninteger} gives $\chi(\Gamma_\alpha)\leq4$, with three colors when $a_d\geq3$.
    Hence $\chi(\Gamma_\alpha)=4$ and $a_d=2$.
    Reciprocity gives $a_0=2$, and $F$ has leading coefficient two.
    None of $F(-2),F(0),F(2)$ vanishes, since $\alpha$ is not a root of unity.
    An absolute value greater than one would give a three-coloring by Lemma~\ref{lem:elementary-evaluation}.
    Lemma~\ref{lem:trace-evaluation-signs} therefore makes all three values equal to a sign $s$.
    Division of $F-s$ by the monic polynomial $y(y^2-4)$ gives the stated form with $Q\in\Z[y]$.
    Since $F$ has positive degree and leading coefficient two, $Q$ is nonzero and has leading coefficient two.

    If $\Phi_m\mid f$ modulo three, Proposition~\ref{prop:cyclotomic-congruence} would give a three-coloring.
    Here divisibility is by the whole reduction of $\Phi_m$, including its multiplicities.
    Finally, $F(\sqrt2)\ne0$, since otherwise $f$ would have an eighth root of unity as a root.
    Thus $D_8$ is a positive integer, and $f(1)=F(2)=s$ is odd.
    Lemma~\ref{lem:eighth-root-norm} gives three colors unless $D_8=1$ or $7$.
\end{proof}

\begin{corollary}
    \label{cor:noninteger-counterexample-degree}
    Any counterexample to Conjecture~\ref{conj:noninteger-three-color} has degree at least ten over $\Q$.
    In particular, every algebraic noninteger direction parameter of degree at most eight gives a three-colorable graph.
\end{corollary}

\begin{proof}
    A counterexample has $F(y)=y(y^2-4)Q(y)\pm1$ as in Proposition~\ref{prop:noninteger-counterexample-restrictions}.
    If $Q=2$, the worked sextic gives $D_8=31$; changing the constant term to $-1$ leaves the norm $|1-32|$ unchanged.
    This contradicts $D_8\in\{1,7\}$.
    If $Q(y)=2y+c$ for $c\in\Z$, then $F(y)\equiv(\pm1-8)-2cy\pmod{y^2-2}$.
    The signed norm is $49-8c^2$ or $81-8c^2$, each congruent to one modulo eight.
    Its absolute value can belong to $\{1,7\}$ only if the signed norm is $1$ or $-7$.
    This would require $c^2\in\{6,7,10,11\}$, none of which is a square.
    Thus $\deg Q\geq2$, so $\deg F\geq5$ and $\deg f\geq10$.
\end{proof}

These restrictions are necessary, not sufficient: a surviving parameter may admit a three-coloring by another construction.
We next ask what changes when the leading coefficient is one.

\subsection{Algebraic Integer Parameters}
\label{sec:algebraic-integers}
\label{subsec:sextic-character}

For an algebraic integer, the minimal polynomial is monic with constant coefficient one.
Each shifted relation now determines the next character value uniquely, so the extension argument cannot keep choosing values in a prescribed arc.
The evaluation methods still give colorings, and we propose the following bound.

\begin{conjecture}
    \label{conj:four-color}
    For every $\alpha\in\C$ with $|\alpha|=1$, we have 
    \[\chi(\Gamma_\alpha)\leq4.\]
\end{conjecture}

Section~\ref{sec:transcendental-cyclotomic} and Theorem~\ref{thm:algebraic-noninteger} leave only noncyclotomic algebraic integers unresolved.
To find a graph requiring four colors, we must first escape the three elementary evaluations.
Lemma~\ref{lem:trace-evaluation-signs} then forces $F(y)=y(y^2-4)Q(y)\pm1$.
The first possible monic trace degree is three, where $Q=1$.
Choose $F_0(y)=y^3-4y-1$, giving
\[f_0(x)=x^3F_0(x+x^{-1})=x^6-x^4-x^3-x^2+1.\]
This is the monic counterpart of the preceding sextic: all three elementary trace values are again units, now equal to $-1$.

The cubic $F_0$ has no rational root and changes sign on each of $(-2,-1)$, $(-1,0)$, and $(2,3)$.
Solving $x+x^{-1}=y$ for its two roots in $(-2,2)$ gives two conjugate pairs on the unit circle.
The third root gives a positive real reciprocal pair, as Figure~\ref{fig:sextic-roots} illustrates.

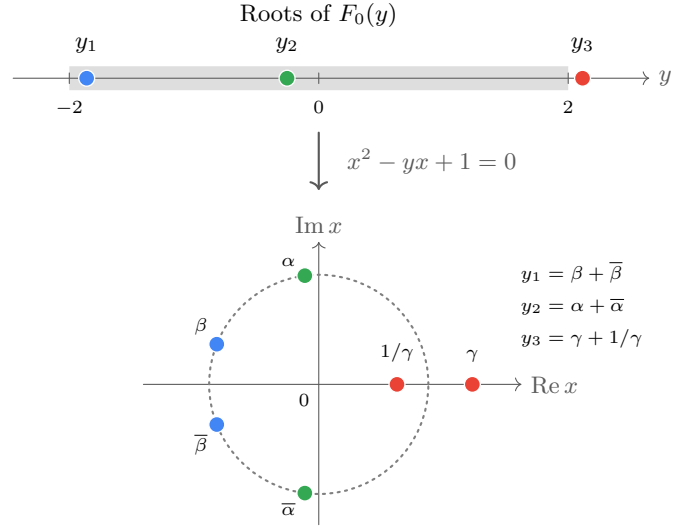
\begin{figure}[htbp]
    \centering
    \input{tikz-code/sextic-roots-noncyclotomic.tex}
    \caption{From the three roots of $F_0(y)=y^3-4y-1$ to the six roots of $f_0$.
    Each trace root lifts through $x^2-yx+1=0$ to a reciprocal pair.
    The two roots in $(-2,2)$ give unit-circle pairs; the root greater than two gives $\gamma$ and $1/\gamma$ on the real axis.}
    \label{fig:sextic-roots}
\end{figure}

As in the preceding sextic, a unit-modulus root is nonreal and quadratic over a real cubic field, so its degree is six and $f_0$ is its minimal polynomial.
The real root $\gamma>1$ then excludes roots of unity.
This $\gamma$ is the smallest degree-six Salem number~\cite[Section~7, Example~1]{GrossMcMullen2002}.

The eighth-root remainder is
\[F_0(y)\equiv-1-2y\pmod{y^2-2},\]
giving the exceptional norm $D_8=|1-2\cdot4|=7$ from Lemma~\ref{lem:eighth-root-norm}.
Here $S_8\cong\Z/7\Z$ with $[\sqrt2]=3$.
Define $h:R_8(f_0)\to\Z/7\Z$ by
\[h\bigl(a+b[\zeta_8]\bigr)=2a+2b\pmod7.\]
The relation $[\zeta_8]^2=3[\zeta_8]-1$ gives the period-eight root labels $(2,2,4,3,5,5,3,4)$, and hence
\[\Gamma_\alpha\longrightarrow\Gamma_8(f_0)\longrightarrow K_{7/2}.\]
Dividing by seven gives a character $\psi$ with separation $2/7$.
Figure~\ref{fig:period-eight-comparison} compares this character with the norm-$31$ construction.
To prove that a fourth color is necessary, we still need a finite obstruction.

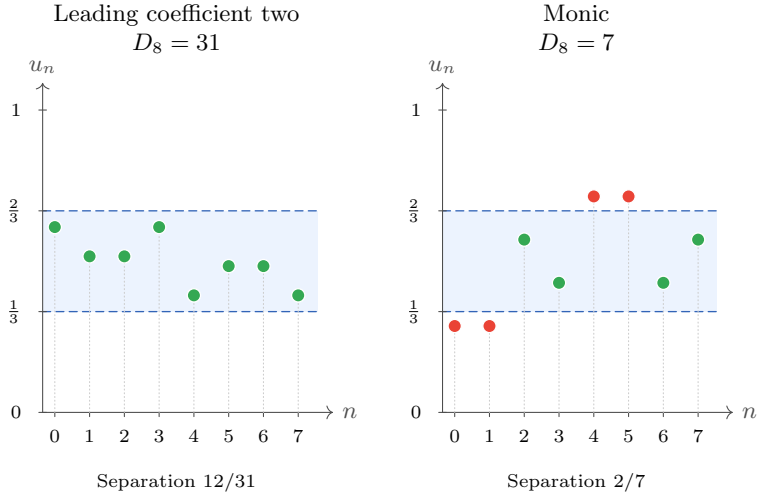
\begin{figure}[htbp]
    \centering
    \input{tikz-code/period-eight-comparison.tex}
    \caption{The period-eight character values $u_n=\psi(\alpha^n)$ for the two sextic examples, on the same scale.
    The shaded band is $[1/3,2/3]$; green points lie inside it and red points lie outside.
    For norm $31$, every step value lies in this band and the separation is $12/31$.
    For norm $7$, the separation is $2/7$.
    The second character supplies a four-coloring; failure of three-colorability is established separately by the witness.}
    \label{fig:period-eight-comparison}
    \label{fig:noninteger-period-eight}
\end{figure}

The witness answers the one-parameter question of Cervantes and Krebs~\cite[Example~2.14]{CervantesKrebs2023}.
The spindle in the introduction used two parameters to force four colors; here the powers of one parameter suffice.

\begin{theorem}
    \label{thm:sextic-four-chromatic}
    If $\alpha$ is a unit-modulus root of $x^6-x^4-x^3-x^2+1$, then $\chi(\Gamma_\alpha)=4$ and $\chi_c(\Gamma_\alpha)\leq7/2$.
\end{theorem}

\begin{proof}
    The character above gives the upper bounds.
    Let $H$ have the vertices in Table~\ref{tab:sextic-witness-vertices}, joining $v_i$ and $v_j$ precisely when $v_j-v_i\in\{\pm\alpha^k:0\leq k\leq12\}$.
    Exact polynomial reduction gives $24$ distinct vertices and $51$ edges, each an allowed step of $\Gamma_\alpha$.
    Exhaustive backtracking finds no proper three-coloring of $H$, as verified by the accompanying program described below.
    Thus $\chi(\Gamma_\alpha)\geq\chi(H)>3$.
    These checks use only the relation $f_0(\alpha)=0$ and the independence of $1,\alpha,\ldots,\alpha^5$, so the result holds for every unit-modulus root of $f_0$.
\end{proof}

We found $H$ by searching the finite sumset
\[B=\{0,1\}+\{0,\alpha\}+\cdots+\{0,\alpha^9\}.\]
Different selections of powers can represent the same vertex; for example, $1+\alpha^6=\alpha^2+\alpha^3+\alpha^4$.
Reducing the $2^{10}$ selections modulo $f_0$ leaves $830$ distinct vertices.
For the search graph we used directions $\pm\alpha^k$ through exponent $14$, giving $4{,}848$ edges.
A SAT search found no three-coloring and first isolated a $90$-vertex obstruction.
Further extraction and vertex deletion produced $H$, which is vertex-critical: deleting any one vertex makes it three-colorable.
Figure~\ref{fig:sextic-coloring} in the introduction shows its position in the search graph.

For the exact description, we translated one retained vertex to the origin.
Each row of Table~\ref{tab:sextic-witness-vertices} gives the six coefficients of $v_j$ in the basis $1,\alpha,\ldots,\alpha^5$.
The witness uses the directions through exponent $12$ specified in the proof, independently of the larger search cutoff.
For $v=c_0+c_1\alpha+\cdots+c_5\alpha^5$, its character label is
\[\ell(v)=2c_0+2c_1+4c_2+3c_3+5c_4+5c_5\pmod7,\]
with $\psi(v)=\ell(v)/7$.
Assign red, blue, green, and yellow to the label classes $\{0,1\}$, $\{2,3\}$, $\{4,5\}$, and $\{6\}$, respectively.
Figure~\ref{fig:sextic-witness} shows this coloring on the tabulated vertices.

\begin{table}[ht]
    \centering
    \small
    \setlength{\tabcolsep}{3pt}
    \begin{tabular}[t]{|r|rrrrrr|}
        \hline
        $j$ & $c_0$ & $c_1$ & $c_2$ & $c_3$ & $c_4$ & $c_5$\\
        \hline
        0 & 0 & 0 & 0 & 0 & 0 & 0\\
        1 & 0 & 0 & 0 & 1 & 1 & 1\\
        2 & 0 & 0 & 1 & 0 & 0 & 0\\
        3 & 0 & 0 & 1 & 1 & 0 & 0\\
        4 & 0 & 1 & 0 & 0 & 0 & 0\\
        5 & 1 & 0 & 0 & 0 & -1 & 0\\
        6 & 1 & 0 & 0 & 0 & 0 & 0\\
        7 & 1 & 1 & 0 & -1 & -2 & -2\\
        8 & 1 & 1 & 0 & -1 & -2 & -1\\
        9 & 1 & 1 & 0 & -1 & -1 & -1\\
        10 & 1 & 1 & 0 & 0 & -2 & -2\\
        11 & 2 & 1 & -1 & -2 & -3 & -2\\
        \hline
    \end{tabular}%
    \hspace{2em}%
    \begin{tabular}[t]{|r|rrrrrr|}
        \hline
        $j$ & $c_0$ & $c_1$ & $c_2$ & $c_3$ & $c_4$ & $c_5$\\
        \hline
        12 & 2 & 1 & -1 & -2 & -2 & -1\\
        13 & 2 & 1 & -1 & -1 & -3 & -2\\
        14 & 2 & 1 & 0 & -2 & -3 & -2\\
        15 & 2 & 2 & -1 & -3 & -3 & -2\\
        16 & 2 & 2 & -1 & -2 & -4 & -3\\
        17 & 2 & 2 & -1 & -2 & -4 & -2\\
        18 & 2 & 2 & -1 & -2 & -3 & -3\\
        19 & 2 & 2 & -1 & -2 & -3 & -2\\
        20 & 3 & 1 & -1 & -3 & -4 & -2\\
        21 & 3 & 2 & -1 & -4 & -5 & -3\\
        22 & 3 & 2 & -1 & -3 & -5 & -4\\
        23 & 3 & 2 & -1 & -3 & -5 & -3\\
        \hline
    \end{tabular}
    \caption{Exact coefficients of the witness vertices $v_j=\sum_{k=0}^5c_k\alpha^k$.}
    \label{tab:sextic-witness-vertices}
\end{table}

The files in \path{anc/} include \path{sextic-witness.csv}, which records the same coefficients, and the independent verifier \path{verify-sextic-witness.py}.
The program uses only integer arithmetic and the Python standard library.
It reconstructs the edges by reducing powers with $\alpha^6=-1+\alpha^2+\alpha^3+\alpha^4$ and checks the displayed four-coloring.
For the lower bound, it selects an uncolored vertex, tries every color absent from its colored neighbors, and backtracks whenever none is available.
Every available choice is explored, and every branch for three colors fails.
Together with the displayed four-coloring, this verifies that the witness has chromatic number four.
The verification depends on the exact coefficients, not on the numerical coordinates used for drawing.
Whether the circular upper bound $7/2$ is attained remains open.

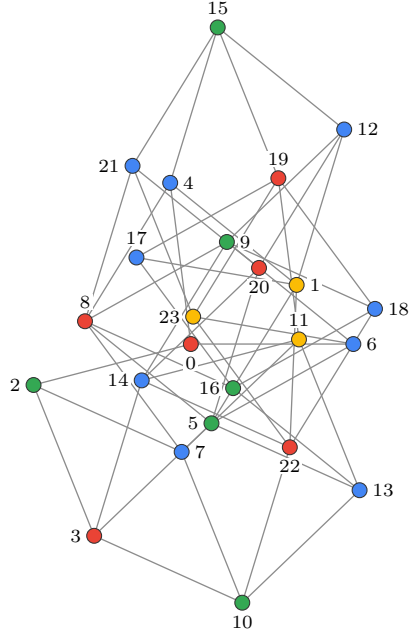
\begin{figure}[H]
    \centering
    \input{tikz-code/sextic-witness-noncyclotomic.tex}
    \caption{The sextic witness with vertex indices matching Table~\ref{tab:sextic-witness-vertices}.
    The drawing uses the root $\alpha$ with positive imaginary part and $-1<\alpha+\alpha^{-1}<0$.
    Colors come from the label classes of $\ell$.
    Crossings without dots are not vertices.}
    \label{fig:sextic-witness}
    \label{fig:sextic-extraction}
\end{figure}

\subsection{Four Colors Through Degree Eight}
\label{subsec:degree-eight}

The same trace evaluations explain why a graph requiring more than four colors must occur in higher degree.
Recall the degree-ten example from Subsection~\ref{subsec:trace-quotients}, whose trace polynomial is
\[F(y)=y(y^2-1)(y^2-4)+1.\]
Its value is one at all five integer evaluation points $-2,-1,0,1,2$.
The five intersections in Figure~\ref{fig:five-trace-evaluations} are the roots of $F-1$.

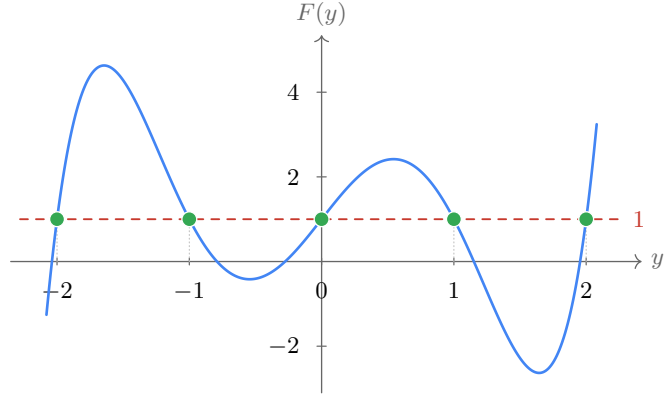
\begin{figure}[htbp]
    \centering
    \input{tikz-code/five-trace-evaluations.tex}
    \caption{Five failed integer evaluation tests for $F(y)=y(y^2-1)(y^2-4)+1$.
    The five intersections with the line of height one are distinct roots of $F-1$.
    A nonconstant trace polynomial of degree at most four cannot have this property.}
    \label{fig:five-trace-evaluations}
\end{figure}

\begin{theorem}
    \label{thm:degree-eight-bound}
    If $\alpha$ is an algebraic direction parameter with $\deg\alpha\leq8$, then 
    \[\chi(\Gamma_\alpha)\leq4.\]
\end{theorem}

\begin{proof}
    Roots of unity and algebraic nonintegers are covered by Theorems~\ref{thm:cyclotomic} and~\ref{thm:algebraic-noninteger}.
    Suppose that $\alpha$ is a noncyclotomic algebraic integer and $\chi(\Gamma_\alpha)>4$.
    Its monic trace polynomial $F$ has positive degree at most four.
    Each of $F(-2),F(-1),F(0),F(1),F(2)$ is nonzero, and an absolute value greater than one would give at most four colors by the trace evaluations in Subsection~\ref{subsec:trace-quotients}.
    Lemma~\ref{lem:trace-evaluation-signs} therefore makes all five values equal to a sign $s$.
    The nonzero polynomial $F-s$ has five distinct roots, contradicting its degree bound.
\end{proof}

\begin{corollary}
    \label{thm:degree-six-bound}
    Let $\alpha$ be an algebraic direction parameter of degree at most six, with primitive integer minimal polynomial $f$ normalized to have positive leading coefficient.
    With $f_0(x)=x^6-x^4-x^3-x^2+1$, we have
    \[\chi(\Gamma_\alpha)=
    \begin{cases}
        4,&f(x)=f_0(x)\text{ or }f(x)=f_0(-x),\\
        2,&f(1)\text{ is even},\\
        3,&\text{otherwise}.
    \end{cases}\]
\end{corollary}

\begin{proof}
    For a noncyclotomic algebraic integer in this range, $F$ is monic of positive degree at most three.
    If the evaluations at $-2,0,2$ all fail to give three colors, Lemma~\ref{lem:trace-evaluation-signs} gives $F(y)=y(y^2-4)Q(y)\pm1$.
    Monicity and the degree bound force $Q=1$, recovering $F(y)=y^3-4y\pm1$ and hence $f=f_0$ or $f(x)=f_0(-x)$.
    These graphs have chromatic number four by Theorem~\ref{thm:sextic-four-chromatic}, since $\Gamma_{-\alpha}=\Gamma_\alpha$.
    Every other algebraic integer in this degree range admits three colors by the elementary evaluations or Theorem~\ref{thm:cyclotomic}, and Corollary~\ref{cor:noninteger-counterexample-degree} treats the algebraic nonintegers.
    The bipartite criterion distinguishes two colors from three according to the parity of $f(1)$.
\end{proof}

A counterexample to Conjecture~\ref{conj:four-color} must therefore be a noncyclotomic algebraic integer of degree at least ten, with
\[F(y)=y(y^2-1)(y^2-4)Q(y)\pm1\]
for some nonzero monic $Q\in\Z[y]$.
It must also have $D_8=1$ and avoid every congruence $\Phi_m\mid f\pmod q$ with $q\geq2$, since the eighth-root and Cyclotomic Congruence Criteria would otherwise give four colors.
The general norm criterion (Theorem~\ref{thm:cyclotomic-norm}) further requires $D_m\leq B_m$ for each fixed $m\geq3$.
These are necessary restrictions, not evidence that a surviving parameter needs a fifth color.
The open problem is to find a coloring for every remaining parameter.
Cyclotomic evaluation offers one approach: choose an order for which the resulting finite quotient admits four colors.

%% file: tikz-code/semiminimal-extension.tex
\begin{tikzpicture}[
    line cap=round,
    line join=round,
    font=\small,
    candidate label/.style={font=\scriptsize,inner sep=1pt},
    boundary tick/.style={draw=googleBlue!80!black,line width=0.65pt}
]
    \definecolor{googleBlue}{HTML}{4285F4}
    \definecolor{googleGreen}{HTML}{34A853}

    \def\circleradius{1.40}
    \def\labelradius{1.78}
    \def\secondx{5.20}
    \pgfmathsetmacro{\arcstart}{90+360*3/7}
    \pgfmathsetmacro{\arcend}{90+360*4/7}

    \foreach \centerx in {0,\secondx} {
        \begin{scope}[shift={(\centerx,0)}]
            \draw[googleBlue!30,line width=7pt,line cap=butt]
                (\arcstart:\circleradius)
                arc[start angle=\arcstart,end angle=\arcend,
                    radius=\circleradius];
            \draw[black!45,line width=0.5pt]
                (0,0) circle[radius=\circleradius];
            \foreach \endpoint in {\arcstart,\arcend} {
                \draw[boundary tick]
                    (\endpoint:1.31) -- (\endpoint:1.49);
            }
        \end{scope}
    }

    \node at (0,2.18) {$\psi(\alpha^2)$};
    \foreach \j in {0,...,6} {
        \pgfmathsetmacro{\candidateangle}{90+360*\j/7}
        \coordinate (L\j) at (\candidateangle:\circleradius);
        \ifnum\j=3\relax
            \filldraw[fill=googleGreen,draw=white,line width=0.5pt]
                (L\j) circle[radius=3.1pt];
        \else
            \ifnum\j=4\relax
                \filldraw[fill=white,draw=googleGreen,line width=1pt]
                    (L\j) circle[radius=3.1pt];
            \else
                \filldraw[fill=black!75,draw=white,line width=0.4pt]
                    (L\j) circle[radius=2.4pt];
            \fi
        \fi
        \ifnum\j=0\relax
            \node[candidate label] at (\candidateangle:\labelradius) {$0$};
        \else
            \node[candidate label] at (\candidateangle:\labelradius)
                {$\j/7$};
        \fi
    }

    \begin{scope}[shift={(\secondx,0)}]
        \node at (0,2.18) {$\psi(\alpha^3)$};
        \draw[black!50,line width=0.5pt] (90:1.34) -- (90:1.46);
        \node[candidate label,text=black!65] at (90:1.15) {$0$};
        \foreach \j/\numerator in {0/13,1/27,2/41,3/55,4/69,5/83,6/97} {
            \pgfmathsetmacro{\candidateangle}{90+360*(\numerator/98)}
            \coordinate (R\j) at (\candidateangle:\circleradius);
            \ifnum\j=3\relax
                \filldraw[fill=googleGreen,draw=white,line width=0.5pt]
                    (R\j) circle[radius=3.1pt];
            \else
                \filldraw[fill=black!75,draw=white,line width=0.4pt]
                    (R\j) circle[radius=2.4pt];
            \fi
            \def\candidatelabelradius{\labelradius}
            \ifnum\j=1\relax\def\candidatelabelradius{2.05}\fi
            \ifnum\j=4\relax\def\candidatelabelradius{2.05}\fi
            \node[candidate label] at (\candidateangle:\candidatelabelradius)
                {$\numerator/98$};
        }
    \end{scope}
\end{tikzpicture}

%% file: tikz-code/sextic-roots-noncyclotomic.tex
\begin{tikzpicture}[line cap=round,line join=round,font=\small]
    \definecolor{googleBlue}{HTML}{4285F4}
    \definecolor{googleGreen}{HTML}{34A853}
    \definecolor{googleRed}{HTML}{EA4335}

    \begin{scope}[x=1.65cm,y=1cm,shift={(0,4.05)}]
        \draw[black!13,line width=9pt,line cap=butt] (-2,0) -- (2,0);
        \draw[->,black!65] (-2.45,0) -- (2.65,0) node[right] {$y$};
        \foreach \x in {-2,0,2} {
            \draw[black!65] (\x,-0.07) -- (\x,0.07);
            \node[below=5pt,font=\scriptsize] at (\x,0) {$\x$};
        }
        \foreach \x/\name/\pointcolor in {-1.8608058531/1/googleBlue,-0.2541016884/2/googleGreen,2.1149075415/3/googleRed} {
            \filldraw[fill=\pointcolor,draw=white,line width=0.5pt] (\x,0) circle[radius=3pt];
            \node[above=6pt] at (\x,0) {$y_{\name}$};
        }
        \node at (0,0.83) {Roots of $F_0(y)$};
    \end{scope}

    \draw[->,thick,black!65] (0,3.33) -- (0,2.57)
        node[midway,right=7pt] {$x^2-yx+1=0$};

    \begin{scope}[x=1.45cm,y=1.45cm]
        \draw[->,black!65] (-1.6,0) -- (1.85,0) node[right] {$\operatorname{Re}x$};
        \draw[->,black!65] (0,-1.28) -- (0,1.3) node[above] {$\operatorname{Im}x$};
        \draw[dotted,thick,black!50] (0,0) circle[radius=1];
        \node[below left,font=\scriptsize] at (0,0) {$0$};
        \foreach \sign in {-1,1} {
            \filldraw[fill=googleBlue,draw=white,line width=0.5pt]
                (-0.9304029266,{\sign*0.3665383940}) circle[radius=3pt];
            \filldraw[fill=googleGreen,draw=white,line width=0.5pt]
                (-0.1270508442,{\sign*0.9918962058}) circle[radius=3pt];
        }
        \node[above left,font=\scriptsize] at (-0.9304029266,0.3665383940) {$\beta$};
        \node[below left,font=\scriptsize] at (-0.9304029266,-0.3665383940) {$\overline\beta$};
        \node[above left,font=\scriptsize] at (-0.1270508442,0.9918962058) {$\alpha$};
        \node[below left,font=\scriptsize] at (-0.1270508442,-0.9918962058) {$\overline\alpha$};
        \foreach \x/\name in {0.7136391735/{1/\gamma},1.4012683679/\gamma} {
            \filldraw[fill=googleRed,draw=white,line width=0.5pt] (\x,0) circle[radius=3pt];
            \node[above=5pt,font=\scriptsize] at (\x,0) {$\name$};
        }
    \end{scope}

    \node[anchor=west,font=\scriptsize,align=left] at (2.55,1.06)
        {$y_1=\beta+\overline\beta$\\[4pt]
         $y_2=\alpha+\overline\alpha$\\[4pt]
         $y_3=\gamma+1/\gamma$};
\end{tikzpicture}

%% file: tikz-code/period-eight-comparison.tex
\begin{tikzpicture}[
    x=0.46cm,y=4cm,
    line cap=round,line join=round,
    font=\small,
    interval boundary/.style={draw=googleBlue!75!black,densely dashed,line width=0.5pt}
]
    \definecolor{googleBlue}{HTML}{4285F4}
    \definecolor{googleGreen}{HTML}{34A853}
    \definecolor{googleRed}{HTML}{EA4335}

    \foreach \offset/\normvalue in {0/31,11.5/7} {
        \begin{scope}[shift={(\offset,0)}]
            \fill[googleBlue!10] (-0.35,{1/3}) rectangle (7.55,{2/3});
            \draw[interval boundary] (-0.35,{1/3}) -- (7.55,{1/3});
            \draw[interval boundary] (-0.35,{2/3}) -- (7.55,{2/3});
            \draw[->,black!70] (-0.35,0) -- (8.0,0) node[right] {$n$};
            \draw[->,black!70] (-0.35,0) -- (-0.35,1.09) node[above] {$u_n$};
            \foreach \height/\ticklabel in {0/0,{1/3}/{\frac13},{2/3}/{\frac23},1/1} {
                \draw[black!70] (-0.46,\height) -- (-0.24,\height);
                \node[left=3pt,font=\scriptsize] at (-0.46,\height) {$\ticklabel$};
            }
            \foreach \n in {0,...,7} {
                \draw[black!70] (\n,0) -- (\n,-0.014);
                \node[below=3pt,font=\scriptsize] at (\n,0) {$\n$};
            }
            \ifnum\normvalue=31\relax
                \node[align=center] at (3.5,1.27) {Leading coefficient two\\$D_8=31$};
                \def\denominator{31}
                \def\values{0/19,1/16,2/16,3/19,4/12,5/15,6/15,7/12}
                \node[font=\scriptsize] at (3.5,-0.23) {Separation $12/31$};
            \else
                \node[align=center] at (3.5,1.27) {Monic\\$D_8=7$};
                \def\denominator{7}
                \def\values{0/2,1/2,2/4,3/3,4/5,5/5,6/3,7/4}
                \node[font=\scriptsize] at (3.5,-0.23) {Separation $2/7$};
            \fi
            \foreach \n/\numerator in \values {
                \coordinate (V\normvalue-\n) at (\n,{\numerator/\denominator});
                \draw[black!20,densely dotted] (\n,0) -- (V\normvalue-\n);
                \pgfmathtruncatemacro{\outside}{(3*\numerator<\denominator)||(3*\numerator>2*\denominator)}
                \ifnum\outside=1\relax
                    \def\pointcolor{googleRed}
                \else
                    \def\pointcolor{googleGreen}
                \fi
                \filldraw[fill=\pointcolor,draw=white,line width=0.5pt]
                    (V\normvalue-\n) circle[radius=2.5pt];
            }
        \end{scope}
    }
\end{tikzpicture}

%% file: tikz-code/sextic-witness-noncyclotomic.tex
\begin{tikzpicture}[x=2.15cm,y=2.15cm,line cap=round,line join=round,
    graph edge/.style={draw=black!45,line width=0.5pt},
    index label/.style={font=\scriptsize,fill=white,inner sep=0.7pt}]
    \definecolor{googleRed}{HTML}{EA4335}
    \definecolor{googleBlue}{HTML}{4285F4}
    \definecolor{googleGreen}{HTML}{34A853}
    \definecolor{googleYellow}{HTML}{FBBC05}
    \coordinate (v0) at (0.0000000000,0.0000000000); 
    \coordinate (v1) at (0.6511313015,0.3638575539); 
    \coordinate (v2) at (-0.9677161660,-0.2520425006); 
    \coordinate (v3) at (-0.5947670102,-1.1798942814); 
    \coordinate (v4) at (-0.1270508442,0.9918962058); 
    \coordinate (v5) at (0.1270508442,-0.4878112046); 
    \coordinate (v6) at (1.0000000000,0.0000000000); 
    \coordinate (v7) at (-0.0563642913,-0.6636706829); 
    \coordinate (v8) at (-0.6511313015,0.1402274472); 
    \coordinate (v9) at (0.2218178544,0.6280386518); 
    \coordinate (v10) at (0.3165848645,-1.5915224637); 
    \coordinate (v11) at (0.6654535631,0.0284123938); 
    \coordinate (v12) at (0.9436357087,1.3201217286); 
    \coordinate (v13) at (1.0384027189,-0.8994393870); 
    \coordinate (v14) at (-0.3022626029,-0.2236301067); 
    \coordinate (v15) at (0.1654535631,1.9481603804); 
    \coordinate (v16) at (0.2602205732,-0.2714007351); 
    \coordinate (v17) at (-0.3345464369,0.5324973950); 
    \coordinate (v18) at (1.1331697290,0.2164104695); 
    \coordinate (v19) at (0.5384027189,1.0203085996); 
    \coordinate (v20) at (0.4195552514,0.4684529700); 
    \coordinate (v21) at (-0.3586268942,1.0964916219); 
    \coordinate (v22) at (0.6090892718,-0.6352582891); 
    \coordinate (v23) at (0.0143222616,0.1686398410); 
    \draw[graph edge] (v0) -- (v2); 
    \draw[graph edge] (v0) -- (v4); 
    \draw[graph edge] (v0) -- (v6); 
    \draw[graph edge] (v1) -- (v4); 
    \draw[graph edge] (v1) -- (v5); 
    \draw[graph edge] (v1) -- (v12); 
    \draw[graph edge] (v1) -- (v17); 
    \draw[graph edge] (v1) -- (v22); 
    \draw[graph edge] (v2) -- (v3); 
    \draw[graph edge] (v2) -- (v7); 
    \draw[graph edge] (v3) -- (v5); 
    \draw[graph edge] (v3) -- (v10); 
    \draw[graph edge] (v3) -- (v14); 
    \draw[graph edge] (v4) -- (v8); 
    \draw[graph edge] (v4) -- (v15); 
    \draw[graph edge] (v5) -- (v6); 
    \draw[graph edge] (v5) -- (v8); 
    \draw[graph edge] (v5) -- (v13); 
    \draw[graph edge] (v5) -- (v20); 
    \draw[graph edge] (v6) -- (v9); 
    \draw[graph edge] (v6) -- (v23); 
    \draw[graph edge] (v7) -- (v8); 
    \draw[graph edge] (v7) -- (v10); 
    \draw[graph edge] (v7) -- (v11); 
    \draw[graph edge] (v8) -- (v9); 
    \draw[graph edge] (v8) -- (v16); 
    \draw[graph edge] (v8) -- (v21); 
    \draw[graph edge] (v9) -- (v12); 
    \draw[graph edge] (v9) -- (v14); 
    \draw[graph edge] (v9) -- (v18); 
    \draw[graph edge] (v10) -- (v13); 
    \draw[graph edge] (v10) -- (v22); 
    \draw[graph edge] (v11) -- (v13); 
    \draw[graph edge] (v11) -- (v14); 
    \draw[graph edge] (v11) -- (v19); 
    \draw[graph edge] (v12) -- (v15); 
    \draw[graph edge] (v12) -- (v20); 
    \draw[graph edge] (v13) -- (v16); 
    \draw[graph edge] (v14) -- (v20); 
    \draw[graph edge] (v14) -- (v22); 
    \draw[graph edge] (v15) -- (v19); 
    \draw[graph edge] (v15) -- (v21); 
    \draw[graph edge] (v16) -- (v17); 
    \draw[graph edge] (v16) -- (v18); 
    \draw[graph edge] (v17) -- (v19); 
    \draw[graph edge] (v18) -- (v19); 
    \draw[graph edge] (v18) -- (v22); 
    \draw[graph edge] (v19) -- (v23); 
    \draw[graph edge] (v20) -- (v21); 
    \draw[graph edge] (v21) -- (v23); 
    \draw[graph edge] (v22) -- (v23); 
    \filldraw[fill=googleRed,draw=black!80,line width=0.4pt] (v0) circle[radius=2.8pt]; 
    \filldraw[fill=googleYellow,draw=black!80,line width=0.4pt] (v1) circle[radius=2.8pt]; 
    \filldraw[fill=googleGreen,draw=black!80,line width=0.4pt] (v2) circle[radius=2.8pt]; 
    \filldraw[fill=googleRed,draw=black!80,line width=0.4pt] (v3) circle[radius=2.8pt]; 
    \filldraw[fill=googleBlue,draw=black!80,line width=0.4pt] (v4) circle[radius=2.8pt]; 
    \filldraw[fill=googleGreen,draw=black!80,line width=0.4pt] (v5) circle[radius=2.8pt]; 
    \filldraw[fill=googleBlue,draw=black!80,line width=0.4pt] (v6) circle[radius=2.8pt]; 
    \filldraw[fill=googleBlue,draw=black!80,line width=0.4pt] (v7) circle[radius=2.8pt]; 
    \filldraw[fill=googleRed,draw=black!80,line width=0.4pt] (v8) circle[radius=2.8pt]; 
    \filldraw[fill=googleGreen,draw=black!80,line width=0.4pt] (v9) circle[radius=2.8pt]; 
    \filldraw[fill=googleGreen,draw=black!80,line width=0.4pt] (v10) circle[radius=2.8pt]; 
    \filldraw[fill=googleYellow,draw=black!80,line width=0.4pt] (v11) circle[radius=2.8pt]; 
    \filldraw[fill=googleBlue,draw=black!80,line width=0.4pt] (v12) circle[radius=2.8pt]; 
    \filldraw[fill=googleBlue,draw=black!80,line width=0.4pt] (v13) circle[radius=2.8pt]; 
    \filldraw[fill=googleBlue,draw=black!80,line width=0.4pt] (v14) circle[radius=2.8pt]; 
    \filldraw[fill=googleGreen,draw=black!80,line width=0.4pt] (v15) circle[radius=2.8pt]; 
    \filldraw[fill=googleGreen,draw=black!80,line width=0.4pt] (v16) circle[radius=2.8pt]; 
    \filldraw[fill=googleBlue,draw=black!80,line width=0.4pt] (v17) circle[radius=2.8pt]; 
    \filldraw[fill=googleBlue,draw=black!80,line width=0.4pt] (v18) circle[radius=2.8pt]; 
    \filldraw[fill=googleRed,draw=black!80,line width=0.4pt] (v19) circle[radius=2.8pt]; 
    \filldraw[fill=googleRed,draw=black!80,line width=0.4pt] (v20) circle[radius=2.8pt]; 
    \filldraw[fill=googleBlue,draw=black!80,line width=0.4pt] (v21) circle[radius=2.8pt]; 
    \filldraw[fill=googleRed,draw=black!80,line width=0.4pt] (v22) circle[radius=2.8pt]; 
    \filldraw[fill=googleYellow,draw=black!80,line width=0.4pt] (v23) circle[radius=2.8pt]; 
    \node[index label,below=4pt] at (v0) {0};
    \node[index label,right=4pt] at (v1) {1};
    \node[index label,left=4pt] at (v2) {2};
    \node[index label,left=4pt] at (v3) {3};
    \node[index label,right=4pt] at (v4) {4};
    \node[index label,left=4pt] at (v5) {5};
    \node[index label,right=4pt] at (v6) {6};
    \node[index label,right=4pt] at (v7) {7};
    \node[index label,above=4pt] at (v8) {8};
    \node[index label,right=4pt] at (v9) {9};
    \node[index label,below=4pt] at (v10) {10};
    \node[index label,above=4pt] at (v11) {11};
    \node[index label,right=4pt] at (v12) {12};
    \node[index label,right=4pt] at (v13) {13};
    \node[index label,left=4pt] at (v14) {14};
    \node[index label,above=4pt] at (v15) {15};
    \node[index label,left=4pt] at (v16) {16};
    \node[index label,above=4pt] at (v17) {17};
    \node[index label,right=4pt] at (v18) {18};
    \node[index label,above=4pt] at (v19) {19};
    \node[index label,below=4pt] at (v20) {20};
    \node[index label,left=4pt] at (v21) {21};
    \node[index label,below=4pt] at (v22) {22};
    \node[index label,left=4pt] at (v23) {23};
\end{tikzpicture}

%% file: tikz-code/five-trace-evaluations.tex
\begin{tikzpicture}[
    x=1.75cm,y=0.56cm,
    line cap=round,line join=round,
    font=\small
]
    \definecolor{googleBlue}{HTML}{4285F4}
    \definecolor{googleGreen}{HTML}{34A853}
    \definecolor{googleRed}{HTML}{EA4335}

    \draw[->,black!60,line width=0.5pt] (-2.35,0) -- (2.42,0)
        node[right] {$y$};
    \draw[->,black!60,line width=0.5pt] (0,-3.1) -- (0,5.35)
        node[above] {$F(y)$};
    \foreach \value in {-2,-1,0,1,2} {
        \draw[black!60,line width=0.45pt] (\value,-0.08) -- (\value,0.08);
        \node[below=3pt] at (\value,-0.08) {$\value$};
    }
    \foreach \height in {-2,2,4} {
        \draw[black!60,line width=0.45pt] (-0.035,\height) -- (0.035,\height);
        \node[left=3pt] at (-0.035,\height) {$\height$};
    }
    \foreach \value in {-2,-1,1,2} {
        \draw[black!25,densely dotted] (\value,0) -- (\value,1);
    }

    \draw[googleBlue,line width=1pt]
        plot[domain=-2.08:2.08,samples=201,smooth]
        (\x,{\x*(\x*\x-1)*(\x*\x-4)+1});
    \draw[googleRed!85!black,dashed,line width=0.7pt]
        (-2.28,1) -- (2.28,1)
        node[right,font=\small] {$1$};
    \foreach \value in {-2,-1,0,1,2} {
        \filldraw[fill=googleGreen,draw=white,line width=0.5pt]
            (\value,1) circle[radius=2.8pt];
    }
\end{tikzpicture}

%% file: 5-conclusion.tex

\section{Conclusion and Future Directions}
\label{sec:conclusion}

The one-parameter results give a first illustration of how the monomial Cayley graph program can turn algebraic relations into information about colorings.
The multiplier method uses these relations to construct colorings, while finite subgraphs such as the sextic witness show when additional colors are necessary.
The next steps are to complete the one-parameter classification and to understand how several parameters can interact to require more colors.

Two conjectures remain.
Conjecture~\ref{conj:four-color}, the uniform four-color bound, is open only for noncyclotomic algebraic integers of degree at least ten.
For Conjecture~\ref{conj:noninteger-three-color}, only algebraic nonintegers with leading coefficient two and degree at least ten remain.
Together with the bipartite criterion, these conjectures would reduce the classification to distinguishing the algebraic integer parameters that require three colors from those that require four.

Preliminary exploratory work with additional cyclotomic tests suggests that the four-color problem through degree twenty-six can be reduced to a finite collection of unresolved minimal polynomials.
We do not develop these extensions here.
One remaining example is the degree-twenty polynomial $f_\ast(x)=x^{10}F_\ast(x+x^{-1})$, where
\[
F_\ast(y)=y(y^2-1)(y^2-4)Q_\ast(y)-1
\]
and $Q_\ast(y)=y^5+y^4-4y^3-3y^2+3y+1$.
Thus all five integer evaluations from Subsection~\ref{subsec:degree-eight} equal $-1$, so they give no coloring target.
Four-colorability for its unit-modulus roots remains unresolved.

Gross and McMullen's Salem-polynomial construction~\cite[Theorem~7.3]{GrossMcMullen2002} gives a general limitation of fixed-order tests, stated in Theorem~\ref{thm:salem-finite-orders}.
For every prescribed finite set of cyclotomic orders, there are irreducible Salem polynomials of arbitrarily large degree whose evaluations at those orders are units.
Each of these evaluations gives the trivial quotient, so no fixed finite collection of orders supplies a coloring target for every parameter.
A uniform proof by cyclotomic targets must therefore select an order from the minimal polynomial itself.
It must also account for the order-dependent threshold $B_m$ in Theorem~\ref{thm:cyclotomic-norm}: growth of the norms alone does not ensure that they exceed their respective thresholds.

For a noncyclotomic algebraic integer parameter with minimal polynomial of degree $d$, the recurrence for successive character values defines an automorphism of $\T^d$.
A forward orbit contained in $[1/4,3/4]^d$ is sufficient for a four-coloring: the corresponding character stays separated from zero on every allowed step.
Lindenstrauss and Schmidt's work on symbolic representations of nonexpansive automorphisms, including those associated with Salem polynomials, provides a related setting for studying these orbits~\cite{LindenstraussSchmidt2005}.

To pass to several parameters, return to the Moser spindle in Figure~\ref{fig:moser-spindle}.
Its parameters $\alpha$ and $\beta$ have squared cosines $1/4$ and $11/12$, respectively.
Each individual parameter gives a three-colorable graph, while the spindle forces at least four colors in $\MonCay(\alpha,\beta)$.
The connecting edge makes the interaction explicit: the mixed relation $(1+\alpha)(\beta^2-1)=\alpha^2\beta$ links the two diamonds and prevents their three-colorings from fitting together.
A character must respect these relations while separating every mixed monomial step from zero.

This example suggests beginning with parameters $\alpha=e^{i\theta}$ for which $q=\cos^2\theta$ is rational.
The family builds on the algebraic rotation constructions explored in Polymath16~\cite{Mixon2018PolymathThird}; Conway, Radin, and Sadun studied rational linear relations among angles with rational squared trigonometric functions~\cite{ConwayRadinSadun1999}.
Each parameter satisfies $\alpha^4+(2-4q)\alpha^2+1=0$, so it has degree at most four and gives a three-colorable graph by Corollary~\ref{thm:degree-six-bound}.
The appendix's embeddings of the five-chromatic graphs of de Grey and Parts use four parameters from this family~\cite{deGrey2018,Parts2020}.

The construction also extends to higher dimensions by applying products of chosen orthogonal transformations to a fixed unit vector.
Their orbit generates the additive vertex group, with the orbit vectors and their negatives as steps.
Unlike planar rotations, these transformations need not commute, so their order matters and their relations must be expressed as words in the transformations.
The vertex group remains abelian, leaving quotient maps and separated additive characters available as coloring tools.

Returning to the plane, the multiparameter problem has an extremal formulation.
For each positive integer $r$, let
\[
M(r)=\max\bigl\{\chi(\MonCay(\mathcal A)):
\mathcal A\subseteq\{z\in\C:|z|=1\},\ 1\leq|\mathcal A|\leq r\bigr\}.
\]
The maximum exists because the possible values are integers bounded above by $\chi(\R^2)\leq7$, and $M(r)$ is nondecreasing.
The sextic gives $M(1)\geq4$, with equality predicted by Conjecture~\ref{conj:four-color}; the four-parameter embeddings give $M(4)\geq5$.
Can two parameters already force five colors?
The same question can be asked with all parameters restricted to rational squared cosines.
In that restricted family, the one-parameter maximum is three, attained by a sixth root of unity.

The unrestricted function reaches the chromatic number of the plane for some finite $r$.
Indeed, compactness gives a finite connected unit-distance graph with chromatic number $\chi(\R^2)$~\cite{deBruijnErdos1951}.
The algebraic-coordinate argument in the introduction embeds this graph using finitely many algebraic edge directions.
Choosing those directions as parameters places the graph in a monomial Cayley graph, giving $M(r)=\chi(\R^2)$ for a sufficiently large finite $r$.

The lower bound of OpenAI~\cite{OpenAI2026Plane} therefore implies that $M(r)\geq6$ for some finite $r$.
Determining the least such $r$, and exhibiting parameters that attain it, are concrete problems for the multiparameter program.
The preprint obtains the existence of a finite graph that is not five-colorable through compactness, following an analytic and geometric argument.
Can a finite obstruction be extracted from that argument and embedded using a small collection of monomial generators?
An explicit construction would allow its coloring obstruction to be studied through the algebraic relations among those generators.
The broader question remains: how few parameters suffice to attain the chromatic number of the plane, and which relations make them do so?

%% file: 6-appendix.tex

\section{Limits of Fixed Cyclotomic Tests}
\label{app:salem-obstruction}

A finite collection of cyclotomic evaluations can miss minimal polynomials of arbitrarily large degree.
We apply Gross and McMullen's construction of Salem polynomials~\cite[Theorem~7.3]{GrossMcMullen2002} to make each prescribed evaluation a unit, so the corresponding quotient ring is trivial.
It uses the trace variable $y=x+x^{-1}$ to arrange the evaluations and control the roots at the same time.

For example, consider the orders $1$, $2$, $3$, and $6$.
Their traces $\zeta_m+\zeta_m^{-1}$ are $2$, $-2$, $-1$, and $1$, respectively.
Set $H(y)=(y^2-1)(y^2-4)$ and, for a positive integer $T$, define
\[
F_T(y)=(y-T)(y^2-1)(y^2-4)-1.
\]
The reciprocal polynomial $f_T(x)=x^5F_T(x+x^{-1})$ then satisfies $f_T(\zeta_m)=-\zeta_m^5$ at all four orders.
These values are units.
For sufficiently large $T$, four roots of $F_T$ lie in $(-2,2)$ and the fifth lies beyond $2$.
The first four lift to roots on the unit circle, while the last lifts to a positive real reciprocal pair.
The proof below also establishes irreducibility for all sufficiently large $T$.
Thus this example produces Salem polynomials: monic irreducible reciprocal integer polynomials with one positive real reciprocal pair off the unit circle and all remaining roots on it.

The following consequence also follows from Za\"imi's theorem on Salem numbers whose prescribed powers minus one are units~\cite[Theorem~2]{Zaimi2023}.
We give a direct proof using the Gross--McMullen construction to explain how the prescribed orders enter the argument.

\begin{theorem}
    \label{thm:salem-finite-orders}
    For every finite set $\mathcal M$ of positive integers and every degree bound $D$, there is an irreducible Salem polynomial $f$ of degree greater than $D$ such that $f(\zeta_m)$ is a unit in $\Z[\zeta_m]$ for every $m\in\mathcal M$.
\end{theorem}

\begin{proof}
    For $m\geq3$, let $\Psi_m$ be the minimal polynomial of $\zeta_m+\zeta_m^{-1}$.
    Its roots are distinct and lie in $(-2,2)$, and $\deg\Psi_m=\varphi(m)/2$.
    Different orders give disjoint root sets: the trace determines the pair consisting of a root of unity and its inverse.
    Start with the product of $y-2$, $y+2$, and the distinct $\Psi_m$ for $m\in\mathcal M\setminus\{1,2\}$.
    Adjoin further distinct trace factors so that the resulting monic polynomial $H$ has even degree $h$ and $2h+2>D$.
    This is always possible: unused factors $\Psi_{3^k}$ have odd degree and can correct the parity, while unused factors $\Psi_{2^k}$, for $k\geq3$, have even degrees tending to infinity.
    Every root of $H$ is simple, with its smallest and largest roots equal to $-2$ and $2$.

    Set $F_T(y)=(y-T)H(y)-1$.
    We first locate its roots for large positive $T$.
    Around each interior root of $H$, choose a small interval contained in $(-2,2)$ on whose endpoints $H$ has opposite signs, with these intervals pairwise disjoint.
    For large $T$, the endpoint signs of $F_T$ are opposite as well, giving a root in each interval.
    Since $H$ is monic of even degree, $H'(-2)<0$ and $H'(2)>0$.
    Hence $H$ is negative immediately inside both endpoints, where $F_T$ becomes positive as $T$ grows.
    But $F_T(-2)=F_T(2)=-1$, so the two remaining roots near the endpoints also lie strictly inside $(-2,2)$.
    Finally, $F_T(T)=-1$ and $F_T(T+1)=H(T+1)-1>0$ for large $T$.
    We have therefore found all $h+1$ roots: $h$ distinct roots in $(-2,2)$ and one in $(T,T+1)$.

    Only finitely many of these integer values of $T$ can give a reducible $F_T$.
    Indeed, any factorization has a monic irreducible factor $Q$ that omits the unique root beyond $2$.
    All roots of $Q$ lie in $(-2,2)$, so Kronecker's theorem implies that $Q=\Psi_n$ for some $n\geq3$~\cite{Kronecker1857}.
    There are only finitely many possibilities with $\deg\Psi_n\leq h$.
    A factor of $H$ cannot divide $F_T$, since $F_T\equiv-1\pmod H$.
    Any other $\Psi_n$ can divide $F_T$ for at most one value of $T$: if it divided both $F_T$ and $F_U$ for $T\ne U$, their difference $(U-T)H$ would force it to divide $H$.
    Choose a sufficiently large integer $T$ outside these finitely many exceptions.

    The reciprocal lift
    \[
    f_T(x)=x^{h+1}F_T(x+x^{-1})
    \]
    is a monic integer polynomial of degree $2h+2$.
    It is irreducible as well.
    To see this, choose a root $a\in(-2,2)$ of $F_T$ and a root $\lambda$ of $x^2-ax+1$.
    Then $\Q(a)\subset\R$, while $\lambda$ is nonreal, so $[\Q(\lambda):\Q(a)]=2$.
    Since $a=\lambda+\lambda^{-1}$ and $F_T$ is irreducible of degree $h+1$, the degree of $\lambda$ is $2h+2$.
    Its minimal polynomial is therefore $f_T$.
    The root locations established above give exactly one positive real reciprocal pair off the unit circle, so $f_T$ is a Salem polynomial.

    For every $m\in\mathcal M$, the trace $\zeta_m+\zeta_m^{-1}$ is a root of $H$.
    Consequently,
    \[
    f_T(\zeta_m)=-\zeta_m^{h+1},
    \]
    a unit in $\Z[\zeta_m]$.
    The choice of $h$ gives the required degree bound.
\end{proof}

Each nonreal root of the polynomial in Theorem~\ref{thm:salem-finite-orders} is an admissible unit-modulus parameter, but every prescribed quotient $\Z[\zeta_m]/(f(\zeta_m))$ is trivial.
Thus a uniform coloring argument based on cyclotomic evaluations must allow the order to depend on the minimal polynomial.

\section{Embedding Five-Chromatic Graphs}
\label{app:de-grey}

The Moser spindle supplies the starting point for two explicit five-chromatic embeddings.
We first recall a four-coloring of its two-parameter monomial Cayley graph, then introduce the additional rotations used in de Grey's construction.
Parts's smaller graph admits a certificate using the same spindle parameters and a different fourth parameter.
The geometric constructions and five-chromaticity are due to de Grey and Parts~\cite{deGrey2018,Parts2020}; our task is to express their vertices and edges in the monomial framework.

\subsection{The Spindle Parameters}

Retain $\alpha=(1+i\sqrt3)/2$ and $\beta=(\sqrt{33}+i\sqrt3)/6$ from the introduction.
They satisfy $\alpha^2-\alpha+1=0$ and the mixed relation $(1+\alpha)(\beta^2-1)=\alpha^2\beta$.
Following Speyer's finite-field coloring method~\cite{Speyer2018FiniteField}, we send both parameters to $u$ in $\mathbb F_4=\mathbb F_2[u]/(u^2+u+1)$.

To justify this assignment, work first in $A=\Z[\alpha,1/3]$ and put $t=2\alpha-1$.
Then $t^2=-3$, and $\beta$ satisfies the monic relation $\beta^2-(t/3)\beta-1=0$.
Its discriminant $11/3$ is not a square in $\Q(\alpha)$, whose real elements are rational, so evaluation identifies $A[\beta]$ with $A[W]/(W^2-(t/3)W-1)$.
The map $A\to\mathbb F_4$ sending $\alpha$ to $u$ sends $t/3$ to $1$, and the relation for $W$ becomes $W^2+W+1=0$.
Thus $\beta\mapsto u$ extends this map, and restriction gives a homomorphism $\Z[\alpha,\beta]\to\mathbb F_4$.
Every allowed step $\pm\alpha^a\beta^b$ has nonzero image $u^{a+b}$, so the four field elements give a proper coloring.
Together with the spindle, this proves
\[
\chi(\MonCay(\alpha,\beta))=4.
\]

\subsection{De Grey's Graph}

De Grey builds his graph from a $39$-point seed by reflection, rotation, and translation~\cite[Section~5.1]{deGrey2018}.
Let $S$ be that seed, identifying $(x,y)$ with $x+iy$, and retain the order of his coordinate list.
We use the rotated seed $P=\beta^2S=\{p_1,\ldots,p_{39}\}$ so that every coordinate is an integer polynomial in $\alpha$ and $\beta$.
For example, the seed vertex $1/3$ becomes
\[
p_2=\frac{\beta^2}{3}=2\beta^2-1-\beta^4,
\]
while $(\sqrt{33}+13+i(\sqrt{11}-\sqrt3))/12$ becomes $p_{22}=1+\alpha\beta$.
Table~\ref{tab:de-grey-seed} gives the full list.

\begin{table}[H]
    \centering
    \small
    \renewcommand{\arraystretch}{1.15}
    \begin{tabular}[t]{|r|l|}
        \hline
        $i$ & $p_i$ \\
        \hline
        $1$ & $0$ \\
        $2$ & $2\beta^2-1-\beta^4$ \\
        $3$ & $\beta^2$ \\
        $4$ & $2\beta^2$ \\
        $5$ & $\alpha\beta^3-\alpha^2\beta$ \\
        $6$ & $\alpha^2(\beta-\beta^3)$ \\
        $7$ & $1+\alpha+\alpha^2\beta^3$ \\
        $8$ & $(1+\alpha)\beta^2$ \\
        $9$ & $2\beta^2-1$ \\
        $10$ & $\alpha\beta^4-\alpha^2\beta$ \\
        $11$ & $1-\beta+\beta^3$ \\
        $12$ & $\beta^3+\alpha^2(\beta-1)$ \\
        $13$ & $\alpha\beta^4-\alpha^2\beta^3$ \\
        $14$ & $\beta^3$ \\
        $15$ & $\beta+\alpha\beta^2$ \\
        $16$ & $\beta^3+\alpha\beta^4-\alpha^2\beta$ \\
        $17$ & $\alpha^2(1-\beta^3)$ \\
        $18$ & $\beta+\beta^2-1$ \\
        $19$ & $\beta^4-\beta^2+\beta$ \\
        $20$ & $\beta+\alpha\beta^2-\beta^4$ \\
        \hline
    \end{tabular}
    \hspace{2em}
    \begin{tabular}[t]{|r|l|}
        \hline
        $i$ & $p_i$ \\
        \hline
        $21$ & $\alpha^2(1-\beta)$ \\
        $22$ & $1+\alpha\beta$ \\
        $23$ & $\alpha\beta^2-\alpha^2\beta^4$ \\
        $24$ & $\beta+\beta^2+\alpha^2\beta^3$ \\
        $25$ & $\beta^3+\beta^4+\alpha^2(\beta^2-1)$ \\
        $26$ & $1+\alpha\beta^3$ \\
        $27$ & $\beta+\beta^2+\alpha^2\beta^4$ \\
        $28$ & $\alpha$ \\
        $29$ & $\beta^2-\beta^4+\alpha\beta$ \\
        $30$ & $\alpha\beta^3-\alpha^2\beta^2$ \\
        $31$ & $\beta^4-\alpha^2\beta^3-1$ \\
        $32$ & $1+\alpha^2\beta+\beta^3-\beta^4$ \\
        $33$ & $\alpha^2(1-\beta^2)$ \\
        $34$ & $\beta^3+\alpha^2(\beta^2-1)$ \\
        $35$ & $\beta^4-\alpha^2\beta-1$ \\
        $36$ & $(1-\alpha^2)\beta^4-\beta$ \\
        $37$ & $1+\alpha-\beta$ \\
        $38$ & $\beta^4-\beta-\alpha^2\beta^3$ \\
        $39$ & $\alpha\beta^4-\alpha^2\beta^2$ \\
        \hline
    \end{tabular}
    \caption{Polynomial coordinates of the rotated seed $P=\beta^2S$, in de Grey's order.}
    \label{tab:de-grey-seed}
\end{table}

Reflecting the original seed and rotating through multiples of $\pi/3$ gives
\[
T=\bigcup_{j=0}^{5}\alpha^j\bigl(P\cup\beta^4\overline P\bigr).
\]
Here the bar denotes complex conjugation; the factor $\beta^4$ appears because $\beta^2\overline S=\beta^4\overline P$.
All vertices of $T$ still lie in $\Z[\alpha,\beta]$.
Indeed, conjugation replaces $\alpha$ by $\alpha^5$, and multiplying by $\beta^4$ clears the negative powers of $\beta$ in the conjugated seed table.
The set $T$ has $397$ vertices and $1974$ unit edges, whose differences belong to
\[
\mathcal D_0=\{\alpha^j\beta^k:0\leq j\leq5,\ 0\leq k\leq4\}.
\]

The next rotation joins two copies of $T$ with six additional unit edges.
Set $\delta=(-1+i\sqrt{15})/4$, which satisfies $2\delta^2+\delta+2=0$.
The multiplier $-\delta^2=(7+i\sqrt{15})/8$ rotates through $2\arcsin(1/4)$.
The sets $T$ and $-\delta^2T$ meet only at the origin, so their union has $793$ vertices.
Its six edges between the copies join $2\beta^2\alpha^j$ to $-2\delta^2\beta^2\alpha^j$, for $0\leq j\leq5$.
Their differences are monomials because
\[
(-2\delta^2\beta^2\alpha^j)-(2\beta^2\alpha^j)
=\delta\beta^2\alpha^j.
\]
Following de Grey, remove $p_2$ and $-p_2$ to obtain $B=(T\cup(-\delta^2T))\setminus\{p_2,-p_2\}$, with $791$ vertices and $3938$ edges.
Then translate to $U=B+2\beta^2$.
This places the former vertex $-2\beta^2$ at the origin and the former vertex $2\beta^2$ at $4\beta^2$.

The final rotation joins two copies of $U$ with a single additional unit edge.
Set $\epsilon=(-1+3i\sqrt7)/8$, which satisfies $4\epsilon^2+\epsilon+4=0$, and define
\[
V_{\mathrm{dG}}=\epsilon^2U\cup(-U).
\]
The copies have relative rotation $-\epsilon^2=(31+3i\sqrt7)/32$, through $2\arcsin(1/8)$, and meet only at the origin.
The linking vertices are $4\beta^2\epsilon^2$ and $-4\beta^2$, with difference
\[
(-4\beta^2)-(4\beta^2\epsilon^2)=\beta^2\epsilon.
\]
Thus the final graph $G_{\mathrm{dG}}$, with all unit-distance pairs as edges, has $1581$ vertices and $7877$ edges.
It is de Grey's graph after the rigid motion $z\mapsto\beta^2\epsilon(z+2)$ applied to the coordinates in~\cite[Section~5.1]{deGrey2018}.

Writing $D(G)$ for the edge differences in both orientations, the three kinds of edges give
\begin{equation}
\label{eq:de-grey-direction-certificate}
\begin{aligned}
D(G_{\mathrm{dG}})\subseteq{}&
\bigcup_{\ell,r\in\{0,2\}}\delta^\ell\epsilon^r\mathcal D_0\\
&\cup\{\alpha^j\beta^2\delta\epsilon^r:
0\leq j\leq5,\ r\in\{0,2\}\}\\
&\cup\{\pm\beta^2\epsilon\}.
\end{aligned}
\end{equation}
The first family comes from edges within copies of $T$, the second from the six linking edges in each copy of $B$, and the last from the final linking edge.
Exact verification of the counts and the absence of further unit edges is described in Subsection~\ref{subsec:embedding-verification}.

\begin{proposition}
    \label{prop:de-grey-embedding}
    The graph $G_{\mathrm{dG}}$ is a subgraph of $\MonCay(\alpha,\beta,\delta,\epsilon)$.
    Consequently, $\chi(\MonCay(\alpha,\beta,\delta,\epsilon))\geq5$.
\end{proposition}

\begin{proof}
    The seed table and the construction place every vertex in $\Z[\alpha,\beta,\delta,\epsilon]$.
    Every difference in~\eqref{eq:de-grey-direction-certificate} is a signed monomial with nonnegative exponents; both signs in the first two families are included because $\alpha^3=-1$.
    Thus every edge is an allowed step.
    The lower bound follows from de Grey's five-chromaticity result~\cite{deGrey2018}.
\end{proof}

\subsection{The Parts Graph}
\label{subsec:parts-embedding}

Parts's five-chromatic graph $G_{\mathrm P}$ has $509$ vertices and $2442$ edges~\cite[Section~6]{Parts2020}.
It combines a $374$-vertex piece $L_{\mathrm P}$ with a $136$-vertex piece $S_{\mathrm P}$, rotated by the multiplier $-\delta^2$ used above.
The pieces share only the origin, giving $V(G_{\mathrm P})=L_{\mathrm P}\cup(-\delta^2S_{\mathrm P})$.
We use the exact coordinate list specified in~\cite{PartsCoordinates}.

The new connecting directions are captured by $\kappa=(\sqrt5+i\sqrt{11})/4$.
With this fourth parameter, all signed edge directions have the form
\begin{equation}
\label{eq:parts-direction-certificate}
D(G_{\mathrm P})=\{\alpha^j r:0\leq j\leq5,\ r\in\mathcal R_{\mathrm P}\},
\end{equation}
where
\[
\begin{aligned}
\mathcal R_{\mathrm P}={}&\{\beta^k:-2\leq k\leq2\}
\cup\{\beta^k\delta^2:-1\leq k\leq1\}\\
&\cup\{\beta^{-1}\delta,\beta\delta\}
\cup\{\delta\kappa^{-1},\delta\kappa\}.
\end{aligned}
\]
These twelve representatives give $72$ signed directions.
The negative exponents can all be cleared by the same rotation: multiplying by $\beta^2\kappa$ makes every direction a signed monomial with nonnegative exponents.

\begin{proposition}
    \label{prop:parts-embedding}
    The rotation $z\mapsto\beta^2\kappa z$ gives a graph embedding
    \[
    G_{\mathrm P}\hookrightarrow\MonCay(\alpha,\beta,\delta,\kappa).
    \]
    Consequently, $\chi(\MonCay(\alpha,\beta,\delta,\kappa))\geq5$.
\end{proposition}

\begin{proof}
    The direction certificate shows that every rotated edge is an allowed step.
    Since $G_{\mathrm P}$ is connected and contains the origin, a path from the origin expresses each rotated vertex as a sum of these steps.
    Hence all rotated vertices belong to $\Z[\alpha,\beta,\delta,\kappa]$.
    The rotation is injective, and Parts's five-chromaticity result gives the lower bound~\cite{Parts2020}.
\end{proof}

Both embeddings use parameters with rational squared cosine.
The values for $\alpha$, $\beta$, $\delta$, $\epsilon$, and $\kappa$ are $1/4$, $11/12$, $1/16$, $1/64$, and $5/16$, respectively.
Each parameter alone gives a three-colorable graph by Corollary~\ref{thm:degree-six-bound}, while either four-parameter collection requires at least five colors.

\subsection{Exact Verification}
\label{subsec:embedding-verification}

The accompanying verification programs check the embeddings using rational arithmetic in the radicals $\sqrt3$, $\sqrt5$, $\sqrt{11}$, and $\sqrt7$.
They compare coordinate pairs by the identity $(v-u)\overline{(v-u)}=1$.
Pairs may first be excluded by rational interval bounds; every surviving pair is checked exactly.
Thus the edge lists include all unit-distance pairs, without a numerical tolerance.

The folder contains complete exact coordinate lists for both graphs, including an exact conversion of the Parts data~\cite{PartsCoordinates}.
The single embedding verifier applies the stated rigid motions and checks $1581$ distinct vertices and $7877$ unit edges for de Grey's graph, and $509$ distinct vertices and $2442$ unit edges for Parts's graph.
For each graph it checks that every edge difference is a signed monomial with nonnegative exponents in the respective four parameters.
It also checks connectedness and the vertex at zero, so a path from zero expresses every vertex as a sum of these monomials and certifies membership in the required ring.
The file \texttt{anc/README.md} gives the commands, data provenance, and expected outputs.
These computations verify the embeddings; the five-chromaticity results used in the propositions are those of the original papers.